\pdfoutput=1
\documentclass[11pt, a4paper]{article}
\usepackage[top=1.35in,bottom=1.5545in,left=1.5545in, right=1.5545in]{geometry}
\usepackage[svgnames]{xcolor}
\usepackage{mathtools}
\usepackage[abbrev,lite,nobysame]{amsrefs}
\usepackage{amsthm,amsfonts,amsmath}
\usepackage{hyphenat}
\usepackage{csquotes}
\usepackage{enumerate}
\usepackage[font=footnotesize,labelfont=bf]{caption}
\usepackage{subcaption}
\usepackage[colorlinks=true,linkcolor=NavyBlue,
citecolor=NavyBlue,urlcolor=NavyBlue]{hyperref}
\usepackage{cleveref}
\crefname{equation}{}{}

\usepackage[looser]{newtxtext}
\usepackage[bigdelims,varvw]{newtxmath}
\usepackage[protrusion=true, tracking=true, kerning=true,spacing=true]{microtype}
\usepackage{bm}

\usepackage{tikz}
\usepgflibrary{patterns}
\usetikzlibrary{patterns.meta,decorations.pathmorphing
	,decorations.text,positioning,shapes.geometric,arrows.meta,decorations.markings,fadings,decorations.pathreplacing,calligraphy} 
\usepackage{pgfplots}
\pgfplotsset{compat = newest}
\usepgfplotslibrary{colormaps}
\tikzfading[name=fade right,
left color=transparent!0, right color=transparent!100]
\tikzfading[name=fade left,
right color=transparent!0, left color=transparent!100]
\tikzset{>={Latex[width=3mm,length=3mm]}}

\tikzstyle{blue0} = [rectangle, rounded corners, minimum width=2.2cm, minimum height=1.3cm,text centered, draw=black, fill=MidnightBlue!3, text width=4.2cm]
\tikzstyle{blue0b} = [rectangle, rounded corners, minimum width=2.2cm, minimum height=1.2cm,text centered, draw=black, fill=MidnightBlue!3, text width=7.8cm]
\tikzstyle{blue0bb} = [rectangle, rounded corners, minimum width=2.2cm, minimum height=1.6cm,text centered, draw=black, fill=MidnightBlue!3, text width=12cm]
\tikzstyle{blue0d} = [rectangle, rounded corners, minimum width=2.2cm, minimum height=1.2cm,text centered, draw=black, fill=MidnightBlue!3, text width=2cm]
\tikzstyle{blue00} = [rectangle, rounded corners, minimum width=2.2cm, minimum height=1.8cm,text centered, draw=black, fill=MidnightBlue!3, text width=7cm]
\tikzstyle{blue00c} = [rectangle, rounded corners, minimum width=2.2cm, minimum height=1.2cm,text centered, draw=black, fill=MidnightBlue!3, text width=7cm]
\tikzstyle{blue000} = [rectangle, rounded corners, minimum width=2.2cm, minimum height=1.8cm,text centered, draw=black, fill=MidnightBlue!3, text width=16.5cm]
\tikzstyle{blue} = [rectangle, text centered, rounded corners, minimum width=0.6cm, minimum height=0.4cm, draw=white, fill=white, text width=10cm]
\tikzstyle{bluea} = [rectangle, rounded corners, minimum width=2cm, minimum height=1.2cm,text centered, draw=black, fill=MidnightBlue!3, text width=2.5cm]
\tikzstyle{blueb} = [rectangle, rounded corners, minimum width=3.5cm, minimum height=1.2cm,text centered, draw=black, fill=MidnightBlue!3, text width=4cm]
\tikzstyle{bluec} = [rectangle, rounded corners, minimum width=5.6cm, minimum height=1.2cm,text centered, draw=black, fill=MidnightBlue!3, text width=5.8cm]
\tikzstyle{arrow} = [thick,->]

\tikzdeclarepattern{
	name=mylines,
	parameters={
		\pgfkeysvalueof{/pgf/pattern keys/size},
		\pgfkeysvalueof{/pgf/pattern keys/angle},
		\pgfkeysvalueof{/pgf/pattern keys/line width},
	},
	bounding box={
		(0,-0.5*\pgfkeysvalueof{/pgf/pattern keys/line width}) and
		(\pgfkeysvalueof{/pgf/pattern keys/size},
		0.5*\pgfkeysvalueof{/pgf/pattern keys/line width})},
	tile size={(\pgfkeysvalueof{/pgf/pattern keys/size},
		\pgfkeysvalueof{/pgf/pattern keys/size})},
	tile transformation={rotate=\pgfkeysvalueof{/pgf/pattern keys/angle}},
	defaults={
		size/.initial=5pt,
		angle/.initial=45,
		line width/.initial=.4pt,
	},
	code={
		\draw [line width=\pgfkeysvalueof{/pgf/pattern keys/line width}]
		(0,0) -- (\pgfkeysvalueof{/pgf/pattern keys/size},0);
	},
}

\numberwithin{equation}{section}
\theoremstyle{plain}
\newtheorem{thrm}{Theorem}[section]

\newtheorem{lmm}[thrm]{Lemma}

\newtheorem{ssmptn}[thrm]{Assumption}

\theoremstyle{definition}

\newtheorem{xmpl}[thrm]{Example}
\newtheorem{rmrk}[thrm]{Remark}
\newtheorem{prblm}[thrm]{Problem}

\theoremstyle{plain}

\newcommand{\xvec}[1]{\bm{#1}}
\newcommand{\xsym}[1]{\bm{#1}}
\newcommand{\xdop}[1]{\bm{\mathrm{#1}}}
\def\xnab{\xdop{\nabla}}

\newcommand{\xwcurl}[1]{\xdop{\nabla}\wedge{#1}}

\newcommand{\xdiv}[1]{\xdop{\nabla}\cdot{#1}}
\newcommand{\xmcal}[1]{\bm{\mathcal{#1}}}
\newcommand{\xdx}[1]{{{\rm d}#1}}
\def\xdrv#1#2{\frac{{\rm d}#1}{{\rm d}#2}}
\def\cf{\emph{cf.\/}}\def\ie{\emph{i.e.\/}}\def\eg{\emph{e.g.\/}}

\def\@Rref#1{\hbox{\rm \ref{#1}}}
\def\Rref#1{\@Rref{#1}}
\def\xCzero{{\rm C}^{0}}
\def\xCone{{\rm C}^{1}} 
 
\def\xCinfty{{\rm C}^{\infty}} 
\def\xCn#1{{\rm C}^#1}

\def\xWn#1{{\rm W}^#1}

\def\xLtwo{{\rm L}^{2}}

\def\xLinfty{{\rm L}^{\infty}} 
\def\xLn#1{{\rm L}^#1}

\def\xU{\mathrm{U}}
\def\xX{{\rm X}}
\def\xY{{\rm Y}}
\def\xZ{{\rm Z}}

\def\xB{{\rm B}}
\def\xD{\mathrm{D}}
\def\cD{\mathcal{D}}

\def\xS{{\rm S}}

\makeatletter
\newcommand*{\toccontents}{\@starttoc{toc}}
\makeatother

\begin{document}\linespread{1.05}\selectfont
	\date{}
	
	\author{Manuel~Rissel\,\protect\footnote{NYU-ECNU Institute of Mathematical Sciences at NYU Shanghai, 3663 Zhongshan Road North, Shanghai, 200062, China, e-mail: \href{mailto:manuel.rissel@nyu.edu}{Manuel.Rissel@nyu.edu}}}
	
	\title{Exact controllability of incompressible ideal magnetohydrodynamics in $2$D}
	
	\maketitle

	\begin{abstract}
		This work examines the controllability of planar incompressible ideal magnetohydrodynamics (MHD). Interior controls are obtained for problems posed in doubly-connected regions; simply-connected configurations are driven by boundary controls.  Up to now, only straight channels regulated at opposing walls have been studied. Hence, the present program adds to the literature an exploration of interior controllability, extends the known boundary controllability results, and contributes ideas for treating general domains. To transship obstacles stemming from the MHD coupling and the magnetic field topology, a divide-and-control strategy is proposed. This leads to a family of nonlinear velocity-controlled sub-problems which are solved using J.-M. Coron's return method. The latter is here developed based on a reference trajectory in the domain's first cohomology space. 
	\end{abstract}

	

	\setcounter{tocdepth}{2}
	\toccontents 
	\normalsize
	\newgeometry{margin=1.5545in}
	
	\section{Introduction}\label{section:introduction}
	We investigate the global exact controllability of incompressible ideal magnetohydrodynamics (MHD) in $2$D. The model combines Euler's system for perfect fluids and Maxwell's description of electromagnetism (\cf~\cite{Schmidt,Secchi1993}); regulating such highly conductive media concerns applications along with theory. Motivations also originate from controllability problems raised by J.-L. Lions during the 1990s (\cf~\cite{LionsJL1991,Lions1997}), including inviscid flows, as treated in  \cite{Coron1996EulerEq,Glass2000}, and the conjectured global approximate controllability of the Navier--Stokes system (\cf~\cite{CoronMarbachSueurZhang2019,CoronMarbachSueur2020}). Compared with incompressible Euler, incompressible ideal MHD introduces additional difficulties such as crossing characteristics and structural requirements on the controls, and the present article continues the works \cite{RisselWang2021, KukavicaNovackVicol2022} which have studied these issues exclusively in straight channel domains.

	The present objectives are outlined using interior controls. Let~$\cD\subset\mathbb{R}^2$ be a container with perfectly conducting impermeable walls and outward unit normal~$\xvec{n}_{\cD}$. Moreover, $\mathbb{I}_{\omegaup}$ denotes the indicator function of a nonempty relatively open set~$\omegaup \subset \overline{\cD}$. A plasma spreading over~$\cD$, having velocity~$\xvec{u}$, interacting with the magnetic field~$\xvec{B}$, and exerting a total pressure~$p$, is assumed to obey the incompressible ideal MHD system
	\begin{equation}\label{equation:1}
		\begin{gathered}
				\partial_t \xvec{u} + (\xvec{u} \cdot \xsym{\nabla}) \xvec{u} - (\xvec{B} \cdot 	\xsym{\nabla})\xvec{B} + \xsym{\nabla} p = \mathbb{I}_{\omegaup}\xsym{\xi},\\
				\partial_t \xvec{B} + (\xvec{u} \cdot \xsym{\nabla}) \xvec{B} - (\xvec{B} \cdot 	\xsym{\nabla}) \xvec{u} = \mathbb{I}_{\omegaup}\xsym{\eta},\\
				\xdiv{\xvec{u}} = \xdiv{\xvec{B}} = 0,\\
			\xvec{u}_{|_{\partial\cD}} \cdot \xvec{n}_{\cD} = 0, \quad  \xvec{B}_{|_{\partial\cD}} \cdot \xvec{n}_{\cD} = 0,
		\end{gathered}
	\end{equation}
	where $\xsym{\xi}$ and $\xsym{\eta}$ are the controls we wish to exploit to steer \eqref{equation:1}'s solutions towards prescribed target states. In other words, given $T > 0$ and $\xvec{u}_0$, $\xvec{u}_T$, $\xvec{B}_0$, $\xvec{B}_T$, under which assumptions are there controls~$\xsym{\xi}$ and $\xsym{\eta}$ so that the solution to \eqref{equation:1} transitions from the initial state $(\xvec{u}_0, \xvec{B}_0)$ to the target state $(\xvec{u}_T, \xvec{B}_T)$ in time~$T$?

	So far, the controllability of ideal MHD is only known for straight $2$D channels forced through the boundary conditions at two opposed walls (\cf~\cite{RisselWang2021,KukavicaNovackVicol2022}). This sparsity of available results parallels several difficulties linked to the nonlinear coupling effects in ideal MHD, \eg, regularity loss issues and the consequential reliance on symmetrizations that are not directly amenable to techniques designed for pure fluids. Besides, as seen by taking in the second line of \eqref{equation:1} the divergence and normal trace, the localized magnetic field control has to satisfy
	\begin{equation}\label{equation:intrprctrl}
		\xdiv{(\mathbb{I}_{\omegaup}\xsym{\eta})} = 0, \quad (\mathbb{I}_{\omegaup}\xsym{\eta})_{|_{\partial\cD}} \cdot \xvec{n}_{\cD} = 0.
	\end{equation}

	The current approach features a divide-and-conquer paradigm that produces nonlinear sub-problems solved with the velocity being close to a curl-free flushing profile. This agenda involves trajectory splitting arguments  and a MHD adaptation of J.-M. Coron's return method (\cf~\cite[Part 2, Chapter 6]{Coron2007}). Ideas come also from \cite{RisselWang2021,KukavicaNovackVicol2022,CoronMarbachSueur2020,RisselWang2022}.

	Despite that the global regularity properties of solutions to ideal MHD remain in general unknown (\cf~\cite{Wu2018}), here, like in \cite{RisselWang2021,KukavicaNovackVicol2022}, any risk of blow-up is ruled out by the controls. Such observations are not uncommon; see \cite{Glass2000,Nersisyan2011} regarding three-dimensional Euler problems.

	\paragraph{Notations.} 	The $\xLtwo(\cD;\mathbb{R}^n)$-inner product with $n \in \mathbb{N}$ is abbreviated by $\langle\cdot,\cdot\rangle_{\xLtwo(\cD;\mathbb{R}^n)}$. Given $\xS \subset \partial \cD$ and $\alpha \in (0,1)$, the divergence-free vector fields tangential at $\partial\cD\setminus\xS$ having H\"older continuous derivatives up to the order $l \in \mathbb{N}$ are collected in
	\[
		\xCn{{l,\alpha}}_{*}(\overline{\cD}, \xS; \mathbb{R}^2) \coloneq \left\{ \xvec{f} \in \xCn{{l,\alpha}}(\overline{\cD};\mathbb{R}^2) \, \left| \right. \, \xdiv{\xvec{f}} = 0 \mbox{ in } \cD, \, \xvec{f} \cdot \xvec{n}_{\cD} = 0 \mbox{ on } \partial \cD \setminus \xS \right\},
	\]
	where $\xCn{{l,\alpha}}(\overline{\cD};\mathbb{R}^n)$ is equipped with the norm 
	\[
		\|\xvec{f}\|_{l, \alpha,\cD} \coloneq \sum\limits_{0\leq|\xsym{\xsym{\beta}}|\leq l} \sup\limits_{\xvec{a} \in \cD}|\partial^{\xsym{\beta}}\xvec{f}(\xvec{a})| + \sum\limits_{|\xsym{\beta}| = l} \sup\limits_{\substack{\xvec{a}, \xvec{b} \in \cD \\ \xvec{a}\neq \xvec{b}}}\frac{|\partial^{\xsym{\beta}}\xvec{f}(\xvec{a})-\partial^{\xsym{\beta}}\xvec{f}(\xvec{b})|}{|\xvec{a}-\xvec{b}|^{\alpha}}.
	\]
	In this article, the set $\cD$ is either simply-connected and $\xS \neq \emptyset$, or $\cD$ is doubly-connected and $\xS = \emptyset$. Especially, any $\xvec{f} = [f_1, f_2] \in \xCn{{l,\alpha}}_{*}(\overline{\cD},\xS; \mathbb{R}^2)$ admits a stream function representation 
	\[
		\xvec{f} = \xnab^{\perp} \phi = \begin{bmatrix}
			\partial_2 \phi \\ -\partial_1 \phi
		\end{bmatrix}.
	\]
	Owing to the Gau\ss-Green theorem, such a $\phi$ can be unambiguously defined by a path integral (\cf~\cite{MarchioroPulvirenti1994})
	\[
		 \phi(\xvec{x}) \coloneq - \int_{C_{\xvec{o}\xvec{x}}} \xvec{f}^{\perp} \cdot \xdx{l} + \operatorname{constant},
	\]
	where $\xvec{f}^{\perp} \coloneq [f_2, -f_1]$ and $C_{\xvec{o}\xvec{x}}$ is a curve connecting the reference point $\xvec{o}$ with $\xvec{x}$. Moreover, the potential $\phi$ solves Poisson's equation
	\[
		-\Delta \phi = \xwcurl{\xvec{f}} \coloneq \partial_1 f_2 - \partial_2 f_1, \quad \phi_{|_{\partial\cD\setminus\xS}} = \mbox{ piecewise constant.}
	\]

	\subsection{Interior controllability}\label{subsection:introinterior}
	Let $\mathcal{E} \subset \mathbb{R}^2$ represent the region enclosed by two nested non-intersecting smooth Jordan curves $\Gamma^0$ and $\Gamma^1$, writing $\xvec{n}_{\mathcal{E}}$ for the outward unit normal at $\partial \mathcal{E}$ (\cf~\Cref{Figure:annulusexmpl}). The control zone~$\omegaup \subset \overline{\mathcal{E}}$ is relatively open and some $\Lambda \subset \omegaup$ must render $\overline{\mathcal{E}}\setminus \Lambda$ simply-connected.
	
	It is observed in \Cref{subsection:setup} that solutions to MHD problems like \eqref{equation:1}, driven by our controls, satisfy ${\rm d}/{\rm d}t \langle\xvec{B}(\cdot,t), \xvec{Q}\rangle_{\xLtwo(\mathcal{E};\mathbb{R}^2)} = 0$
	for all $\xvec{Q} \in \xLtwo(\mathcal{E}; \mathbb{R}^2)$ obeying 
	\begin{equation}\label{equation:divcurl0}
		\xwcurl{\xvec{Q}} = 0 \mbox{ in } \mathcal{E}, \quad \xdiv{\xvec{Q}} = 0 \mbox{ in } \mathcal{E}, \quad \xvec{Q} \cdot \xsym{n}_{\mathcal{E}} = 0 \mbox{ on } \partial \mathcal{E}.
	\end{equation}
	Here, similar to \cite{KukavicaNovackVicol2022}, only magnetic fields with vanishing first cohomology projection are considered:
	\begin{equation}\label{equation:conditionq}
		\langle \xvec{B}_0, \xvec{Q}\rangle_{\xLtwo(\mathcal{E};\mathbb{R}^2)} = \langle \xvec{B}_T, \xvec{Q}\rangle_{\xLtwo(\mathcal{E};\mathbb{R}^2)} = 0 \, \mbox{ for all } \xvec{Q} \in \xLtwo(\mathcal{E}; \mathbb{R}^2) \mbox { with \eqref{equation:divcurl0}}.
	\end{equation}
	This is a mild restriction, recalling that the solutions to \eqref{equation:divcurl0} span an one-dimensional subspace (\cf~\cite[Chapter IX]{DautrayLions1990}).
	
	\begin{xmpl}
		Let $0 < r_1 < r_2$ and consider the annulus $\mathcal{E} \coloneq \{ r_1 < |\xvec{x} | < r_2 \}$. If one takes $\xvec{g} \coloneq \xnab^{\perp} \ln |\cdot| \in\xLtwo(\mathcal{E};\mathbb{R}^2)$, then $\xvec{Q} \coloneq \xvec{g}$ solves \eqref{equation:divcurl0} and it holds
		\[
			\langle\xvec{g}, \xvec{Q}\rangle_{\xLtwo(\mathcal{E};\mathbb{R}^2)} = \langle\xvec{g}, \xvec{g}\rangle_{\xLtwo(\mathcal{E};\mathbb{R}^2)} \neq 0.
		\]
	\end{xmpl}
	
	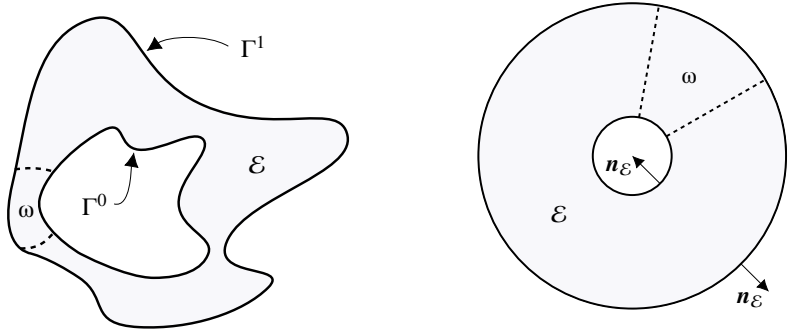
\begin{figure}[ht!]
		\centering
		\resizebox{0.83\textwidth}{!}{
			\begin{subfigure}[b]{0.5\textwidth}
				\centering
				\resizebox{1\textwidth}{!}{
					\begin{tikzpicture}
						\clip(0.4,-1.1) rectangle (7.5,5.3);

						\draw [line width=0.5mm, color=black, fill=MidnightBlue!3] plot[smooth, tension=1] coordinates {(1,0.5) (0.65,1.9) (2.1,5.2) (4.5,3.5) (7.4,2.88) (5,0.8) (5.8,-0.2) (3.2,-1) (2,0) (1,0.5)};
						
						\draw [dashed, line width=0.5mm, color=black] plot[smooth, tension=1] coordinates {(0.81,0.6) (1.25,0.7) (1.6,1.1)};
						\draw [dashed, line width=0.5mm, color=black] plot[smooth, tension=1] coordinates {(0.7,2.2) (1.3,2.2) (1.85,2)};

						\draw [line width=0.5mm, color=black, fill=white]  plot[smooth, tension=0.8] coordinates {(2,0.5) (1.2,1.6) (2.5,3) (3.2,2.6) (4.5,2.8) (4,1.5) (4.6,0.5) (3.5,0)  (2,0.5)};
						
						\coordinate[label=left:\Large{$\mathcal{E}$}] (A) at (5.9,2.3);
						\coordinate[label=left:\Large{$\omegaup$}] (B) at (1.25,1.35);
						
						\coordinate[label=left:\Large{$\Gamma^0$}] (C) at (2.7,1.5);
						\coordinate[label=left:\Large{$\Gamma^1$}] (D) at (5.9,4.7);
						
						\draw[line width=0.2mm, ->] (2.7,1.5)  to [out=-15,in=-100,looseness=0.8] (3.1, 2.55);
						\draw[line width=0.2mm, ->] (5,4.7)  to [out=150,in=35,looseness=0.8] (3.35, 4.6);
					\end{tikzpicture}
				}
				\label{Figure:e1}
			\end{subfigure}
			\quad \quad \quad \quad 
			\begin{subfigure}[b]{0.5\textwidth}
				\centering
				\resizebox{1\textwidth}{!}{
					\begin{tikzpicture}
						\clip(-4.5,-4.5) rectangle (4.5,4.5);
						\draw[line width=0.1mm, color=black] (80:1) coordinate (A) arc (80:30:1) coordinate (C)  (30:3.9) coordinate (D) arc (30:80:3.9) coordinate (B);
						\draw[line width=0.1mm, color=black] (80:1) arc (80:395:1) (2:3.9) arc (2:395:3.9);
						
						\draw [line width=0.5mm, color=black, fill=MidnightBlue!3] (0,0) circle (3.9);
						\draw [line width=0.5mm, color=black, fill=white] (0,0) circle (1);
						
						\draw[dashed,line width=0.5mm, color=black] (A) -- (B);
						\draw[dashed,line width=0.5mm, color=black] (C) -- (D);
						
						\draw[->,line width=0.2mm, color=black] (3.9*0.707,-3.9*0.707) -- (4.9*0.707,-4.9*0.707);
						\draw[->,line width=0.2mm, color=black] (0.707,-0.707) -- (0,0);
						
						\coordinate[label=left:\LARGE{$\mathcal{E}$}] (A) at (-1.5,-1.5);
						\coordinate[label=left:\LARGE{$\omegaup$}] (B) at (1.8,2);
						\coordinate[label=left:\LARGE{$\xvec{n}_{\mathcal{E}}$}] (C) at (4.9*0.707,-5.2*0.707);
						\coordinate[label=left:\LARGE{$\xvec{n}_{\mathcal{E}}$}] (D) at (0.15,-0.3);
						
					\end{tikzpicture}
				}
				\label{Figure:e2}
		\end{subfigure}}
		\caption{Two examples of a doubly-connected domain $\mathcal{E}$ and an admissible control region $\omegaup \subset \overline{\mathcal{E}}$.}
		\label{Figure:annulusexmpl}
	\end{figure}

	As stated next, we achieve the global (large data) exact interior controllability of incompressible ideal MHD. Throughout, the parameters $m \geq 2$ and $\alpha \in (0,1)$ are fixed.
	\begin{thrm}[Main result I]\label{theorem:maininterior}
		Given $T > 0$ and $\xvec{u}_0, \xvec{B}_0, \xvec{u}_T, \xvec{B}_T \in \xCn{{m,\alpha}}_{*}(\overline{\mathcal{E}}, \emptyset; \mathbb{R}^2)$ with~\eqref{equation:conditionq}, there exist controls
		\begin{gather}
			\xsym{\xi}, \xsym{\eta} \in \xLinfty((0,T);\xCn{{m-1,\alpha}}(\overline{\mathcal{E}}; \mathbb{R}^2))\label{equation:controlsregularity}
		\end{gather}
		such that the incompressible ideal MHD problem
		\begin{equation}\label{equation:MHD00interior}
			\begin{cases}
				\partial_t \xvec{u} + (\xvec{u} \cdot \xsym{\nabla}) \xvec{u} - (\xvec{B} \cdot \xsym{\nabla})\xvec{B} + \xsym{\nabla} p = \mathbb{I}_{\omegaup}\xsym{\xi} & \mbox{ in } \mathcal{E} \times (0,T),\\
				\partial_t \xvec{B} + (\xvec{u} \cdot \xsym{\nabla}) \xvec{B} - (\xvec{B} \cdot \xsym{\nabla}) \xvec{u} = \mathbb{I}_{\omegaup}\xsym{\eta} & \mbox{ in } \mathcal{E} \times (0,T),\\
				\xdiv{\xvec{u}} = \xdiv{\xvec{B}} = 0  & \mbox{ in } \mathcal{E} \times (0,T),\\
				\xvec{u} \cdot \xvec{n}_{\mathcal{E}} = \xvec{B} \cdot \xvec{n}_{\mathcal{E}} = 0  & \mbox{ on } \partial \mathcal{E} \times(0,T),\\
				\xvec{u}(\cdot, 0)  =  \xvec{u}_0,\, \xvec{B}(\cdot, 0)  =  \xvec{B}_0  & \mbox{ in } \mathcal{E}
			\end{cases}
		\end{equation}
		admits a unique solution
		\begin{equation*}
			\begin{gathered}
				\xvec{u}, \xvec{B} \in  \xCn{0}([0,T];\xCn{{m-1,\alpha}}(\overline{\mathcal{E}}; \mathbb{R}^2)) \cap \xLinfty((0,T);\xCn{{m,\alpha}}(\overline{\mathcal{E}}; \mathbb{R}^2)), \\
				p \in \xCzero([0,T];\xCn{{m-1,\alpha}}(\overline{\mathcal{E}};\mathbb{R})) \cap \xLinfty((0,T);\xCn{{m,\alpha}}(\overline{\mathcal{E}};\mathbb{R}))
			\end{gathered}
		\end{equation*}
		satisfying the target constraints
		\begin{equation*}\label{equation:MHD-Endconditionint}
			\begin{aligned}
			\xvec{u}(\cdot, T)  =  \xvec{u}_T, \quad \xvec{B}(\cdot, T)  =  \xvec{B}_T.
			\end{aligned}
		\end{equation*}
	\end{thrm}

	\begin{rmrk}\label{remark:uniqueness}
		After $\xsym{\xi}$ and $\xsym{\eta}$ are found (and fixed), the uniqueness of solutions to~\eqref{equation:MHD00interior} follows via energy estimates (\cf~\cite[Remark 1.4]{RisselWang2021}). Existence, uniqueness, and continuous dependence results for \eqref{equation:MHD00interior} are also provided by \cite{Schmidt,Secchi1993}.
	\end{rmrk}

	\begin{prblm}[See also \cite{KukavicaNovackVicol2022}]\label{problem:i}
		Does \Cref{theorem:maininterior} hold when $\langle \xvec{B}_0 - \xvec{B}_T, \xvec{Q}\rangle_{\xLtwo(\mathcal{E};\mathbb{R}^2)} = 0$ but $\langle \xvec{B}_0, \xvec{Q}\rangle_{\xLtwo(\mathcal{E};\mathbb{R}^2)} \neq 0$ for all $\xvec{Q}\neq\xsym{0}$ with \eqref{equation:divcurl0}?
	\end{prblm}

	\begin{prblm}\label{problem:p1}
		What are the optimal assumptions on $\omegaup \subset \mathcal{E}$? 
	\end{prblm}

	\begin{rmrk}\label{remark:rr}
		For several reasons, the current approach relies on the geometric constraints imposed on $\omegaup$.
		\begin{itemize}
			\item The integral curves of the return method profile (as selected in \Cref{subsection:flushingprofile}) form a homotopy between $\Gamma^0$ and $\Gamma^1$.
			\item The velocity control acts on the first cohomology part of the momentum equation (\cf~\Cref{subsubsection:sfp}).
			\item In view of Kelvin's law, it is necessary that $\omegaup \cap \Gamma^0 \neq \emptyset$ and $\omegaup \cap \Gamma^1 \neq \emptyset$ (\cf~\cite{Coron1996EulerEq}).
		\end{itemize}
	\end{rmrk}

	\subsection{Boundary controllability}\label{subsection:boundarycontrol}
	Let $\Omega \subset \mathbb{R}^2$ be the simply-connected region enclosed by a Jordan curve $\Gamma$ with outward unit normal $\xsym{n}_{\Omega} = [n_1, n_2]$. 
	The controls act on a portion $\Gamma_{\operatorname{c}}\subset\Gamma$ that has nonempty interior. 
	
	\subsubsection{Result}
	We rely on an extension property valid for reasonable classes of data and domains (\cf~\Cref{subsubsection:examples}). As in \Cref{subsection:introinterior}, the parameters $m \geq 2$ and $\alpha \in (0,1)$ are fixed.

	\begin{ssmptn}[Extension property]\label{assumption:extensionproperty}
		There exist a doubly-connected smoothly bounded domain $\mathcal{E} \subset \mathbb{R}^2$ and states $\widetilde{\xvec{u}}_0, \widetilde{\xvec{B}}_0, \widetilde{\xvec{u}}_T, \widetilde{\xvec{B}}_T \in \xCn{{m,\alpha}}_{*}(\overline{\mathcal{E}},\emptyset; \mathbb{R}^2)$
		satisfying 
		\begin{gather*}
			\Omega \subset \mathcal{E}, \quad \Gamma_{\operatorname{c}} \cap \mathcal{E} \neq \emptyset, \quad \Gamma\setminus\Gamma_{\operatorname{c}} \subset \partial \mathcal{E}, \\
			\widetilde{\xvec{u}}_0 |_{\Omega} = \xvec{u}_0, \quad \widetilde{\xvec{u}}_T |_{\Omega} = \xvec{u}_T, \quad \widetilde{\xvec{B}}_0 |_{\Omega} = \xvec{B}_0, \quad \widetilde{\xvec{B}}_T |_{\Omega} = \xvec{B}_T,\\
			\langle  \widetilde{\xvec{B}}_0, \xvec{Q}\rangle_{\xLtwo(\mathcal{E};\mathbb{R}^2)} = \langle  \widetilde{\xvec{B}}_T, \xvec{Q}\rangle_{\xLtwo(\mathcal{E};\mathbb{R}^2)} = 0 \, \mbox{ for all } \xvec{Q} \mbox { with \eqref{equation:divcurl0}}.
		\end{gather*}
	\end{ssmptn}

	In this framework, we set out to study the boundary-controlled incompressible ideal MHD system
	\begin{equation}\label{equation:MHD00}
		\begin{cases}
			\partial_t \xvec{u} + (\xvec{u} \cdot \xsym{\nabla}) \xvec{u} - (\xvec{B} \cdot \xsym{\nabla})\xvec{B} + \xsym{\nabla} p = \xsym{0} & \mbox{ in } \Omega \times (0,T),\\
			\partial_t \xvec{B} + (\xvec{u} \cdot \xsym{\nabla}) \xvec{B} - (\xvec{B} \cdot \xsym{\nabla}) \xvec{u} = \xsym{0} & \mbox{ in } \Omega \times (0,T),\\
			\xdiv{\xvec{u}} = \xdiv{\xvec{B}} = 0  & \mbox{ in } \Omega \times (0,T),\\
			\xvec{u} \cdot \xsym{n}_{\Omega} = \xvec{B} \cdot \xsym{n}_{\Omega} = 0  & \mbox{ on } (\Gamma \setminus \Gamma_{\operatorname{c}}) \times(0,T),\\
			\xvec{u}(\cdot, 0)  =  \xvec{u}_0,\, \xvec{B}(\cdot, 0)  =  \xvec{B}_0  & \mbox{ in } \Omega.
		\end{cases}
	\end{equation}
	As the boundary conditions in \eqref{equation:MHD00} are only specified at~$\Gamma\setminus\Gamma_{\operatorname{c}}$, many solutions might be available.
	The system \eqref{equation:MHD00} is called globally exactly boundary controllable if, for any fixed time~$T > 0$ and possibly large admissible states $\xvec{u}_0$, $\xvec{B}_0$, $\xvec{u}_T$, and $\xvec{B}_T$, there exists at least one solution~$(\xvec{u},\xvec{B}, p)$ meeting the target constraints
	\begin{equation*}\label{equation:terminal_condition}
		\xvec{u}(\cdot, T) = \xvec{u}_T, \quad \xvec{B}(\cdot, T) = \xvec{B}_T.
	\end{equation*}
	The controls acting in \eqref{equation:MHD00} are implicit, at first. After a desired controlled solution to~\eqref{equation:MHD00} is found, one may select boundary controls explicitly in terms of traces at~$\Gamma_{\operatorname{c}}$ of that solution or its derivatives (\cf~\cite{RisselWang2021,Coron1996EulerEq,Glass2000}).

	\begin{thrm}[Main result II]\label{theorem:main}
		Given $T > 0$ and $\xvec{u}_0, \xvec{B}_0, \xvec{u}_T, \xvec{B}_T \in \xCn{{m,\alpha}}_{*}(\overline{\Omega}, \Gamma_{\operatorname{c}}; \mathbb{R}^2)$ obeying \Cref{assumption:extensionproperty}, the system \eqref{equation:MHD00} admits a solution
		\begin{equation*}
			\begin{gathered}
				\xvec{u}, \xvec{B} \in \xCn{0}([0,T];\xCn{{m-1,\alpha}}(\overline{\Omega}; \mathbb{R}^2)) \cap \xLinfty((0,T);\xCn{{m,\alpha}}(\overline{\Omega}; \mathbb{R}^2)),\\
				p \in \xCzero([0,T];\xCn{{m-1,\alpha}}(\overline{\Omega};\mathbb{R})) \cap \xLinfty((0,T);\xCn{{m,\alpha}}(\overline{\Omega};\mathbb{R}))
			\end{gathered}
		\end{equation*}
		satisfying in $\Omega$ the target conditions
		\begin{equation*}\label{equation:MHD-Endconditionbc}
			\begin{aligned}
			\xvec{u}(\cdot, T)  =  \xvec{u}_T, \quad \xvec{B}(\cdot, T)  =  \xvec{B}_T.
			\end{aligned}
		\end{equation*}
	\end{thrm}
	\begin{proof}
		Owing to \Cref{assumption:extensionproperty}, we apply \Cref{theorem:local} to an extended problem posed in a doubly-connected domain $\mathcal{E}$ of the type introduced in \Cref{subsection:introinterior}. Eventually,  by taking restrictions, a boundary-controlled solution in $\Omega\times(0,T)$ is extracted from an interior-controlled one in $\mathcal{E}\times(0,T)$.
	\end{proof}

	\subsubsection{Examples}\label{subsubsection:examples}

	To illustrate that \Cref{assumption:extensionproperty} is natural, we~content ourselves with an exemplary class of domains where $\Gamma_{\operatorname{c}}$ comprises straight lines (\cf~\Cref{Figure:exampledomains}). Then, all \enquote{reasonable} states are admissible. One could compose $\Gamma_{\operatorname{c}}$ also of other curves; we avoid dwelling on it. 
	\begin{ssmptn}\label{assumption:funnel}
		$\Gamma_{\operatorname{c}} = \mathcal{L}_1 \cup \dots \cup \mathcal{L}_{i_{\operatorname{c}}}$ with line segments $\mathcal{L}_1, \dots, \mathcal{L}_{i_{\operatorname{c}}}$ meeting $\Gamma\setminus\Gamma_{\operatorname{c}}$ in a right angle, and there are pairwise disjoint open sets $\xU_{1}, \dots, \xU_{i_{\operatorname{c}}} \subset \mathbb{R}^2$ such that (\cf~\Cref{Figure:exampledomains})
		\begin{itemize}
			\item $\overline{\mathcal{L}_l} \subset \xU_l$ for each $l \in \{1,\dots,i_{\operatorname{c}}\}$,
			\item for any $l \in \{1,\dots,i_{\operatorname{c}}\}$, the intersection $\Omega\cap \xU_{l}$ is either a rectangle or an annulus sector.
		\end{itemize}
	\end{ssmptn}

	\begin{figure}[ht!]
		\centering
		\resizebox{0.93\textwidth}{!}{
			\begin{subfigure}[b]{0.5\textwidth}
				\centering
				\resizebox{1\textwidth}{!}{
					\begin{tikzpicture}
						\clip(-0.31,-0.3) rectangle (6.4,4.3);
						
						\draw[line width=0.6mm, color=black] plot[smooth, tension=0.6] coordinates {(0.5,3.5)  (1.5,3.5) (2,3.5)  (2.5,2)  (4,3.5) (4.5,3.5) (5.5,3.5)};
						\draw[line width=0.6mm, color=black] (0.5,0.5)--(5.5,0.5);
						
						\draw[dashed,line width=1mm, color=black] (0.5,0.5)--(0.5,3.5);
						\draw[dashed,line width=1mm, color=black] (5.5,0.5)--(5.5,3.5); 
						
						\coordinate[label=left:\Large{$\mathcal{L}_1$}] (D) at (0.595,2);
						\coordinate[label=right:\Large{$\mathcal{L}_2$}] (D) at (5.45,2);
						
						\draw [line width=0.35mm, color=MidnightBlue] (-0.3,-0.25) rectangle (0.7, 4.25);
						\draw [line width=0.35mm, color=MidnightBlue] (5,-0.25) rectangle (6.35, 4.25);
					\end{tikzpicture}
				}
				\label{Figure:example1}
			\end{subfigure}
			\hfill \quad \hfill
			\begin{subfigure}[b]{0.5\textwidth}
				\centering 
				\tikzset{>={Latex[width=2mm,length=2mm]}}
				\resizebox{0.87\textwidth}{!}{
					\begin{tikzpicture}
						\clip(-0.25, 0) rectangle (3.3,4);
						
						\draw[line width=0.35mm, color=black] plot[smooth, tension=1] coordinates {(0.3,0.9) (0.6,0.9) (1,0.9) (1.5,2.3) (1.8,1) (2.9,3) (2.5,1.3)  (3,0.2) (1,0.2) (0.6,0.2) (0.3,0.2)};

						\draw[dotted,line width=0.5mm, color=black] (0.3,0.179)--(0.3,0.93); 
						
						\coordinate[label=right:\small{$\mathcal{L}_1$}] (D) at (0.41,2);
						
						\draw[line width=0.2mm, ->] (0.5,2)  to [out=-150,in=150,looseness=0.8] (0.3, 0.55);
						
						\draw [line width=0.25mm, color=MidnightBlue] (-0.2,0.05) rectangle (0.5, 1.5);
					\end{tikzpicture}
				}
				\label{Figure:example2}
			\end{subfigure}
			\begin{subfigure}[b]{0.401\textwidth}
				\centering
				\tikzset{>={Latex[width=3mm,length=3mm]}}
				\resizebox{1\textwidth}{!}{
					\begin{tikzpicture}
						\clip(0,-0.2) rectangle (5,4.9);
						\draw[line width=0.5mm, color=black] (80:2.5) coordinate (A) arc (80:10:2.5) coordinate (C)  (10:3.9) coordinate (D) arc (10:80:3.9) coordinate (B);
						
						\draw[dashed,line width=1mm, color=black] (A) -- (B);
						\draw[dashed,line width=1mm, color=black] (C) -- (D);
						
						\draw[line width=0.2mm, color=MidnightBlue] (18:1.5) coordinate (AA) arc 	(28:5:1.5) coordinate (CC)  (28:4.9) coordinate (DD) arc (28:5:4.9) coordinate (BB);
						\draw[line width=0.2mm, color=MidnightBlue] (18:1.5) coordinate (AA) arc 	(28:5:1.5) coordinate (CC)  (28:4.9) coordinate (DD) arc (28:5:4.9) coordinate (BB);
						\draw[line width=0.2mm, color=MidnightBlue] (CC) -- (BB);
						\draw[line width=0.2mm, color=MidnightBlue] (DD) -- (AA);
						
						\draw[line width=0.2mm, color=MidnightBlue] (85:1.5) coordinate (AAA) arc 	(85:55:1.5) coordinate (CCC)  (85:4.9) coordinate (DDD) arc (85:55:4.9) coordinate (BBB);
						\draw[line width=0.2mm, color=MidnightBlue] (85:1.5) coordinate (AAA) arc 	(85:55:1.5) coordinate (CCC)  (85:4.9) coordinate (DD) arc (85:55:4.9) coordinate (BBB);
						\draw[line width=0.2mm, color=MidnightBlue] (CCC) -- (BBB);
						\draw[line width=0.2mm, color=MidnightBlue] (DDD) -- (AAA);
						\coordinate[label=right:\Large{$\mathcal{L}_1$}] (D) at (0.6,4.3);
						\coordinate[label=right:\Large{$\mathcal{L}_2$}] (D) at (3.3,2.5);
						\draw[line width=0.2mm, ->] (3.5,2.5)  to [out=180,in=90,looseness=0.8] (3.1, 0.65);
						\draw[line width=0.2mm, ->] (1.39,4.3)  to [out=0,in=-10,looseness=0.8] (0.65, 3.1);
					\end{tikzpicture}
				}
				\label{Figure:example3}
			\end{subfigure}
		}
		\caption{Several domains satisfying \Cref{assumption:funnel} are displayed. The dashed lines refer to~$\Gamma_{\operatorname{c}}$. Exemplary neighborhoods $\xU_l$ of $\overline{\mathcal{L}_l}$, with $l \in \{1,\dots, i_{\operatorname{c}}\}$, are indicated.}
		\label{Figure:exampledomains}
	\end{figure}
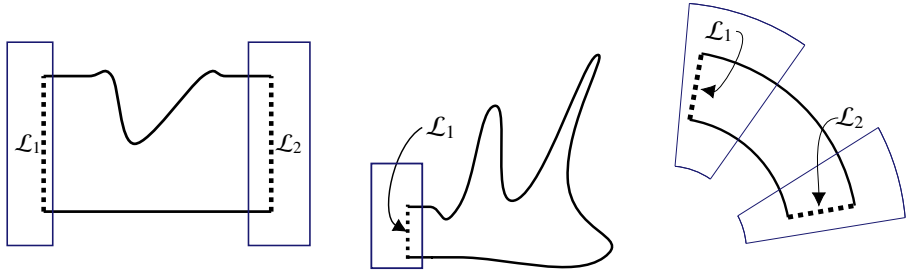

	The next lemma states that \Cref{assumption:funnel} implies \Cref{assumption:extensionproperty} under sharp constraints on the data (\cf~\Cref{remark:strcstme}). Hereto, we decompose
	\begin{equation}\label{equation:ddcomp}
	 	\Gamma\setminus\Gamma_{\operatorname{c}} = \mathscr{G}^1 \cup \mathscr{G_2}^2, \quad \mathscr{G}^1 \neq \emptyset
	\end{equation}
 	so that $\Omega$ admits a bounded doubly-connected extension (through $\Gamma_{\operatorname{c}}$) whose boundary comprises disjoint Jordan curves $\widetilde{\mathscr{G}}^1$ and~$\widetilde{\mathscr{G}}^2$ satisfying $\mathscr{G}^i \subset \widetilde{\mathscr{G}}^i$ for~$i\in\{1,2\}$.   
	
	\begin{lmm}\label{lemma:implassmp}
		Suppose that $\Gamma\setminus\Gamma_{\operatorname{c}}$ has a decomposition like \eqref{equation:ddcomp} mentioned above. Moreover, the states $\xvec{u}_0, \xvec{B}_0, \xvec{u}_T, \xvec{B}_T \in \xCn{{m,\alpha}}_{*}(\overline{\Omega}, \Gamma_{\operatorname{c}}; \mathbb{R}^2)$ are assumed to admit respective stream functions $\phi_0, \psi_0, \phi_T, \psi_T \in \xCn{{m+1,\alpha}}(\overline{\Omega}; \mathbb{R})$ such that
		\begin{equation}\label{equation:condntr}
			\begin{cases}
				\phi_0(\xvec{x}) = \psi_0(\xvec{x}) = \phi_T(\xvec{x}) = \psi_T(\xvec{x}) = 0 & \mbox{ if } \xvec{x} \in \mathscr{G}^1,\\
				\psi_0(\xvec{x}) = \psi_T(\xvec{x}) = 0, \, \phi_0(\xvec{x}) = d_1, \, \phi_T(\xvec{x}) = d_2 & \mbox{ if }  \xvec{x} \in \mathscr{G}^2,
			\end{cases}
		\end{equation}
		where $d_1,d_2\in \mathbb{R}$. Then, \Cref{assumption:funnel} yields \Cref{assumption:extensionproperty}.
	\end{lmm}
	\begin{proof}
		After extending $\Omega$ in a first step, the stream functions $\phi_0$, $\psi_0$, $\phi_T$, and $\psi_T$ are continued via reflections. For brevity, we only focus on the cases where either $\mathscr{G}^2=\emptyset$ or $d_1 = d_2 = 0$. Otherwise, one could begin with deciding on a doubly-connected extension $\mathcal{E}$ having two connected boundary components $\widetilde{\mathscr{G}}^1$ and $\widetilde{\mathscr{G}}^2$ so that $\mathscr{G}^j \subset \widetilde{\mathscr{G}}^j$. Then, one would solve the Dirichlet problems
		\[
			\Delta \eta_i = 0 \, \mbox{ in } \mathcal{E}, \quad \eta_i = 0 \, \mbox{ on } \mathscr{G}^1, \quad \eta_i = d_i \, \mbox{ on } \mathscr{G}^2, \quad i \in \{1,2\}
		\] 
		and apply the below arguments using the modified velocities $\xvec{u}_0 - \xnab^{\perp} \eta_1$ and $\xvec{u}_T - \xnab^{\perp} \eta_2$.
		\paragraph{Step 1. Doubly-connected domain extension.} Let the neighborhoods $\xU_1, \dots, \xU_{i_{\operatorname{c}}}$ of $\overline{\mathcal{L}}_1, \dots \overline{\mathcal{L}}_{i_{\operatorname{c}}}$ be given as in \Cref{assumption:funnel}. Hence, there exists a bounded simply-connected domain $\widetilde{\Omega} \supset \Omega$ with $(\Gamma\setminus \Gamma_{\operatorname{c}}) \subset \partial \widetilde{\Omega}$, and so that each $\xU_l\cap \widetilde{\Omega}$ represents either a rectangle or an annulus sector.
		Thus, we can take any open ball $\emptyset\neq\widetilde{\xB} \subset \widetilde{\Omega}\setminus\overline{\Omega}$ and define $\mathcal{E} \coloneq \widetilde{\Omega} \setminus \widetilde{\xB}$ (\cf~\Cref{Figure:extension0}). 
		
		\paragraph{Remark.} The idea is slightly different when $d_1 \neq 0$ or $d_2 \neq 0$. In that case, one extends $\mathscr{G}^1$ and $\mathscr{G}^2$ to smooth disjoint nested Jordan curves enclosing $\mathcal{E}$, as shown for an example in \Cref{Figure:extension0}. The subsequent steps are adapted accordingly.
		\paragraph{Step 2. Extensions of the data.} After cropping members of the family $(\xU_l)_{l\in\{1,\dots,i_{\operatorname{c}}\}}$ in a way that avoids intersecting the cavity $\widetilde{\xB}$, the previous step provides pairwise disjoint open neighborhoods $\widetilde{\xU}_1, \dots, \widetilde{\xU}_{i_{\operatorname{c}}}$ of $\overline{\mathcal{L}}_1, \dots \overline{\mathcal{L}}_{i_{\operatorname{c}}}$ such that $\widetilde{\xU}_l \cap \overline{\mathcal{E}}$ constitutes for each $l \in \{1,\dots,i_{\operatorname{c}}\}$ either a rectangle or an annulus sector (\cf~\Cref{Figure:extension0}). Depending on the respective shapes of $\widetilde{\xU}_1\cap\overline{\mathcal{E}}, \dots, \widetilde{\xU}_{i_{\operatorname{c}}}\cap\overline{\mathcal{E}}$, the stream functions of the initial and target states are continued beyond $\Gamma_{\operatorname{c}}$ to
		\[
			\Omega_1 \coloneq \Omega \cup (\mathcal{E}\cap \widetilde{\xU}_1)\cup\dots\cup (\mathcal{E}\cap \widetilde{\xU}_{i_{\operatorname{c}}})
		\]
		in the following manner.
		\begin{itemize}
			\item When $\widetilde{\xU}_l \cap \overline{\mathcal{E}}$ is a rectangle, reflections are used at $\Gamma_{\operatorname{c}}$ in Cartesian coordinates for extending $\phi_0, \psi_0, \phi_T, \psi_T$ to $\Omega \cup (\mathcal{E}\cap \widetilde{\xU}_l)$.
			\item If $\widetilde{\xU}_l \cap \overline{\mathcal{E}}$ is an annulus sector, rotational reflections at $\Gamma_{\operatorname{c}}$ are employed for continuing $\phi_0, \psi_0, \phi_T, \psi_T$ to $\Omega \cup (\mathcal{E}\cap \widetilde{\xU}_l)$. As an example that goes without loss of generality, we consider $0 < r_1 < r_2 < +\infty$ and extend functions
			\[
				(r, \theta) \mapsto \psi(r\cos(\theta),r\sin(\theta))
			\]
			from $[r_1, r_2] \times [-\pi/4,0]$ to the larger domain $[ r_1, r_2]\times[-\pi/4,\pi/4]$ by assigning
			\begin{equation*}
				(r, \theta) \mapsto \begin{cases}
					\sum_{j=1}^{m+2} a_j \psi(r\cos(-j^{-1}\theta), r\sin(-j^{-1}\theta)) & \mbox{ if } \theta \in(0,\pi/4],\\
					\psi(r\cos(\theta),r\sin(\theta)) & \mbox{ if } \theta \in [-\pi/4,0],
				\end{cases}
			\end{equation*}
			where the coefficients $(a_j)_{j\in\{1,\dots,m+2\}}$ solve a linear Vandermonde system.
		\end{itemize} 
		The stream function extensions beyond $\Gamma_{\operatorname{c}}$, as obtained in the above-described way, are denoted by $\widehat{\phi}_0, \widehat{\psi}_0, \widehat{\phi}_T, \widehat{\psi}_T \in \xCn{{m+1,\alpha}}(\overline{\Omega}_1; \mathbb{R})$. Since the zero boundary values of the original stream functions are reflected at $\Gamma_{\operatorname{c}}$, there exists~$s_0 > 0$ such that
		\[
			\widehat{\phi}_0(\xvec{x}) = \widehat{\psi}_0(\xvec{x}) = \widehat{\phi}_T(\xvec{x}) = \widehat{\psi}_T(\xvec{x}) = 0
		\]
		for all $\xvec{x} \in \partial \Omega_1$ with $\operatorname{dist}(\xvec{x}, \overline{\Omega}) < s_0$. Now, let $\chi_0\colon \mathbb{R}^2 \longrightarrow [0,1]$ be a smooth cutoff satisfying
		\[
			\chi_0(\xvec{x}) = \begin{cases}
				1 & \mbox{ if } \xvec{x} \in \overline{\Omega},\\
				0 & \mbox{ if } \operatorname{dist}(\xvec{x}, \overline{\Omega}) > s_0/2.
			\end{cases}
		\]
		Moreover, fix arbitrary continuations of $\widehat{\phi}_0, \widehat{\psi}_0, \widehat{\phi}_T, \widehat{\psi}_T$  beyond $\overline{\Omega}_1$ to $\overline{\mathcal{E}}$, for instance,
		\[
			\widetilde{\phi}_0, \widetilde{\psi}_0, \widetilde{\phi}_T, \widetilde{\psi}_T = 
			\begin{cases}
				\widehat{\phi}_0, \widehat{\psi}_0, \widehat{\phi}_T, \widehat{\psi}_T & \mbox{ in } \overline{\Omega}_1,\\
				0 & \mbox{ otherwise.}
			\end{cases}
		\]
		Ultimately, we define
		\[
			\widetilde{\xvec{u}}_0 = \xnab^{\perp} \left(\chi_0\widetilde{\phi}_0\right), \quad \widetilde{\xvec{B}}_0 = \xnab^{\perp} \left(\chi_0\widetilde{\psi}_0\right), \quad \widetilde{\xvec{u}}_T = \xnab^{\perp} \left(\chi_0\widetilde{\phi}_T\right), \quad \widetilde{\xvec{B}}_T = \xnab^{\perp} \left(\chi_0\widetilde{\psi}_T\right).
		\]
		\end{proof}
			
		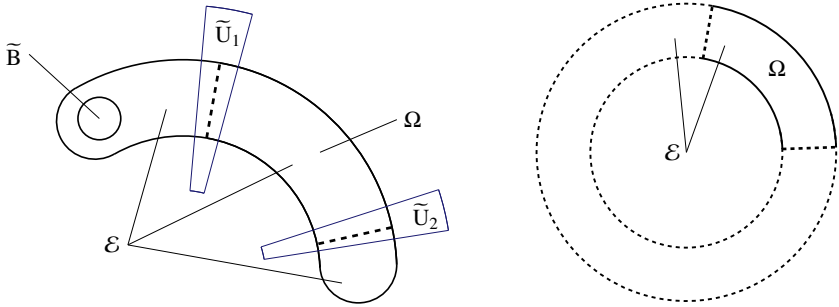
\begin{figure}[ht!]
			\centering
			\resizebox{0.93\textwidth}{!}{
				\begin{subfigure}[b]{0.5\textwidth}
					\centering
					\resizebox{1\textwidth}{!}{
						\begin{tikzpicture}
							\clip(-3.5,-0.6) rectangle (5,4.9);
							\draw[line width=0.3mm, color=black] (80:2.5) coordinate (A) arc 	(80:2:2.5) coordinate (C)  (2:3.9) coordinate (D) arc (2:80:3.9) coordinate (B);
							\draw[line width=0.3mm, color=black] (80:2.5) coordinate (A) arc 	(80:2:2.5) coordinate (C)  (2:3.9) coordinate (D) arc (2:80:3.9) coordinate (B);

							\draw[line width=0.3mm, color=black] (80:2.5) coordinate (E) arc 	(80:120:2.5) coordinate (F)  (80:3.9) coordinate (G) arc (80:120:3.9) coordinate (H);

							\draw[line width=0.3mm, color=black] (-1.95,3.38) arc[start angle=120, end 	angle=300, x radius=0.7, y radius=0.7];

							\draw[line width=0.3mm, color=black] (D) arc[start angle=360, end 	angle=180, x radius=0.7, y radius=0.7];
							
							\filldraw[color=black, fill=white, line width=0.3mm,  fill opacity=1] 	(-1.52,2.82) circle (0.39);

							\draw[line width=0.1mm, color=MidnightBlue] (18:1.5) coordinate (AA) arc 	(18:9:1.5) coordinate (CC)  (18:4.9) coordinate (DD) arc (18:9:4.9) coordinate (BB);
							\draw[line width=0.1mm, color=MidnightBlue] (18:1.5) coordinate (AA) arc 	(18:9:1.5) coordinate (CC)  (18:4.9) coordinate (DD) arc (18:9:4.9) coordinate (BB);
							\draw[line width=0.1mm, color=MidnightBlue] (CC) -- (BB);
							\draw[line width=0.1mm, color=MidnightBlue] (DD) -- (AA);
							
							\draw[line width=0.1mm, color=MidnightBlue] (85:1.5) coordinate (AAA) arc 	(85:75:1.5) coordinate (CCC)  (85:4.9) coordinate (DDD) arc (85:75:4.9) coordinate (BBB);
							\draw[line width=0.1mm, color=MidnightBlue] (85:1.5) coordinate (AAA) arc 	(85:75:1.5) coordinate (CCC)  (85:4.9) coordinate (DD) arc (85:75:4.9) coordinate (BBB);
							\draw[line width=0.1mm, color=MidnightBlue] (CCC) -- (BBB);
							\draw[line width=0.1mm, color=MidnightBlue] (DDD) -- (AAA);
							
							\draw[dashed,line width=0.5mm, color=black] (A) -- (B);
							\draw[dashed,line width=0.5mm, color=black] (2.47,0.52) -- (3.8,0.82);
							\coordinate[label=right:\large{$\Omega$}] (L) at (3.9,2.8);
							\draw[line width=0.1mm] (2.5,2.2) -- (L);
							\coordinate[label=left:\large{$\mathcal{E}$}] (H) at (-1,0.5);
							\draw[line width=0.1mm] (2,1.97) -- (H);
							\draw[line width=0.1mm] (-0.3,2.99) -- (H);
							\draw[line width=0.1mm] (2.9,-0.2) -- (H);
							\coordinate[label=left:\large{$\widetilde{\xB}$}] (I) at (-2.8,4);
							\draw[line width=0.1mm] (-1.52,2.82) -- (I);
							\coordinate[label=left:\large{$\widetilde{\xU}_1$}] (U1) at (1.199,4.39);
							\coordinate[label=left:\large{$\widetilde{\xU}_2$}] (U2) at (4.8,1.03);
						\end{tikzpicture}
					}
					\label{Figure:p1}
				\end{subfigure}
				\quad \quad
				\begin{subfigure}[b]{0.39\textwidth}
					\centering
					\resizebox{1\textwidth}{!}{
						\begin{tikzpicture}
							\clip(-4,-4) rectangle (5.5,4.9);
							\draw[line width=0.5mm, color=black] (80:2.5) coordinate (A) arc 	(80:2:2.5) coordinate (C)  (2:3.9) coordinate (D) arc (2:80:3.9) coordinate (B);

							\draw[line width=0.5mm, dashed, color=black] (80:2.5) coordinate (E) arc 	(80:375:2.5) coordinate (F)  (80:3.9) coordinate (G) arc (80:375:3.9) coordinate (H);
											
							\draw[dashed,line width=0.7mm, color=black] (A) -- 	(B);
							\draw[dashed,line width=0.7mm, color=black] (C) -- (D);
							\coordinate[label=right:\LARGE{$\Omega$}] (L) at (2,2.2);
							
							\coordinate[label=left:\LARGE{$\mathcal{E}$}] (H) at (0,0);
							\draw[line width=0.2mm] (1,2.8) -- (H);
							\draw[line width=0.2mm] (-0.3,2.99) -- (H);
						\end{tikzpicture}
					}
					\label{Figure:p2}
			\end{subfigure}}
			\caption{Illustrations of two different doubly-connected extensions $\mathcal{E}$ for an annulus sector $\Omega$ where~$\Gamma_{\operatorname{c}}$ consists of two disjoint line segments. When $d_1 = d_2 = 0$ in \Cref{lemma:implassmp}, the cavity $\widetilde{\xB}$ is located in the extended part. Here, the intersections $\widetilde{\xU}_1\cap\overline{\mathcal{E}}$ and $\widetilde{\xU}_2\cap\overline{\mathcal{E}}$ are annulus sectors. The right sketch refers to the case where $\mathscr{G}^2 \neq \emptyset$ with $d_1 \neq 0$ or $d_2 \neq 0$.}
			\label{Figure:extension0}
		\end{figure}

	\begin{rmrk}\label{remark:strcstme}
		If Assumptions~\Rref{assumption:extensionproperty} and \Rref{assumption:funnel} are both true, one can choose $\mathscr{G}^i$ and stream functions with \eqref{equation:condntr}. Indeed, as seen via integration by parts, extensions $\widetilde{\xvec{B}}_0$ and $\widetilde{\xvec{B}}_T$ obeying \Cref{assumption:extensionproperty} have stream functions that are constant on $\partial \mathcal{E}$ (\cf~\eqref{equation:constarg}).
	\end{rmrk}

	\begin{xmpl}\label{example:data}
		All $\xvec{u}_0, \xvec{B}_0, \xvec{u}_T, \xvec{B}_T \!\in\! \xCn{{m,\alpha}}_{*}(\overline{\Omega}, \emptyset; \mathbb{R}^2)$ admit stream functions with~\eqref{equation:condntr}.
	\end{xmpl}
	\begin{xmpl}
		If, for instance, $\Gamma\setminus\Gamma_{\operatorname{c}} = \mathscr{G}^1 \cup \mathscr{G}^2 \cup \mathscr{G}^3$ and $\xvec{u}_0 = \xnab^{\perp}\phi$ with constant values~$\phi = c_i$ on~$\mathscr{G}^i$, $i \in \{1,2,3\}$, and $c_1 \neq c_2 \neq c_3 \neq c_1$, then \Cref{assumption:extensionproperty} cannot hold. Indeed, as $\widetilde{\xvec{u}}_0$ must be tangential at $\partial \mathcal{E}$, its stream functions are constant on each connected component of $\partial \mathcal{E}$; hence, $\mathcal{E}$ could not be doubly-connected.
	\end{xmpl}
	\begin{xmpl}
		Assume, \eg, that the uncontrolled boundary $\Gamma\setminus \Gamma_{\operatorname{c}}$ of $\Omega$ comprises two connected components~$\mathscr{G}^1$ and $\mathscr{G}^2$, and that the stream functions of~$\xvec{B}_0$ (or that of~$\xvec{B}_T$) have different constant values on~$\mathscr{G}^1$ compared to their constant values on~$\mathscr{G}^2$. Then, if~$\mathcal{E}$ denotes any doubly-connected extension as in~\Cref{assumption:extensionproperty}, the uncontrolled boundary~$\Gamma\setminus \Gamma_{\operatorname{c}}$ must intersect all connected components of~$\partial \mathcal{E}$, as shown also in \Cref{Figure:extension0}. In particular, as the boundary conditions for the magnetic field require stream functions that are constant at each connected component of the boundary, an arbitrarily extended magnetic field~$\widetilde{\xvec{B}}_0$ of~$\xvec{B}_0$ (or~$\widetilde{\xvec{B}}_T$ of~$\xvec{B}_T$) can generally not be corrected inside~$\overline{\mathcal{E}}\setminus \overline{\Omega}$ to ensure the last condition of \Cref{assumption:extensionproperty}.
	\end{xmpl}

	\subsection{Past, present, and future}
	The controllability of perfect fluids once constituted a challenging open problem raised by J.-L. Lions, \eg, in \cite{LionsJL1991}. Among the various sources of trouble stands out the failure of a common linear test, \eg, the equation $\partial_t \xvec{u} + \xnab p = 0$
	conserves vorticity. An answer was found in the 1990s due to J.-M. Coron, who achieved the breakthrough by exploiting the nonlinear mixing term $(\xvec{u}\cdot\xnab)\xvec{u}$ to establish his return method for the first time in a PDE context. This approach involves a sophisticated time-periodic trajectory reflecting the domain's topology~(\cf~\cite{Coron1996EulerEq,Glass2000,Fernandez-CaraSantosSouza2016} and \cite[Part 2, Section 6.2]{Coron2007}). When the controls are supported in the interior, obstructions emerge in the aftermath of Kelvin's circulation theorem: there exist invariant boundary circulations. Nevertheless, the approximate controllability holds in~$\xLn{p}$ with~$p \in [1,+\infty)$, and supplementary hypotheses facilitate improvements (\cf~\cite{Glass2001,Coron1996EulerEq}). 
	Departing towards incompressible ideal MHD, topological phenomena and coupling effects are causing new difficulties such as: a) the necessity of \eqref{equation:intrprctrl}; b) our current limitation to impose \eqref{equation:conditionq}; c) regularity loss issues related to fixed point iterations. While closed estimates are available for symmetrized unknowns, the associated linearized problems exhibit inhomogeneities that would impede flushing the magnetic field as suggested by the known return method argument for incompressible Euler. So far, only \cite{RisselWang2021,KukavicaNovackVicol2022} have investigated the controllability of ideal MHD, where \cite{KukavicaNovackVicol2022} improves \cite{RisselWang2021} by removing or characterizing an undesired bulk force. However, both studies utilize special shear profiles only accessible in rectangular channels steered from two opposing walls. Further, the interior controllability has apparently not been addressed in the ideal MHD literature yet.
	
	In this article it is observed that interior controls appear as natural building blocks when devising a method that suits curved boundaries, thus we can provide a first interior controllability result for incompressible ideal MHD while not being limited to straight channels. In the present case, the controls' support intersects each connected component of the boundary. Since we cover now rather general geometries, crucial arguments used in \cite{RisselWang2021,KukavicaNovackVicol2022} are unavailable: a) in \cite{KukavicaNovackVicol2022} the incompressible ideal MHD system is solved near a spatially constant shear velocity by using cancellations that only occur because that shear profile depends merely on time; b) in \cite{KukavicaNovackVicol2022}, the magnetic field is automatically flushed by zero data imposed outside the physical domain, while in the present situation a complicated mixture of regions, those where the magnetic field vanishes and those where it is supported, would be circulated around the cavity (see also \Cref{Figure:flushing0} in \Cref{subsection:flushingprofile}).
	Another intricacy consists of devising an Ansatz for the controls, motivations being drawn from the proof of a lemma in \cite[Appendix]{CoronMarbachSueur2020}, where controls for transport problems are described by means of uncontrolled linear equations with localized initial data. In stark contrast, our controls are assembled from basic building blocks at a nonlinear level; furthermore, the constraints in~\eqref{equation:intrprctrl} must be upheld. These building blocks are derived from nonlinear velocity-controlled ideal MHD sub-problems, individually solved by resorting to a return method argument that improves the ideal MHD adaptation \cite{RisselWang2021} of \cite{Coron1996EulerEq}. In this respect, good estimates for a fixed point map rely on a time-weighted norm from \cite{RisselWang2021} and an everywhere-curl-free choice of the return method profile for the velocity; indeed, to ensure~\eqref{equation:intrprctrl} with a control zone~$\omegaup$ of arbitrarily small area, the magnetic field problem has to be properly solved even in~$\omegaup$. When looking at the controlled evolution described by a finite combination of building blocks set on consecutive time intervals, the magnetic field will be deleted in a small \enquote{sector} and erased parts spread downstream, thanks to a frozen-in flux phenomenon; after a short time, the procedure will be restarted with an initial magnetic field having reduced support. Localized disturbances, generated either by a regularity corrector or by an unwanted source term, are eventually suppressed with the help of a local flushing mechanism. A finite number of steps suffices to neutralize the entire magnetic field, and it remains to solve a controlled incompressible Euler problem.

	Several natural questions are left unanswered.  1) How to treat general domains in~$2$D? When the~magnetic field is dragged around multiple cavities, the building blocks for the controls, and the way they are combined, would be more complex as for the situation described here. 2) What versions, or obstructions, of Theorems~\Rref{theorem:maininterior}~and~\Rref{theorem:main} could be expected in $3$D? Due to possibly involved topological features of the field lines in three dimensions, new developments are required to localize the~magnetic~field control, and to characterize admissible initial and target states for which this is possible. 3) Are the requirements on~$\omegaup$ optimal (\cf~\Cref{problem:p1})? For having exact controllability, the control region needs to intersect all~connected components of the boundary, as explained for incompressible Euler in \cite{Coron1996EulerEq}; however, it still might seem restrictive to assume that~$\omegaup$ contains a set $\Lambda$ with~$\overline{\mathcal{E}}\setminus\Lambda$ being simply-connected. 4) It also remains open whether magnetic fields having equal nontrivial first cohomology projections can be connected (see~\Cref{problem:i}, \cite{KukavicaNovackVicol2022}). 
	
	\begin{figure}[ht!]\centering\resizebox{0.85\textwidth}{!}{\centering 
			\begin{tikzpicture}[node distance=0.5cm]
				\node (A) [blue0] {\Large\Cref{theorem:maininterior} (main result I)};
				\node (AB) [right = 1.2cm of A] {};
				\node (B) [blue0, right = 0.8cm of A] {\Large\Cref{theorem:main} (main result II)};
				
				\node (C) [blue0, below = 0.8cm of A] {\Large\Cref{theorem:local} \\ (local = global)};

				\node (ED) [below = 0.8cm of C] {};
				
				\node (E) [blue00c, below = 0.8cm of C] {\Large{\Cref{theorem:tkloc} and \Cref{subsubsection:piecingtogether}} \\ (divide-and-control)};
				\node (D) [blue0, left = 0.8cm of E] {\Large\Cref{theorem:ret}
					(return-method)};

				\draw [arrow,line width=0.5mm] (C) -- (A);
				\draw [arrow,line width=0.5mm] (A) -- (B);
				\draw [arrow,line width=0.5mm] (D) -- (E);
				\draw [arrow,line width=0.5mm] (D) -- (C);
				\draw [arrow,line width=0.5mm] (E) -- (C);
		\end{tikzpicture}}
		\caption{Relationships between the theorems appearing in this article.}
		\label{Figure:at}
	\end{figure}
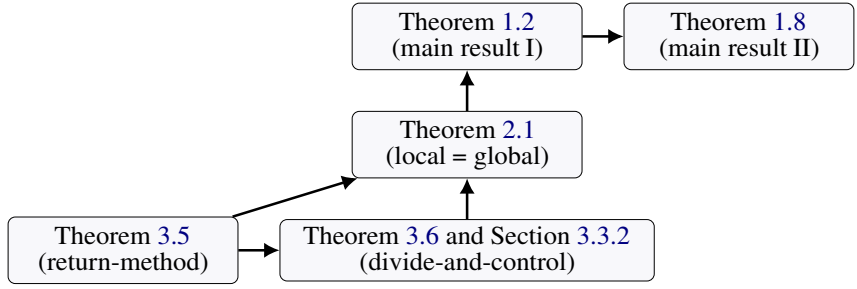
	
	\subsection{Organization}
	In \Cref{section:hysc}, scaling and time-reversibility properties allow reducing \Cref{theorem:maininterior} to the local exact null controllability furnished by \Cref{theorem:local}. \Cref{section:local} then demonstrates \Cref{theorem:local}. Hereto, the interior controls are obtained in \Cref{subsubsection:piecingtogether} through a combination of Theorems~\Rref{theorem:ret}~and~\Rref{theorem:tkloc}; see also \Cref{Figure:at}.

	\section{Local implies global}\label{section:hysc}
	We recast \Cref{theorem:maininterior} as the following local exact null controllability result with the fixed reference time $T = 2$. Its proof will be the topic of \Cref{section:local}; in particular, the argument is concluded in \Cref{subsubsection:piecingtogether}. For simplicity, $\xCn{{m,\alpha}}_{*}(\overline{\mathcal{E}}; \mathbb{R}^2) \coloneq \xCn{{m,\alpha}}_{*}(\overline{\mathcal{E}}, \emptyset; \mathbb{R}^2)$. 
	\begin{thrm}\label{theorem:local}
		There exists a (possibly small) constant $\delta > 0$ such that, for all initial states $\xvec{u}_0, \xvec{B}_0 \in \xCn{{m,\alpha}}_{*}(\overline{\mathcal{E}}; \mathbb{R}^2)$ obeying the assumptions of \Cref{theorem:maininterior} and
		satisfying
		\begin{equation}\label{equation:smallnesscond}
			\|\xvec{u}_0\|_{m,\alpha,\mathcal{E}} + \|\xvec{B}_0\|_{m,\alpha,\mathcal{E}} < \delta,
		\end{equation}
		there are controls
		\begin{gather*}
			\xsym{\xi}, \xsym{\eta} \in \xLinfty((0,2);\xCn{{m-1,\alpha}}(\overline{\mathcal{E}};\mathbb{R}^2)), \quad
			\operatorname{supp}(\xsym{\xi}) \cup \operatorname{supp}(\xsym{\eta}) \subset \omegaup\times [0,2]
		\end{gather*}
	 	ensuring that
		\begin{equation}\label{equation:MHD00scaling}
			\begin{cases}
				\partial_t \xvec{u} + (\xvec{u} \cdot \xsym{\nabla}) \xvec{u} - (\xvec{B} \cdot \xsym{\nabla})\xvec{B} + \xsym{\nabla} p = \mathbb{I}_{\omegaup}\xsym{\xi} & \mbox{ in } \mathcal{E} \times (0,2),\\
				\partial_t \xvec{B} + (\xvec{u} \cdot \xsym{\nabla}) \xvec{B} - (\xvec{B} \cdot \xsym{\nabla}) \xvec{u} = \mathbb{I}_{\omegaup}\xsym{\eta} & \mbox{ in } \mathcal{E} \times (0,2),\\
				\xdiv{\xvec{u}} = \xdiv{\xvec{B}} = 0  & \mbox{ in } \mathcal{E} \times (0,2),\\
				\xvec{u} \cdot \xvec{n}_{\mathcal{E}} = \xvec{B} \cdot \xvec{n}_{\mathcal{E}} = 0  & \mbox{ on } \partial\mathcal{E} \times(0,2),\\
				\xvec{u}(\cdot, 0)  =  \xvec{u}_0,\, \xvec{B}(\cdot, 0)  =  \xvec{B}_0  & \mbox{ in } \mathcal{E},
			\end{cases}
		\end{equation}
		admits a solution
		\begin{equation*}
			\begin{gathered}
				\xvec{u}, \xvec{B} \in \xCn{0}([0,2];\xCn{{m-1,\alpha}}(\overline{\mathcal{E}}; \mathbb{R}^2)) \cap \xLinfty((0,2);\xCn{{m,\alpha}}(\overline{\mathcal{E}}; \mathbb{R}^2)), \\
				p \in \xCn{0}([0,2];\xCn{{m-1,\alpha}}(\overline{\mathcal{E}}; \mathbb{R})) \cap \xLinfty((0,2);\xCn{{m,\alpha}}(\overline{\mathcal{E}}; \mathbb{R})),
			\end{gathered}
		\end{equation*}
		meeting the target constraints
		\begin{equation*}\label{equation:MHD-Endconditionlocal}
			\begin{aligned}
				\xvec{u}(\cdot, 2)  =  \xsym{0}, \quad \xvec{B}(\cdot, 2)  =  \xsym{0}.
			\end{aligned}
		\end{equation*}
	\end{thrm}
	
	To see that \Cref{theorem:local} already implies \Cref{theorem:maininterior} for any control time $T > 0$, assume that $\xvec{u}$, $\xvec{B}$, and $p$ fulfill for some $0 < r < s < +\infty$ the relations
	\begin{equation}\label{equation:mhdwoini}
		\begin{cases}
			\partial_t \xvec{u} + (\xvec{u} \cdot \xsym{\nabla}) \xvec{u} - (\xvec{B} \cdot \xsym{\nabla})\xvec{B} + \xsym{\nabla} p = \mathbb{I}_{\omegaup}\xsym{\xi} & \mbox{ in } \mathcal{E} \times (r,s),\\
			\partial_t \xvec{B} + (\xvec{u} \cdot \xsym{\nabla}) \xvec{B} - (\xvec{B} \cdot \xsym{\nabla}) \xvec{u} = \mathbb{I}_{\omegaup}\xsym{\eta} & \mbox{ in } \mathcal{E} \times (r,s).
		\end{cases}
	\end{equation}
	If $t \in (r,s)$ and $\overline{t}\in(r/\epsilon,s/\epsilon)$ for some $\epsilon > 0$, the system \eqref{equation:mhdwoini} is likewise satisfied by the time-reversed and scaled profiles
	\begin{gather*}\label{equation:IdealControl:reverse}
			\xvec{u}^-(\xvec{x},t) \coloneq - \xvec{u}(\xvec{x}, r+s-t), \quad
			\xvec{u}^{\epsilon}(\xvec{x},\overline{t}) \coloneq  \epsilon \xvec{u}(\xvec{x}, \epsilon \overline{t}), \\
			\xvec{B}^-(\xvec{x},t) \coloneq - \xvec{B}(\xvec{x}, r+s-t)	, \quad \xvec{B}^{\epsilon}(\xvec{x},\overline{t}) \coloneq  \epsilon \xvec{B}(\xvec{x}, \epsilon \overline{t}), \\
			p^-(\xvec{x},t) \coloneq  p(\xvec{x}, r+s-t), \quad p^{\epsilon}(\xvec{x},\overline{t}) \coloneq \epsilon^2 p(\xvec{x}, \epsilon \overline{t}),\\
			\xsym{\xi}^-(\xvec{x},t) \coloneq  \xsym{\xi}(\xvec{x}, r+s-t), \quad
			\xsym{\xi}^{\epsilon}(\xvec{x},\overline{t}) \coloneq \epsilon^2 \xsym{\xi}(\xvec{x}, \epsilon \overline{t}),\\
			\xsym{\eta}^-(\xvec{x},t) \coloneq  \xsym{\eta}(\xvec{x}, r+s-t), \quad \xsym{\eta}^{\epsilon}(\xvec{x},\overline{t}) \coloneq \epsilon^2 \xsym{\eta}(\xvec{x}, \epsilon \overline{t}).
	\end{gather*}
	Anticipating that \Cref{theorem:local} holds with some $\delta > 0$, we take any sufficiently small parameter $\epsilon = \epsilon(\delta) \in (0,T/4)$ and define for the data $\xvec{u}_0, \xvec{B}_0, \xvec{u}_T, \xvec{B}_T \in \xCn{{m,\alpha}}_{*}(\overline{\mathcal{E}}; \mathbb{R}^2)$ from \Cref{theorem:maininterior} the scaled versions
	\begin{equation*}
		\begin{gathered}
			\widetilde{\xvec{u}}_0^{\epsilon} \coloneq \epsilon\xvec{u}_0, \quad \widetilde{\xvec{B}}_0^{\epsilon} \coloneq \epsilon\xvec{B}_0, \quad \widehat{\xvec{u}}_0^{\epsilon} \coloneq -\epsilon\xvec{u}_T, \quad \widehat{\xvec{B}}_0^{\epsilon} \coloneq -\epsilon\xvec{B}_T, \\
			\|\widetilde{\xvec{u}}^{\epsilon}_0\|_{m, \alpha, \mathcal{E}} + \|\widetilde{\xvec{B}}^{\epsilon}_0\|_{m, \alpha, \mathcal{E}} + \|\widehat{\xvec{u}}^{\epsilon}_0\|_{m, \alpha, \mathcal{E}} + \|\widehat{\xvec{B}}^{\epsilon}_0\|_{m, \alpha, \mathcal{E}} < \delta.
		\end{gathered}
	\end{equation*}
	Accordingly, there are controls $\widetilde{\xsym{\xi}}^{\epsilon}, \widehat{\xsym{\xi}}^{\epsilon}, \widetilde{\xsym{\eta}}^{\epsilon}, \widehat{\xsym{\eta}}^{\epsilon} \in \xLinfty((0,2);\xCn{{m-2,\alpha}}(\overline{\mathcal{E}};\mathbb{R}^2))$
	for which the system \eqref{equation:MHD00scaling} admits respective solutions
	\begin{equation*}
		\begin{gathered}
			\widetilde{\xvec{u}}^{\epsilon}, \widetilde{\xvec{B}}^{\epsilon}, \widehat{\xvec{u}}^{\epsilon}, \widehat{\xvec{B}}^{\epsilon} \in  \xCn{0}([0,2];\xCn{{m-1,\alpha}}(\overline{\mathcal{E}}; \mathbb{R}^2)) \cap \xLinfty((0,2);\xCn{{m,\alpha}}(\overline{\mathcal{E}}; \mathbb{R}^2)), \\ \widetilde{p}^{\epsilon}, \widehat{p}^{\epsilon} \in \xCn{0}([0,2];\xCn{{m-1,\alpha}}(\overline{\mathcal{E}}; \mathbb{R})) \cap \xLinfty((0,2);\xCn{{m,\alpha}}(\overline{\mathcal{E}}; \mathbb{R})),
		\end{gathered}
	\end{equation*}
	satisfying in $\mathcal{E}$ the initial and target conditions
	\begin{equation*}\label{equation:MHD-Endconditionlocalb}
		\begin{gathered}
			\widetilde{\xvec{u}}^{\epsilon}(\cdot, 0) = \widetilde{\xvec{u}}^{\epsilon}_0, \quad \widetilde{\xvec{B}}^{\epsilon}(\cdot, 0) = \widetilde{\xvec{B}}^{\epsilon}_0, \quad \widehat{\xvec{u}}^{\epsilon}(\cdot, 0) = \widehat{\xvec{u}}^{\epsilon}_0, \quad \widehat{\xvec{B}}^{\epsilon}(\cdot, 0) = \widehat{\xvec{B}}^{\epsilon}_0, \\ \widetilde{\xvec{u}}^{\epsilon}(\cdot, 2) = \widehat{\xvec{u}}^{\epsilon}(\cdot, 2) = \xsym{0}, \quad \widetilde{\xvec{B}}^{\epsilon}(\cdot, 2) = \widehat{\xvec{B}}^{\epsilon}(\cdot, 2) = \xsym{0}.
		\end{gathered}
	\end{equation*}
	As explained in \cite{RisselWang2021} (\cf~\cite{Coron1996EulerEq,Glass2000} for the Euler system), gluing re-scaled forward and backward trajectories yields \Cref{theorem:maininterior}. Indeed, since $2\epsilon < T/2$, one may recover a suitable solution to \eqref{equation:MHD00scaling} by defining
	\[
		(\xvec{u}, \xvec{B}, p)(\xvec{x},t) \coloneq \begin{cases}
			 (\epsilon^{-1}\widetilde{\xvec{u}}^{\epsilon}, \epsilon^{-1}\widetilde{\xvec{B}}^{\epsilon}, \epsilon^{-2} \widetilde{p}^{\epsilon})(\xvec{x}, \epsilon^{-1} t) & \mbox{ if } t \in [0,2\epsilon],\\
			(\xsym{0}, \xsym{0}, 0) & \mbox{ if } t \in (2\epsilon,T-2\epsilon),\\
				 (-\epsilon^{-1}\widehat{\xvec{u}}^{\epsilon}, -\epsilon^{-1}\widehat{\xvec{B}}^{\epsilon},\epsilon^{-2} \widehat{p}^{\epsilon})(\xvec{x}, \epsilon^{-1}(T-t)) & \mbox{ if } t \in [T-2\epsilon, T],
		\end{cases}
	\]
	and the controls
	\[
	(\xsym{\xi}, \xsym{\eta})(\xvec{x},t) \coloneq \begin{cases}
		(\epsilon^{-2} \widetilde{\xsym{\xi}}^{\epsilon}, \epsilon^{-2} \widetilde{\xsym{\eta}}^{\epsilon})(\xvec{x}, \epsilon^{-1} t) & \mbox{ if } t \in [0,2\epsilon],\\
		(\xsym{0}, \xsym{0}) & \mbox{ if } t \in (2\epsilon,T-2\epsilon),\\
		(\epsilon^{-2} \widehat{\xsym{\xi}}^{\epsilon}, 	\epsilon^{-2} \widehat{\xsym{\eta}}^{\epsilon})(\xvec{x}, \epsilon^{-1}(T-t)) & \mbox{ if } t \in [T-2\epsilon, T].
	\end{cases}
	\]
	\begin{rmrk}
		The specific choice $T = 2$ in \Cref{theorem:local} is not reflecting technical reasons; it simply lightens up several notations in \Cref{section:local}. 
	\end{rmrk}
	
	\section{Exact null controllability of small solutions}\label{section:local}
	This section is devoted to proving \Cref{theorem:local}. The basic setup and a conservation property are given in \Cref{subsection:setup}, a return method profile is provided in \Cref{subsection:flushingprofile}, and \Cref{theorem:local} will be concluded in \Cref{subsection:nonlineardecomposition} (\cf~\Cref{subsubsection:piecingtogether}). 
	As the analysis will involve transport problems, we begin here with some observations. Consider a controlled ideal MHD solution
	\[
		\xvec{u} = [u_1, u_2], \quad \xvec{B} = [B_1, B_2], \quad p, \quad \xsym{\xi} = [\xi_1, \xi_2], \quad \xsym{\eta} = [\eta_1, \eta_2]
	\] 
	to \eqref{equation:MHD00scaling}; but, on a general time interval $(0, \widetilde{T})$ with $\widetilde{T} > 0$. Then, each~$B_i$ satisfies the transport equation
	\begin{equation}\label{equation:transport}
			\partial_t v + (\xvec{z} \cdot \xnab)v = g, \quad v(\cdot, 0) = v_0,
	\end{equation}
	where $v \coloneq B_i$, $\xvec{z} \coloneq \xvec{u}$, $g \coloneq (\xvec{B} \cdot \xnab) u_i +  \mathbb{I}_{\omegaup}\eta_i$, and $v_0 \coloneq B_i(\cdot, 0)$. As a consequence, one can~study the evolution of $\operatorname{supp}(\xvec{B})$ governed by the flow~$\xmcal{Z}$ associated with~$\xvec{z}$ (\cf~\Cref{lemma:mgftrsp} below). In particular, since $\xvec{z}$ is tangential at $\partial \mathcal{E}$, the method of characteristics provides for all $t$ the formula
	\begin{equation}\label{equation:moc}
		v(\xvec{x},t) = v_0(\xmcal{Z}(\xvec{x},t,0)) + \int_0^t g(\xmcal{Z}(\xvec{x},t,s),s) \, \xdx{s}.
	\end{equation}
	Moreover, given $j \in \{0, 1, 2, \dots\}$ and any
	\begin{gather*}
			v_0 \in  \xCn{{j,\alpha}}(\overline{\mathcal{E}};\mathbb{R}), \quad \xvec{z} \in \xCn{{j+1,\alpha}}_{*}(\overline{\mathcal{E}}; \mathbb{R}^2), \quad
			g \in \xCzero([0,\widetilde{T}]; \xCn{{j, \alpha}}(\overline{\mathcal{E}}; \mathbb{R}^2)),
	\end{gather*}
	one obtains from \eqref{equation:moc} that the solution to \eqref{equation:transport} obeys
	\begin{equation}\label{equation:alphabeta}
		v \in \xCzero([0,\widetilde{T}]; \xCn{{j,\beta}}(\overline{\mathcal{E}};\mathbb{R}))\cap\xLinfty((0,\widetilde{T}); \xCn{{j,\alpha}}(\overline{\mathcal{E}};\mathbb{R})), \quad \beta \in (0, \alpha).
	\end{equation}
	Indeed, when assuming for simplicity $g = 0$, the $\beta$-semi-norm of $v(\cdot, t) - v(\cdot, s)$ is bounded for any $0 \leq t \leq s$ by
	\begin{equation*}
		\begin{multlined}
			\sup\limits_{\xvec{x} \neq \xvec{y}} \Bigg[ \left(\frac{\left|v_0(\xmcal{Z}(\xvec{x},t,0))-v_0(\xmcal{Z}(\xvec{y},t,0))-v_0(\xmcal{Z}(\xvec{x},s,0)) + v_0(\xmcal{Z}(\xvec{y},s,0))\right|}{|\xvec{x}-\xvec{y}|^{\alpha}}\right)^{\beta/\alpha} \\
			\times \left|v_0(\xmcal{Z}(\xvec{x},t,0))-v_0(\xmcal{Z}(\xvec{y},t,0))-v_0(\xmcal{Z}(\xvec{x},s,0)) + v_0(\xmcal{Z}(\xvec{y},s,0))\right|^{1-\beta/\alpha}\Bigg],
		\end{multlined}
	\end{equation*}
	which tends to zero as $|t-s| \longrightarrow 0$. The case $g \neq 0$ works similarly, and, by taking derivatives in \eqref{equation:moc}, one inductively obtains by the same approach that
	\[
		\| v(\cdot, t) - v(\cdot, s) \|_{j, \beta, \mathcal{E}} \longrightarrow 0 \mbox{ as } |t-s| \longrightarrow 0.
	\]
	To get an intutition why \eqref{equation:alphabeta} might not hold for $\beta = \alpha$, consider the linear transport equation $\partial_t v + \partial_x v = 0$, first on $\mathbb{R}\times[0,\widetilde{T}]$ with initial condition $v(x, 0) = |x|^{\alpha}$ for~$x \in \mathbb{R}$. Then, given any $t > 0$, a lower bound for the $\alpha$-semi-norm of $v(\cdot, t) - v(\cdot, 0)$ follows via
	\[
			\sup\limits_{{x} \neq {y}}  \frac{\left||x-t|^{\alpha} - |y-t|^{\alpha} - |x|^{\alpha} + |y|^{\alpha} \right|}{|x-y|^{\alpha}} \! \geq \! \frac{\left||t-t|^{\alpha} - |(-t)-t|^{\alpha} - |t|^{\alpha} + |(-t)|^{\alpha} \right|}{|t-(-t)|^{\alpha}} = 1.
	\]
	Since $t > 0$ was arbitrary, this example also works for transport equations posed in bounded domains. 
	
	\subsection{Setup}\label{subsection:setup}
	Let the reference control time be $T = 2$ and denote $\xvec{n} \coloneq \xvec{n}_{\mathcal{E}}$, for simplicity. We suppose that $\xvec{u}_0, \xvec{B}_0 \in \xCn{{m,\alpha}}_{*}(\overline{\mathcal{E}}; \mathbb{R}^2)$ meet the hypotheses of \Cref{theorem:local}, including a smallness constraint of the form (\cf~\eqref{equation:smallnesscond})
	\begin{equation*}
		\|\xvec{u}_0\|_{m,\alpha,\mathcal{E}} + \|\xvec{B}_0\|_{m,\alpha,\mathcal{E}} < \delta,
	\end{equation*}
	where $\delta$ shall be determined later (\cf~\eqref{equation:delta}). Further, let $\Lambda_0 \subset \omegaup$ be open, simply-connected, and such that $\mathcal{E}\setminus \Lambda_0$ is simply-connected. Consequently, there exists a smooth curve (smooth cut) $\Sigma \subset \Lambda_0$ rendering the difference $\mathcal{E}\setminus\Sigma$ simply-connected (\cf~\Cref{Figure:extension}).  Then, we take the closure $\Lambda \coloneq \overline{\Lambda}_0$ and define a rough lower bound for the \enquote{thickness} of $\Lambda$ by
	\begin{equation}\label{equation:dl}
		d_{\Lambda} \coloneq \operatorname{dist}(\Sigma, \partial \Lambda\setminus\partial\mathcal{E}).
	\end{equation}
	Without loss of generality, up to adjusting the size of $\Lambda_0$, it is assumed that
	\begin{equation}\label{equation:lmbddist}
		\operatorname{dist}(\Lambda, \overline{\mathcal{E}}\setminus \omegaup) > 4d_{\Lambda}.
	\end{equation}
	\tikzset{>={Latex[width=1.5mm,length=1.5mm]}}
	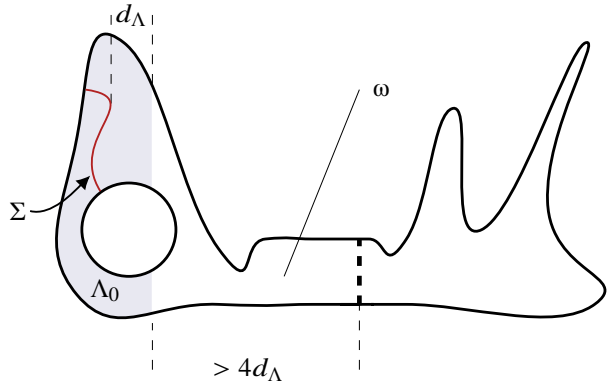
\begin{figure}[ht!]
		\centering
		\resizebox{0.65\textwidth}{!}{
			\begin{tikzpicture}
				\clip(-3.35,-0.69) rectangle (3.3,3.5);

				\draw[line width=0.3mm, color=MidnightBlue!10, fill=MidnightBlue!10] plot[smooth, tension=1] coordinates { (0.5,0.2) (0,0.2) (-1,0.2) (-2.5,0.3) (-2.5,2) (-2,3) (-1,0.8) (-0.5,0.9) (0,0.9) (0.3,0.9) (0.6,0.9) (1,0.9) (1.5,2.3) (1.8,1) (2.9,3) (2.5,1.3)  (3,0.2) (1,0.2) (0.6,0.2) (0.3,0.2)};

				\filldraw[color=white, fill=white] (-1.7,0) rectangle (6.8,4);
				
				\draw[line width=0.3mm, color=black] plot[smooth, tension=1] coordinates { (0.5,0.2) (0,0.2) (-1,0.2) (-2.5,0.3) (-2.5,2) (-2,3) (-1,0.8) (-0.5,0.9) (0,0.9) (0.3,0.9) (0.6,0.9) (1,0.9) (1.5,2.3) (1.8,1) (2.9,3) (2.5,1.3)  (3,0.2) (1,0.2) (0.6,0.2) (0.3,0.2)};

				\filldraw[color=black, fill=white, line width=0.3mm,  fill opacity=1] (-1.95,1) circle (0.5);

				\draw[line width=0.1mm, dashed]  (-2.139,3.3) -- (-2.139,2.35);
				\draw[line width=0.1mm, dashed]  (-1.7,3.3) -- (-1.7,2.55);

				\draw[line width=0.2mm, color=FireBrick] (-2.41,2.5)  to [out=-8,in=-238,looseness=1.6] (-2.253,1.412);

				\coordinate[label=left:\scriptsize$\Sigma$] (A) at (-2.9,1.2);
				\draw[line width=0.2mm, ->] (-3,1.2)  to [out=-15,in=-135,looseness=0.8] (-2.36,1.58);
				\coordinate[label=left:\scriptsize{$\Lambda_0$}] (B) at (-1.9,0.35);
				
				\draw[dashed,line width=0.5mm, color=black] (0.5,0.19)--(0.5,0.9);

				\coordinate[label=right:\scriptsize{$\mathcal{\omegaup}$}] (C) at (0.5,2.5);
				\draw[line width=0.1mm] (-0.3,0.5) -- (C);

				\coordinate[label=right:\scriptsize{$d_{\Lambda}$}] (E) at (-2.23,3.3);

				\draw[line width=0.1mm, dashed]  (-1.7,-0.5) -- (-1.7,0);
				\draw[line width=0.1mm, dashed]  (0.5,-0.5) -- (0.5,0.19);
				\coordinate[label=right:\scriptsize{$> 4d_{\Lambda}$}] (E) at (-1.2,-0.5);
			\end{tikzpicture}
		}
		\caption{The (blue) filled region indicates a choice of $\Lambda_0$ and the smooth curve $\Sigma \subset \Lambda_0$, represented by a (red) line, is selected such that $\mathcal{E}\setminus\Sigma$ is simply-connected.}
		\label{Figure:extension}
	\end{figure}

	To demonstrate \Cref{theorem:local}, we consider the interior-controlled ideal MHD system with initial and target conditions
	\begin{equation}\label{equation:MHD00SimplyConnectedExtended}
		\begin{cases}
			\partial_t \xvec{u} + (\xvec{u} \cdot \xsym{\nabla}) \xvec{u} - (\xvec{B} \cdot \xsym{\nabla})\xvec{B} + \xsym{\nabla} p = \xsym{\xi} & \mbox{ in } \mathcal{E} \times (0,2),\\
			\partial_t \xvec{B} + (\xvec{u} \cdot \xsym{\nabla}) \xvec{B} - (\xvec{B} \cdot \xsym{\nabla}) \xvec{u} = \xsym{\eta} & \mbox{ in } \mathcal{E} \times (0,2),\\
			\xdiv{\xvec{u}} = \xdiv{\xvec{B}} = 0  & \mbox{ in } \mathcal{E} \times (0,2),\\
			\xvec{u} \cdot \xvec{n} = \xvec{B} \cdot \xvec{n} = 0  & \mbox{ on } \partial\mathcal{E} \times(0,2),\\
			\xvec{u}(\cdot, 0)  =  \xvec{u}_0,\, \xvec{B}(\cdot, 0)  =  \xvec{B}_0  & \mbox{ in } \mathcal{E},\\
			\xvec{u}(\cdot, 2)  =  \xsym{0},\, \xvec{B}(\cdot, 2) = \xsym{0}  & \mbox{ in } \mathcal{E},
		\end{cases}
	\end{equation}
	where $\xsym{\xi}$ and $\xsym{\eta}$ shall be obtained so that
	\begin{equation}\label{equation:propeta}
		\begin{gathered}
			\xdiv{\xsym{\eta}} = 0 \mbox{ in } \mathcal{E}, \quad \xsym{\eta}\cdot \xvec{n} = 0 \mbox{ on } \partial \mathcal{E},\\
			\operatorname{supp}(\xsym{\xi}) \subset \omegaup\times[0,2], \quad 	\operatorname{supp}(\xsym{\eta}) \subset \omegaup \times (0,1).
		\end{gathered}
	\end{equation}
	In addition to \eqref{equation:propeta}, the subsequent analysis leads to a choice of $\xsym{\eta}$ with vanishing first cohomology projection. Expressly, we will have the representation
	\begin{equation}\label{equation:propeta2}
		\xsym{\eta} = \xnab^{\perp} \phi = \begin{bmatrix}
			\partial_2 \phi \\ -\partial_1 \phi
		\end{bmatrix} \, \mbox{ in } \mathcal{E}, \quad \phi = 0 \mbox{ on } \partial \mathcal{E}.
	\end{equation}

    As visible from \eqref{equation:propeta}, the control~$\xsym{\eta}$ is scheduled  during the first half of the reference time interval~$[0,2]$. The reason is that~$\xvec{B}$ will already be annihilated until~$t=1$, with the help of~$\xsym{\xi}$ and~$\xsym{\eta}$. For~$t\in[1,2]$, it remains to regulate a velocity-controlled incompressible Euler problem by means of~$\xsym{\xi}$.

	\paragraph{A conservation property.}
	Let us recall the notation $\xwcurl{\xvec{f}} = \partial_1 f_2 - \partial_2 f_1$. Since the domain $\mathcal{E}$ is doubly-connected and sufficiently smooth, there exists a nonzero solution~$\xvec{Q} \in \xCinfty(\overline{\mathcal{E}}; \mathbb{R}^2)$ to the div-curl system
	\[
		\xwcurl{\xvec{Q}} = 0 \mbox{ in } \mathcal{E}, \quad \xdiv{\xvec{Q}} = 0 \mbox{ in } \mathcal{E}, \quad \xvec{Q} \cdot \xvec{n} = 0 \mbox{ on } \partial \mathcal{E}.
	\]
	In fact, one can write $\xvec{Q} = \xnab^{\perp} \psi$ for a harmonic function $\psi \in \xCinfty(\overline{\mathcal{E}}; \mathbb{R})$ which is piecewise constant along $\partial \mathcal{E}$ (\cf~\cite[Pages 15-17]{MarchioroPulvirenti1994} and \cite[Chapter IX]{DautrayLions1990}).
	Furthermore, the first cohomology space
	\begin{equation}\label{equation:Z}
			\xZ(\mathcal{E}) \coloneq \left\{ \xvec{f} \in \xLtwo(\mathcal{E};\mathbb{R}^2) \, \big| \, \xdiv{\xvec{f}} = 0 \mbox{ in } \mathcal{E}, \, \xwcurl{\xvec{f}} = 0 \mbox{ in } \mathcal{E}, \, \xvec{f} \cdot \xvec{n} = 0 \mbox{ on } \partial \mathcal{E} \right\}
	\end{equation}
	is one-dimensional, thus $\xZ(\mathcal{E}) = \operatorname{span}\{\xvec{Q}\}$ (\cf~\cite[Chapter IX]{DautrayLions1990}). Let us anticipate the properties of $\xsym{\eta}$, as proclaimed in~\eqref{equation:propeta} and \eqref{equation:propeta2}, to be true. Multiplying the induction equation of~\eqref{equation:MHD00SimplyConnectedExtended} with~$\xvec{Q} = \xnab^{\perp} \psi \in \xZ(\mathcal{E})$ and integrating by parts, one discovers
	\begin{multline}\label{equation:consv}
		\xdrv{}{t} \int_{\mathcal{E}} \xvec{B}(\xvec{x},t) \cdot \xvec{Q}(\xvec{x}) \, \xdx{\xvec{x}} \\
		\begin{aligned}
			& = \int_{\mathcal{E}} \left(\xsym{\eta} - (\xvec{u} \cdot \xsym{\nabla}) \xvec{B} + (\xvec{B} \cdot \xsym{\nabla}) \xvec{u} \right)(\xvec{x},t) \cdot \xvec{Q}(\xvec{x}) \, \xdx{\xvec{x}}\\
			& = \int_{\mathcal{E}} \xnab^{\perp} \left( \xvec{u} \wedge \xvec{B} \right)(\xvec{x},t) \cdot \xnab^{\perp} \psi (\xvec{x}) \, \xdx{\xvec{x}} + \int_{\mathcal{E}} \xnab^{\perp}\phi(\xvec{x},t) \cdot \xnab^{\perp} \psi (\xvec{x}) \, \xdx{\xvec{x}}\\
			& = \int_{\mathcal{E}} \xnab \left( \xvec{u} \wedge \xvec{B} \right)(\xvec{x},t) \cdot \xnab \psi (\xvec{x}) \, \xdx{\xvec{x}} +  \int_{\mathcal{E}} \xnab\phi(\xvec{x},t) \cdot \xnab \psi (\xvec{x}) \, \xdx{\xvec{x}}\\
			& = 0.
		\end{aligned}
	\end{multline}
	In other words, the system \eqref{equation:MHD00SimplyConnectedExtended} conserves the $\xLtwo(\mathcal{E};\mathbb{R}^2)$-projection of the magnetic field to the subspace $\xZ(\mathcal{E})$.

	\subsection{Convection strategy}\label{subsection:flushingprofile}
	We select a vector field $\xvec{y}^*$ along which information propagates through the control zone $\omegaup$. Profiles of this type are known for general multiply-connected domains \cite{Coron1996EulerEq,Glass2000}. However, in order to maintain \eqref{equation:intrprctrl}, here~$\xvec{y}^*$ is required curl-free, divergence-free, and tangential throughout $\mathcal{E}$. Moreover, for the sake of managing the magnetic field's vanishing set when transported by the fluid, a maximal dragging distance is embedded in the definition of $\xvec{y}^*$.

	\begin{lmm}\label{lemma:flushing}
		There exist a profile $\xvec{y}^* \in \xCinfty(\overline{\mathcal{E}}\times[0,1];\mathbb{R}^{2})$ with $\xwcurl{\xvec{y}^*} = 0$, a pressure function $p^* \in \xCinfty(\overline{\mathcal{E}}\times[0,1];\mathbb{R})$, a number $K \in \mathbb{N}$, and $\xsym{\xi}^* \in \xCinfty(\overline{\mathcal{E}}\times[0,1];\mathbb{R}^{2})$, satisfying all of the following properties.
		\begin{itemize}
			\item $\xvec{y}^*$, $p^*$, and $\xsym{\xi}^*$ are supported in $\bigcup_{j=1}^K(a_j, b_j)$ with respect to time, where
			\[
			\quad a_j \coloneq \frac{j-1}{K}, \quad b_j = \frac{j}{K}, \quad j \in \{1,\dots,K\}.
			\] 
			\item $\xvec{y}^*$, $p^*$, and $\xsym{\xi}^*$ are time-periodic with the period $1/K$.
			\item $\xsym{\xi}^*$ is supported in $\omegaup$ with respect to the space variables.
			\item $(\xvec{y}^*, p^*, \xsym{\xi}^*)$ solve the controlled incompressible Euler problem
			\begin{equation*}\label{equation:euler_ctrl_ret}
				\begin{cases}
					\partial_t \xvec{y}^* + (\xvec{y}^* \cdot \xdop{\nabla}) \xvec{y}^* + \xdop{\nabla} p^* = \xsym{\xi}^* & \mbox{ in } \mathcal{E} \times (0,1),\\
					\xdop{\nabla}\cdot\xvec{y}^* = 0 & \mbox{ in } \mathcal{E} \times (0,1),\\
					\xvec{y}^* \cdot \xvec{n} = 0 & \mbox{ on } \partial \mathcal{E} \times (0,1),\\
					\xvec{y}^*(\cdot, 0) = \xvec{y}^*(\cdot, 1) = \xvec{0} & \mbox{ in } \mathcal{E}.
				\end{cases}
			\end{equation*}
			\item Whenever $\xvec{y}^*(\cdot,t) \neq \xsym{0}$, the integral curves of the vector field $\xvec{y}^*(\cdot,t)$ are closed, encompass the cavity of $\mathcal{E}$, and have the same orientation.
			\item The flow map $\xmcal{Y}^*$ determined via
			\begin{equation}\label{equation:flowofy}
				\begin{gathered}
					\xdrv{}{t} \xmcal{Y}^*(\xvec{x},s,t) = \xvec{y}^*(\xmcal{Y}^*(\xvec{x},s,t),t), \quad
					\xmcal{Y}^*(\xvec{x},s,s) = \xvec{x}
				\end{gathered}
			\end{equation}
			obeys the flushing property
			\begin{equation}\label{equation:flushingproperty}
				\forall \xvec{x} \in \overline{\mathcal{E}}, \, \exists t_{\xvec{x}} \in (0,1) \colon  \xmcal{Y}^{*}(\xvec{x}, 0, t_{\xvec{x}}) \in \Sigma
			\end{equation}
			and adheres to the maximal dragging distance
			\begin{equation}\label{equation:draggingproperty}
				\forall\xvec{x} \in \overline{\mathcal{E}},\, \forall j \in \{1,\dots,K\} \colon \max\limits_{s,t \in [a_j, b_j]}\operatorname{dist}(\xvec{x}, \xmcal{Y}^{*}(\xvec{x}, s, t)) < \frac{d_{\Lambda}}{2}.
			\end{equation}
		\end{itemize}
	\end{lmm}

	\tikzset{>={Latex[width=3.3mm,length=3.3mm]}}
	
	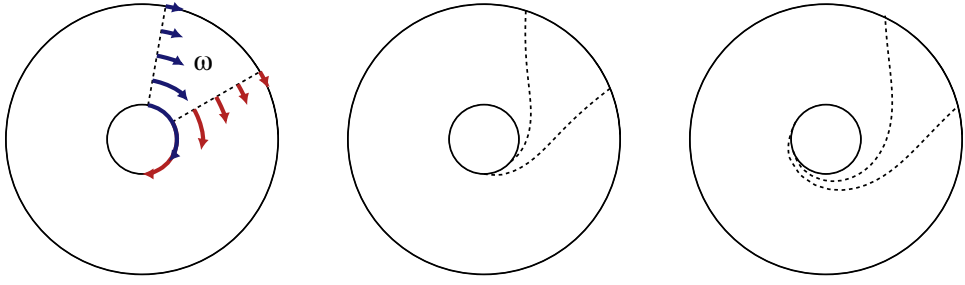
\begin{figure}[ht!]
		\centering
		\resizebox{1\textwidth}{!}{
			\begin{subfigure}[b]{0.5\textwidth}
				\centering
				\resizebox{1\textwidth}{!}{
					\begin{tikzpicture}
						\clip(-4.2,-4.2) rectangle (4.2,4.2);
						\draw[line width=0.1mm, color=black] (80:1) coordinate (A) arc (80:30:1) coordinate (C)  (30:3.9) coordinate (D) arc (30:80:3.9) coordinate (B);
						\draw[line width=0.1mm, color=black] (80:1) arc (80:395:1) (2:3.9) arc (2:395:3.9);
						
						\draw [line width=0.5mm, color=black] (0,0) circle (3.9);
						\draw [line width=0.5mm, color=black, fill=white] (0,0) circle (1);
						
						\draw[dashed,line width=0.5mm, color=black] (A) -- (B);
						\draw[dashed,line width=0.5mm, color=black] (C) -- (D);
						
						\draw[line width=1.2mm, color=FireBrick, ->] (30:1) arc  (30:30-120:1)  coordinate (aa);
						\draw[line width=1.2mm, color=FireBrick, ->] (30:1.72) arc  (30:{30-(120/((1.72)*(1.72)))}:1.72)  coordinate (bb);
						\draw[line width=1.2mm, color=FireBrick, ->] (30:2.45) arc (30:{30-(120/((2.45)*(2.45)))}:2.45) coordinate (cc);
						\draw[line width=1.2mm, color=FireBrick, ->] (30:3.17) arc (30:{30-(120/((3.17)*(3.17)))}:3.17) coordinate (dd);
						\draw[line width=1.2mm, color=FireBrick, ->] (30:3.9)  arc (30:{30-(120/((3.9)^2)}:3.9) coordinate (ee);
						
						\draw[line width=1.2mm, color=MidnightBlue, ->] (80:1) arc  (80:80-120:1)  coordinate (a);
						\draw[line width=1.2mm, color=MidnightBlue, ->] (80:1.72) arc  (80:{80-(120/((1.72)*(1.72)))}:1.72)  coordinate (b);
						\draw[line width=1.2mm, color=MidnightBlue, ->] (80:2.45) arc (80:{80-(120/((2.45)*(2.45)))}:2.45) coordinate (c);
						\draw[line width=1.2mm, color=MidnightBlue, ->] (80:3.17) arc (80:{80-(120/((3.17)*(3.17)))}:3.17) coordinate (d);
						\draw[line width=1.2mm, color=MidnightBlue, ->] (80:3.9)  arc (80:{80-(120/((3.9)^2)}:3.9) coordinate (e);

						\coordinate[label=left:\Huge{$\omegaup$}] (A) at (2.2,2.2);
					\end{tikzpicture}
				}
				\label{Figure:ff1}
			\end{subfigure}
			\quad \quad 
			\begin{subfigure}[b]{0.5\textwidth}
				\centering
				\resizebox{1\textwidth}{!}{
					\begin{tikzpicture}
						\clip(-4.2,-4.2) rectangle (4.2,4.2);
						\draw[line width=0.1mm, color=black] (50:1) coordinate (A) arc (50:30:1) coordinate (C)  (30:3.9) coordinate (D) arc (30:50:3.9) coordinate (B);
						\draw[line width=0.1mm, color=black] (50:1) arc (50:395:1) (2:3.9) arc (2:395:3.9);
						
						\draw [line width=0.5mm, color=black] (0,0) circle (3.9);
						\draw [line width=0.5mm, color=black, fill=white] (0,0) circle (1);

						\newcounter{coordinateindex}
						\setcounter{coordinateindex}{0}
						\tikzset{
							stepcounter/.code={
								\stepcounter{coordinateindex}
							}
						}
						\tikzset{arclen/.style={decoration={markings, mark=at position #1 	with{
										\coordinate[stepcounter] 	(coordinate\thecoordinateindex);
								}},
								postaction=decorate}}
						
						\foreach \i in {1,2,...,60}
						\coordinate[stepcounter] (coordinate\thecoordinateindex) at ({80-(120/(1+\i/20-1/20)^2)}:1+\i/20-1/20);
						
						\newcounter{coordinateindexb}
						\setcounter{coordinateindexb}{0}
						\tikzset{
							stepcounterb/.code={
								\stepcounter{coordinateindexb}
							}
						}

						\foreach \i in {1,2,...,30}
						\coordinate[stepcounterb] (coordinateb\thecoordinateindexb) at ({30-(120/(1+\i/10-1/10)^2)}:1+\i/10-1/10);

						\draw [dashed,line width=0.5mm, color=black] plot [smooth, 	tension=1] coordinates { (coordinate1) (coordinate2) (coordinate3) (coordinate4) (coordinate5) (coordinate6) (coordinate7) (coordinate8) (coordinate9) (coordinate10) (coordinate11) (coordinate12) (coordinate13)  (coordinate14) (coordinate15) (coordinate16) (coordinate17) (coordinate18) (coordinate19) (coordinate20) (coordinate21) (coordinate22) (coordinate23) (coordinate24) (coordinate25) (coordinate26) (coordinate27) (coordinate28) (coordinate29) (coordinate30) (coordinate31) (coordinate32) (coordinate33) (coordinate34) (coordinate35) (coordinate36) (coordinate37) (coordinate38) (coordinate39) (coordinate40) (coordinate41) (coordinate42) (coordinate43)  (coordinate44) (coordinate45) (coordinate46) (coordinate47) (coordinate48) (coordinate49) (coordinate50) (coordinate51) (coordinate52) (coordinate53) (coordinate54) (coordinate55) (coordinate56) (coordinate57) (coordinate58) (coordinate59) (coordinate60)};
						
						\draw [dashed,line width=0.5mm, color=black] plot [smooth, 	tension=1] coordinates { (coordinateb1) (coordinateb2) (coordinateb3) (coordinateb4) (coordinateb5) (coordinateb6) (coordinateb7) (coordinateb8) (coordinateb9) (coordinateb10) (coordinateb11) (coordinateb12) (coordinateb13)  (coordinateb14) (coordinateb15) (coordinateb16) (coordinateb17) (coordinateb18) (coordinateb19) (coordinateb20) (coordinateb21) (coordinateb22) (coordinateb23) (coordinateb24) (coordinateb25) (coordinateb26) (coordinateb27) (coordinateb28) (coordinateb29) (coordinateb30)};
					\end{tikzpicture}
				}
				\label{Figure:ff2}
		\end{subfigure}
		\quad \quad 
		\begin{subfigure}[b]{0.5\textwidth}
			\centering
			\resizebox{1\textwidth}{!}{
				\begin{tikzpicture}
					\clip(-4.2,-4.2) rectangle (4.2,4.2);
					\draw[line width=0.1mm, color=black] (50:1) coordinate (A) arc (50:30:1) coordinate (C)  (30:3.9) coordinate (D) arc (30:50:3.9) coordinate (B);
					\draw[line width=0.1mm, color=black] (50:1) arc (50:395:1) (2:3.9) arc (2:395:3.9);
					
					\draw [line width=0.5mm, color=black] (0,0) circle (3.9);
					\draw [line width=0.5mm, color=black, fill=white] (0,0) circle (1);

					\newcounter{coordinateindexc}
					\setcounter{coordinateindexc}{0}
					\tikzset{
						stepcounterc/.code={
							\stepcounter{coordinateindexc}
						}
					}
					
					\foreach \i in {1,2,...,60}
					\coordinate[stepcounterc] (coordinatec\thecoordinateindexc) at ({80-(240/(1+\i/20-1/20)^2)}:1+\i/20-1/20);
					
					\newcounter{coordinateindexd}
					\setcounter{coordinateindexd}{0}
					\tikzset{
						stepcounterd/.code={
							\stepcounter{coordinateindexd}
						}
					}
					\tikzset{arclend/.style={decoration={markings, mark=at position #1 	with{
									\coordinate[stepcounterd] 	(coordinated\thecoordinateindexd);
							}},
							postaction=decorate}}

					\foreach \i in {1,2,...,30}
					\coordinate[stepcounterd] (coordinated\thecoordinateindexd) at ({30-(240/(1+\i/10-1/10)^2)}:1+\i/10-1/10);
					
					\draw [dashed,line width=0.5mm, color=black] plot [smooth, 	tension=1] coordinates { (coordinatec1) (coordinatec2) (coordinatec3) (coordinatec4) (coordinatec5) (coordinatec6) (coordinatec7) (coordinatec8) (coordinatec9) (coordinatec10) (coordinatec11) (coordinatec12) (coordinatec13)  (coordinatec14) (coordinatec15) (coordinatec16) (coordinatec17) (coordinatec18) (coordinatec19) (coordinatec20) (coordinatec21) (coordinatec22) (coordinatec23) (coordinatec24) (coordinatec25) (coordinatec26) (coordinatec27) (coordinatec28) (coordinatec29) (coordinatec30) (coordinatec31) (coordinatec32) (coordinatec33) (coordinatec34) (coordinatec35) (coordinatec36) (coordinatec37) (coordinatec38) (coordinatec39) (coordinatec40) (coordinatec41) (coordinatec42) (coordinatec43)  (coordinatec44) (coordinatec45) (coordinatec46) (coordinatec47) (coordinatec48) (coordinatec49) (coordinatec50) (coordinatec51) (coordinatec52) (coordinatec53) (coordinatec54) (coordinatec55) (coordinatec56) (coordinatec57) (coordinatec58) (coordinatec59) (coordinatec60)};
					
					\draw [dashed,line width=0.5mm, color=black] plot [smooth, 	tension=1] coordinates { (coordinated1) (coordinated2) (coordinated3) (coordinated4) (coordinated5) (coordinated6) (coordinated7) (coordinated8) (coordinated9) (coordinated10) (coordinated11) (coordinated12) (coordinated13)  (coordinated14) (coordinated15) (coordinated16) (coordinated17) (coordinated18) (coordinated19) (coordinated20) (coordinated21) (coordinated22) (coordinated23) (coordinated24) (coordinated25) (coordinated26) (coordinated27) (coordinated28) (coordinated29) (coordinated30)};
				\end{tikzpicture}
			}
			\label{Figure:ff3}
		\end{subfigure}
		}
		\caption{For illustrative purposes, let $\mathcal{E}$ be an annulus. The (blue and red) arrows indicate how information is transported during a time interval $(a,b) \subset (0,1)$ by the flow $\xmcal{Y}^*$ associated with $\xvec{y}^*$. In the middle and right annuli, dashed lines visualize the transformation of $\omegaup$ under $\xmcal{Y}^*$ during the time intervals~$(a,b)$ and $(a, 2b-a)$, respectively. }
		\label{Figure:flushing0}
	\end{figure}

	\begin{proof}[Proof of \Cref{lemma:flushing}]
		First, the functions $(\xvec{y}^*, p^* ,\xsym{\xi}^*)$ are selected resorting to the space $\xZ(\mathcal{E})$ given in \eqref{equation:Z}. Second, by adjusting some parameters, the flushing property and the maximal dragging distance are ensured. 
		
		\paragraph{Step 1. Constructions.}
		Let $\xvec{y} \in \xZ(\mathcal{E})\setminus\{\xsym{0}\}$ be fixed. As shown in \Cref{Figure:annulusexmpl}, denote by $\Gamma^0$ and~$\Gamma^1$ the boundaries of the bounded and unbounded components of~$\mathbb{R}^2\setminus\mathcal{E}$, respectively. Then, one has $\xvec{y} = \xnab^{\perp}q$ for a harmonic function $q \in \xCinfty(\overline{\mathcal{E}};\mathbb{R})$ solving the Dirichlet problem (\cf~\cite{MarchioroPulvirenti1994,DautrayLions1990})
		\begin{equation}\label{equation:dpq}
			\Delta q = 0 \mbox{ in } \mathcal{E}, \quad q = 0 \mbox{ on } \Gamma^0, \quad q = a \mbox{ on } \Gamma^1
		\end{equation} 
		with a constant $a \in \mathbb{R}\setminus\{0\}$.
		Since $\mathcal{E}$ is doubly-connected, it is known (\cf~\cite{AlessandriniMagnanini1992}) that the harmonic function $q$ cannot have critical points, hence $\xvec{y} \neq \xvec{0}$ in $\overline{\mathcal{E}}$.  
		Now, given~$M > 0$ and $K \in \mathbb{N}$, whose values will be decided later on, we take any~$1/K$-periodic smooth profile
		\begin{equation}\label{equation:lambda}
			\begin{gathered}
				 \lambda = \lambda_{K,M} \colon [0,1] \longrightarrow [0,+\infty), \quad \operatorname{supp}(\lambda) \subset \bigcup_{j=1}^K(a_j, b_j), \\ \forall j\in\{1,\dots,K\} \colon\int_{a_j}^{b_j} \lambda(s) \, \xdx{s} = M.
			\end{gathered}
		\end{equation}
		Next, we choose $r \in \xCinfty(\overline{\mathcal{E}};\mathbb{R}^{2})$ with $\xvec{y} = \xnab r$ in the simply-connected region $\mathcal{E}\setminus \Lambda$. Eventually, 
		\[
			\xvec{y}^* \coloneq \lambda \xvec{y}, \quad p^* \coloneq - r \xdrv{\lambda}{t}  - \frac{1}{2}|\xvec{y}^*|^2, \quad \xsym{\xi}^* \coloneq \partial_t \xvec{y}^* + (\xvec{y}^* \cdot \xnab) \xvec{y}^* + \xnab p^*.
		\]
		\paragraph{Step 2. Flushing property and dragging distance.}
		As $\mathcal{E}_0\coloneq \mathcal{E}\setminus\Lambda_0$ is simply-connected and $\xvec{y}$ constitutes a gradient field in~$\mathcal{E}_0$, the integral curves of~$\xvec{y}$ restricted to~$\overline{\mathcal{E}}_0$ cannot be closed. Further, due to $\xvec{y} \neq\xsym{0}$, the associated autonomous system has no stationary points. Accordingly, all integral curves of~$\xvec{y}$ are of identical orientation and pass through the smooth cut $\Sigma$ contained in~$\Lambda$. Consequently, one can pick the numbers $M > 0$ and $K \in \mathbb{N}$ in the definition of~$\xvec{y}^*$ so that \eqref{equation:flushingproperty} and~\eqref{equation:draggingproperty} hold. Hereto, we first set $M > 0$ sufficiently small so that the solution to 
		\begin{equation*}
			\begin{gathered}
				\xdrv{}{t} \xmcal{Y}_{M}(\xvec{x},s,t) = (\widetilde{\lambda}\xvec{y})(\xmcal{Y}_{M}(\xvec{x},s,t),t), \quad
				\xmcal{Y}_{M}(\xvec{x},s,s) = \xvec{x},
			\end{gathered}
		\end{equation*}
		satisfies
		\begin{equation}\label{equation:auxmdd}
				\forall\xvec{x} \in \overline{\mathcal{E}} \colon \max\limits_{s,t \in [0, 1]}\operatorname{dist}(\xvec{x}, \xmcal{Y}_{M}(\xvec{x}, s, t)) < \frac{d_{\Lambda}}{2},
		\end{equation}
		where the smooth function $\widetilde{\lambda}$ is chosen with
		\[
			\widetilde{\lambda} \colon [0,1] \longrightarrow [0,+\infty), \quad \operatorname{supp}(\widetilde{\lambda}) \subset (0,1), \quad \int_{0}^{1} \widetilde{\lambda}(r) \, \xdx{r} = M.
		\]
		Hereby, we recall that $\xvec{y}$ is tangential at $\partial \mathcal{E}$; the flow $\xmcal{Y}_{M}$ never attempts sending information to $\mathbb{R}^2\setminus\overline{\mathcal{E}}$. The existence of a good $M > 0$ follows (after integrating the equation for $\xmcal{Y}_{M}$) from the inequality
		\begin{equation*}
			\begin{aligned}
				|\xvec{x} - \xmcal{Y}_{M}(\xvec{x}, s, t)| & \leq \int_s^t \left| \widetilde{\lambda}(r) \xvec{y}(\xmcal{Y}_M(\xvec{x},s,r),r) \right| \, \xdx{r} \\
				& \leq \|\xvec{y}\|_{\xCzero(\overline{\mathcal{E}};\mathbb{R}^2)}\int_s^t \widetilde{\lambda}(r) \, \xdx{r}.
			\end{aligned}
		\end{equation*}
		The value of $M$ is now set in stone, but one can still play with $K$ when selecting the parameter~$\lambda = \lambda_{K,M}$ of the type~\eqref{equation:lambda}. Indeed, we have
		\begin{itemize}
			\item the maximal dragging distance \eqref{equation:auxmdd},
			\item $\xvec{y}(\xvec{x}) \neq \xsym{0}$ for all $\xvec{x} \in \overline{\mathcal{E}}$,
			\item all integral curves of $\xvec{y}$ are closed and identically oriented.
		\end{itemize} 
		Thus, by the compactness of $\overline{\mathcal{E}}$ and the smoothness of the flow $\xmcal{Y}_M$, there is a number~$d^*$ such that
		\begin{equation}\label{equation:dstar}
			\forall \xvec{x} \in \mathcal{E}\colon \operatorname{dist}(\xvec{x}, \xmcal{Y}_M(\xvec{x},0,1)) \in (d^*, d_{\Lambda}/2).
		\end{equation}
		As a result, there exists a possibly large $K \in \mathbb{N}$ for which $\xvec{y}^* = \lambda _{K,M}\xvec{y}$ obeys \eqref{equation:flushingproperty}.
	\end{proof} 
	
	Henceforth, the return method profile~$\xvec{y}^*$ and the number $K \in \mathbb{N}$ are fixed via \Cref{lemma:flushing}. 
	Notably, the properties \eqref{equation:flushingproperty} and \eqref{equation:draggingproperty} are stable under small perturbations of $\xvec{y}^*$, as shown next (\cf~\cite[Lemma 7]{Fernandez-CaraSantosSouza2016} and \cite{Coron1996EulerEq,Glass2000}).
	
	\begin{lmm}\label{lemma:nu}
		Given any number $\nu_0 > 0$, there exists a small constant $\nu > 0$ for which each profile $\xvec{z} \in \xCzero(\overline{\mathcal{E}}\times[0,1];\mathbb{R}^2)$ with
		\begin{equation}\label{equation:propertiesv}
			\xvec{z} \cdot \xvec{n} = 0 \mbox{ on } \partial \mathcal{E}, \quad \|\xvec{z}-\xvec{y}^*\|_{\xCzero(\overline{\mathcal{E}}\times[0,1];\mathbb{R}^2)} < \nu
		\end{equation}
		satisfies 
		\begin{equation}\label{equation:flushingpropertyperturbation}
			\begin{gathered}
				\max_{s,t \in [0,1]}\|\xmcal{Z}(\cdot, s, t) - \xmcal{Y}^*(\cdot, s, t)\|_{\xCzero(\overline{\mathcal{E}};\mathbb{R}^2)} < \nu_0,\\
				\forall \xvec{x} \in \overline{\mathcal{E}}, \, \exists t_{\xvec{x}} \in (0,1) \colon \xmcal{Z}(\xvec{x}, 0, t_{\xvec{x}}) \in \Lambda.
			\end{gathered}
		\end{equation}
		and adheres to the maximal dragging distance
		\begin{equation}\label{equation:draggingpropertyperturbation}
			\forall\xvec{x} \in \overline{\mathcal{E}},\, \forall j \in \{1,\dots,K\} \colon \max\limits_{s,t \in [a_j, b_j]}\operatorname{dist}(\xvec{x}, \xmcal{Z}(\xvec{x}, s, t)) < \frac{d_{\Lambda}}{2},
		\end{equation}
		where the flow $\xmcal{Z}$ solves
		\begin{equation}\label{equation:flowofz}
			\xdrv{}{t} \xmcal{Z}(\xvec{x},s,t) = \xvec{z}(\xmcal{Z}(\xvec{x},s,t),t), \quad
			\xmcal{Z}(\xvec{x},s,s) = \xvec{x}.
		\end{equation}
	\end{lmm}
	\begin{proof}
		Since $\xvec{z}$ is tangential at $\partial \mathcal{E}$, the map $\xvec{x} \mapsto \xmcal{Z}(\xvec{x},s,t)$ constitutes for all instances~$s,t\in[0,1]$ a homeomorphism of $\overline{\mathcal{E}}$. Moreover, integrating \eqref{equation:flowofy} and \eqref{equation:flowofz} provides for $\xmcal{U} \coloneq \xmcal{Y}^* - \xmcal{Z}$ the representation
		\begin{equation*}
			\begin{aligned}
				\xmcal{U}(\xvec{x},s,t) & = \int_s^t \left(\xvec{y}^*(\xmcal{Y}^*(\xvec{x},s,r),r) - \xvec{y}^*(\xmcal{Z}(\xvec{x},s,r),r)\right) \, \xdx{r} \\
				& \quad  + \int_s^t \left(\xvec{y}^*(\xmcal{Z}(\xvec{x},s,r),r) - \xvec{z}(\xmcal{Z}(\xvec{x},s,r),r)\right) \, \xdx{r}.
			\end{aligned}
		\end{equation*}
		By the mean value theorem, there is a constant $C > 0$ depending on $\xvec{y}^*$ such that
		\begin{equation*}\label{equation:diffestU}
			\begin{aligned}
				\left|\xmcal{U}(\xvec{x},s,t)\right| & \leq  C\int_s^t \left|\xmcal{U}(\xvec{x},s,r),r)\right|  \, \xdx{r} + |t-s| \|\xvec{y}^* - \xvec{z}\|_{\xCzero(\overline{\mathcal{E}}\times[0,1]; \mathbb{R}^2)}
			\end{aligned}
		\end{equation*}
		Thus, Gr\"onwall's inequality implies
		\begin{equation}\label{equation:gwcest}
			\left|\xmcal{U}(\xvec{x},s,t)\right| \leq |t-s|\|\xvec{y}^* - \xvec{z}\|_{\xCzero(\overline{\mathcal{E}}\times[0,1]; \mathbb{R}^2)} \operatorname{e}^{ |t-s|\|\xvec{y}^{*}\|_{\xCone(\overline{\mathcal{E}}\times[0,1]; \mathbb{R}^2)}}.
		\end{equation}
		Finally, take $d_0 > 0$ with
		\[
			\operatorname{dist}(\xvec{x}, \xmcal{Y}^{*}(\xvec{x}, s, t)) < d_0 < d_{\Lambda}/2, \quad \xvec{x} \in \overline{\mathcal{E}}, \, s,t \in [0,1/K]
		\] 
		and fix $\nu > 0$ in \eqref{equation:propertiesv} so small that \eqref{equation:gwcest} yields
		\[
			\max_{(\xvec{x}, s,t) \in \overline{\mathcal{E}}\times[0,1]\times[0,1]}|\xmcal{U}(\xvec{x},s,t)| < \min\{\nu_0, d_{\Lambda}/2 - d_0 \}.
		\]
	\end{proof}
	
	In ideal MHD, the magnetic field is dragged by the flow. This is stated next in the context of a transport problem.  For the sake of completeness, a proof is sketched (\cf~\cite[Page 9]{Glass2000}, \cite[Page 24]{Fernandez-CaraSantosSouza2016}, and \cite[Appendix]{CoronMarbachSueur2020}).
	\begin{lmm}\label{lemma:mgftrsp}
	For $\widetilde{T}>0$, $\xvec{z} \in \xCzero([0,\widetilde{T}];\xCn{{1,\alpha}}_{*}(\overline{\mathcal{E}}; \mathbb{R}^2))$, $\xsym{g} \in \xLinfty((0,\widetilde{T}); \xCn{{0,\alpha}}(\overline{\mathcal{E}};\mathbb{R}^2))$, and $\xvec{H}_0 \in \xCn{{1,\alpha}}(\overline{\mathcal{E}}; \mathbb{R}^2)$, assume that $\xvec{H} \in \xCzero([0,\widetilde{T}];\xCn{{1,\alpha}}(\overline{\mathcal{E}}; \mathbb{R}^2))$ solves
		\begin{equation*}
			\begin{cases}
				\partial_t \xvec{H} + (\xvec{z} \cdot \xsym{\nabla}) \xvec{H} - (\xvec{H} \cdot \xsym{\nabla}) \xvec{z} = \xsym{g} & \mbox{ in } \mathcal{E} \times (0,\widetilde{T}),\\
				\xvec{H}(\cdot, 0)  =  \xvec{H}_0  & \mbox{ in } \mathcal{E}.
			\end{cases}
		\end{equation*}
		Then, at any time $t \in [0,\widetilde{T}]$, it holds
		\[
			\operatorname{supp}(\xvec{H}(\cdot,t)) \subset \xmcal{Z}(\operatorname{supp}(\xvec{H}_0), 0, t) \cup \bigcup_{s\in[0,t]}\xmcal{Z}(\operatorname{supp}(\xsym{g}(\cdot,s)), s, t),
		\]
		where $\xmcal{Z}$ denotes the flow of $\xvec{z}$ obtained via \eqref{equation:flowofz}.
	\end{lmm}
	\begin{proof}
		For simplicity, let $\xvec{g} = \xsym{0}$; the general case follows from a representation like~\eqref{equation:moc}. Now, given $\xvec{x} \notin \operatorname{supp}(\xvec{H}_0)$ and recalling that $\xvec{z}$ is tangential at $\partial \mathcal{E}$, 
		\begin{equation*}
			\begin{aligned}
				|\xvec{H}(\xmcal{Z}(\xvec{x}, 0, t), t)| & \leq |\xvec{H}_0(\xvec{x})| + \int_0^t |(\xvec{H} \cdot \xsym{\nabla}) \xvec{z}|(\xmcal{Z}(\xvec{x}, 0, s), s) \, \xdx{s} \\
				& \leq \|\xvec{z}\|_{\xCzero([0,\widetilde{T}];\xCn{{1,\alpha}}(\overline{\mathcal{E}}; \mathbb{R}^2))}\int_0^t |\xvec{H}|(\xmcal{Z}(\xvec{x}, 0, s), s) \, \xdx{s}.
			\end{aligned}
		\end{equation*}
		Gr\"onwall's inequality implies
		\[
			\xvec{H}(\xmcal{Z}(\xvec{x}, 0, t), t) = \xsym{0}, \quad t \in [0,\widetilde{T}].
		\]
	\end{proof}

	\subsection{Local exact null controllability (conclusion of \Cref{theorem:local})}\label{subsection:nonlineardecomposition}
	In this section, $\Lambda$, $d_{\Lambda}$, $\xvec{y}^*$, $K$, $a_j$, and $b_j$ are those fixed in Sections~\Rref{subsection:setup} and \Rref{subsection:flushingprofile}. The controls $(\xsym{\xi}, \xsym{\eta})$ in \eqref{equation:MHD00SimplyConnectedExtended} are now constructed through a nonlinear decomposition that facilitates an iterative annihilation of the magnetic field. More precisely, we shall describe profiles $(\widetilde{\xvec{u}}^j, \widetilde{\xvec{B}}^j, \widetilde{p}^j, \widetilde{\xsym{\xi}}^j, \widetilde{\xsym{\eta}}^j)_{j \in \{1,\dots, K\}}$ and $(\xvec{V}, P, \xsym{\Xi})$ such that the glued trajectory
	\begin{equation}\label{equation:preprf0}
		\begin{gathered}
			(\xvec{u}, \xvec{B}, p, \xsym{\xi}, \xsym{\eta})(\cdot, t) \coloneq \begin{cases}
			(\widetilde{\xvec{u}}^1, \widetilde{\xvec{B}}^1, \widetilde{p}^1, \widetilde{\xsym{\xi}}^1, \widetilde{\xsym{\eta}}^1)(\cdot,t) & \mbox{ if } t \in [0, b_1),\\
			(\widetilde{\xvec{u}}^j, \widetilde{\xvec{B}}^j, \widetilde{p}^j, \widetilde{\xsym{\xi}}^j, \widetilde{\xsym{\eta}}^j)(\cdot,t-a_j) & \mbox{ if } t \in [a_j, b_j),\\
			(\widetilde{\xvec{u}}^K, \widetilde{\xvec{B}}^K, \widetilde{p}^K, \widetilde{\xsym{\xi}}^K, \widetilde{\xsym{\eta}}^K)(\cdot,t-a_K) & \mbox{ if } t \in [a_K, 1),\\
			(\xvec{V}, \xsym{0}, P, \xsym{\Xi}, \xsym{0})(\cdot,t-1) & \mbox{ if } t \in [1, 2],
			\end{cases} 
		\end{gathered}
	\end{equation}
	with 
	\[
		a_j \coloneq \frac{j-1}{K}, \quad b_j \coloneq \frac{j}{K}, \quad j \in \{1,\dots,K\},
	\]
	constitutes a solution to \eqref{equation:MHD00SimplyConnectedExtended} driven by controls of the type \eqref{equation:propeta}. 
	The building blocks in~\eqref{equation:preprf0} are given by Theorems~\Rref{theorem:ret} and~\Rref{theorem:tkloc} below; after that, the proof of \Cref{theorem:local} is completed in \Cref{subsubsection:piecingtogether}.

	\paragraph{Choice of $\nu$.} The number $\nu > 0$, which will be used throughout this section, is determined to ensure that all $\xvec{z} \in \xCzero(\overline{\mathcal{E}}\times[0,1];\mathbb{R}^2)$ with \eqref{equation:propertiesv}
	satisfy \cref{equation:flushingpropertyperturbation,equation:draggingpropertyperturbation}. During this process, the parameter $\nu_0 > 0$ in \Cref{lemma:nu} is taken so small that
	\begin{equation}\label{equation:sn0}
		\begin{gathered}
			\operatorname{dist}(\xvec{x}_1, \xvec{x}_2) > \frac{d_{\Lambda}}{2K} \, \, \Longrightarrow \, \, \min_{j \in \{0, \dots, K\}} \operatorname{dist}(\xmcal{Y}^*(\xvec{x}_1, 0, j/K), \xmcal{Y}^*(\xvec{x}_2, 0, j/K)) > 2 \nu_0,
		\end{gathered}
	\end{equation}
	which is possible since $\xmcal{Y}^*(\cdot, j/K, 0)$ is uniformly continuous for $j \in \{1,\dots,K\}$. Thus, as $\xvec{y}^*$ is constructed only based on $\mathcal{E}$ and $\omegaup$, the small numbers~$\nu_0$ and~$\nu$ are well-defined in sole dependence on the geometry.
	
	\begin{rmrk}\label{remark:smallnu0}
		It shall be used at several occasions that \eqref{equation:sn0} implies $\nu_0 \leq d_{\Lambda}/4K$. The specific form of \eqref{equation:sn0} is employed when completing the proof of \Cref{theorem:local} in \Cref{subsubsection:piecingtogether}. 
	\end{rmrk}
	
	To specify a covering of $\mathcal{E}$, which is later utilized to localize magnetic field contributions (when proving \Cref{theorem:tkloc}), we take open $\mathcal{O}_1, \mathcal{O}_2 \subset \mathbb{R}^2$ with (see also \Cref{Figure:flushing} below)
	\begin{equation}\label{equation:coveringproperties}
		\begin{gathered}
		\operatorname{dist}(\mathcal{O}_1, \overline{\mathcal{E}}\setminus\omegaup) > 2d_{\Lambda}, \quad \operatorname{dist}(\mathcal{O}_2, \Lambda) > 2d_{\Lambda}, \quad \overline{\mathcal{E}} \subset (\mathcal{O}_1 \cup \mathcal{O}_2),
		\end{gathered}
	\end{equation}
	where $d_{\Lambda} > 0$ from \eqref{equation:dl} constitutes a lower bound for the length of any continuous curve entering $\Lambda$ and passing through the cut $\Sigma$. The choices in \eqref{equation:coveringproperties} are possible due to \eqref{equation:lmbddist}.
	Then, let~$\{\mu_1, \mu_2\} \subset \xCinfty_0(\mathbb{R}^2;\mathbb{R})$ be a partition of unity so that
	\begin{equation}\label{equation:partitionofunity}
		\forall j \in \{1,2\}\colon \operatorname{supp}(\mu_j) \subset \mathcal{O}_j, \quad \forall \xvec{x} \in \overline{\mathcal{E}}\colon \mu_1(\xvec{x}) + \mu_2(\xvec{x}) = 1.
	\end{equation}
	Given the number $K \in \mathbb{N}$ from \Cref{lemma:flushing}, while selecting $K_0 \in (0,1/2K)$ so small that $\xvec{y}^*(\cdot,t) = \xsym{0}$ for all $t \in [1/K-2K_0,1/K]$, a profile $\beta \in \xCinfty(\mathbb{R}; [0,1])$ is fixed with
	\begin{equation}\label{equation:beta}
		\beta(t) = \begin{cases}
			1 & \mbox{ if } t \in (-\infty,1/K-K_0],\\
			0 & \mbox{ if } t \in [1/K - K_0/2, +\infty).
		\end{cases}
	\end{equation}
	Finally, let $C_* = C_*(\nu_0, m) \geq 1$ be a constant such that for any two nonempty relatively open sets $\xB_1, \xB_2 \subset \overline{\mathcal{E}}$ with $\operatorname{dist}(\xB_1, \xB_2) > \nu_0/2$ there exists $\chi_{C^*} \in \xCinfty(\overline{\mathcal{E}}; \mathbb{R})$ satisfying
	\begin{gather}
				\|\xnab^{\perp}(\chi_{C^*} f)\|_{m, \alpha, \mathcal{E}} \leq C_* \|\xnab^{\perp}f\|_{m, \alpha, \mathcal{E}}, \quad \chi_{C^*}(\xvec{x}) = \begin{cases}
				1 & \mbox{ if } \xvec{x} \in \xB_1,\\
				0 & \mbox{ if } \xvec{x} \in \xB_2
			\end{cases}\label{equation:inequality0}
	\end{gather} 
	and
	\begin{gather}
		\|\xnab^{\perp}(\mu_1 f)\|_{m, \alpha, \mathcal{E}} + \|\xnab^{\perp}(\mu_2 f)\|_{m, \alpha, \mathcal{E}} \leq C_* \|\xnab^{\perp}f\|_{m, \alpha, \mathcal{E}}\label{equation:inequality1}
	\end{gather}
	for all $f \in \xCn{{m+1,\alpha}}(\overline{\mathcal{E}}; \mathbb{R})$ with $f = 0$ at $\partial \mathcal{E}$. This is possible by elliptic regularity theory since $f$ solves $\Delta f = - \xwcurl{(\xnab^{\perp} f)}$.
	\tikzset{>={Latex[width=3.5mm,length=3.5mm]}}
	\begin{figure}[ht!]
		\centering
		\resizebox{0.90\textwidth}{!}{
			\begin{tikzpicture}
				\clip(-4.8,-3.5) rectangle (15.3,4.9);

				\draw[line width=0.5mm, color=SkyBlue!30] (-1.5,0.5)--(-1.5,1.2);
				\draw[line width=0.5mm, color=SkyBlue!30] (-1.5,3.5)--(-1.5,2.8);
				\draw[line width=0.5mm, color=MidnightBlue!10, fill=MidnightBlue!10] (-1,3.5) arc[start angle=90, end angle=270, x radius=3.4, y radius =1.5];
				\draw[line width=0.5mm, color=black] (5.5,0.5) arc[start angle=270, end angle=450, x radius=1.5, y radius =1.5];
				\draw[line width=0.5mm, color=black] plot[smooth, tension=0.6] coordinates { (0.5,3.5)  (1.5,3.5) (2,3.5) (2.5,2)  (4,3.5) (4.5,3.5) (5.5,3.5)};
				\draw[line width=0.5mm, color=black] (0.5,0.5)--(5.5,0.5);
				\draw[line width=0.3mm, color=FireBrick] (-3,0.759)  to [out=76,in=280,looseness=1.6] (-1.3,2);
				\filldraw[color=black, fill=white, line width=0.5mm,  fill opacity=1, decoration={markings, mark=at position 0.193 with {\arrow{>}}, mark=at position 0.73 with {\arrow{>}}},
				postaction={decorate}] (-1.2,2.4) ellipse (1.2cm and 0.55cm);
				\draw[line width=0.5mm, color=black] (-1,3.5)--(0.5,3.5);
				\draw[line width=0.5mm, color=black] (-1,0.5)--(0.5,0.5);

				\draw[dashed,line width=0.7mm, color=black] (0.5,0.5)--(0.5,3.5);
				\draw[dashed,line width=0.7mm, color=black] (5.5,0.5)--(5.5,3.5);

				\draw[line width=0.5mm, ->] (-2.1,1)  to [out=-15,in=180,looseness=0.8] (0.399,0.8);
				\draw[line width=0.5mm, ->] (3.9,2.45)  to [out=-149,in=-70,looseness=1.6] (1.9,2.3);
				\draw[line width=0.2mm, dashed] (-4.6,-0.3) rectangle++ (4.3,4.5); 
				\draw[line width=0.2mm, dashed] (-0.38,-0.8) rectangle++ (7.5,5.5);

				\draw[line width = 0.5mm, color=FireBrick] (-0.35,0.55) circle (0.35);
				
				\draw[line width=3.5mm, color=FireBrick!50] (1.5+10, -5.5+5)  to [out=90,in=-90,looseness=1.6] (1.1+10, -1.13+5);
				
				\fill[color=white, fill=white] (0.8+10,-1.125+5)  rectangle (2+10,-0.85+5);
				
				\draw[line width=0.5mm, color=FireBrick](-0.25,0.23)  to [out=-70,in=-155,looseness=1.6] (-1.9+10,-1.2);
				\draw[line width=0.5mm] (-2.15+10,-5.5+5) -- (5.15+10,-5.5+5);
				\fill[color=white, fill=white] (-2.1+10,-5.5+5)  rectangle (5.05+10,-5.9+5);
				\fill[line width=0pt, pattern={mylines[size=2.5pt,line width=0.9pt,angle=72]},
				pattern color=black] (-2.1+10,-5.5+5) rectangle (5.05+10,-5.7+5);
				
				\draw[line width=0.25mm, color=black, ->] (-1.3+10, -2.8+5)  to [out=0,in=-180,looseness=1.6] (4.2+10,-3+5);
				\filldraw[color=white, fill=white] (-1.1+10, -3+5) rectangle (-1.05+10,-2.6+5);
				\filldraw[color=white, fill=white] (-0.95+10, -3+5) rectangle (-0.9+10,-2.6+5);
				\filldraw[color=white, fill=white] (-0.8+10, -3+5) rectangle (-0.75+10,-2.6+5);
				\filldraw[color=white, fill=white] (-0.65+10, -3+5) rectangle (-0.6+10,-2.6+5);
				
				\draw[color=black, fill=black] (-1.3+10, -2.8+5) circle (0.09);
				\draw[color=black, fill=black] (4.2+10,-3+5) circle (0.09);
				\draw[line width = 0.7mm, color=FireBrick] (1.5+10,-4.8+5) circle (3.7);

				\coordinate[label=left:\Large{\color{FireBrick}$\Sigma$}] (AA) at (-2.73,1.39);
				\coordinate[label=left:\Large{$\mathcal{O}_1$}] (B) at (-3.6,3.8);
				\coordinate[label=left:\Large{$\mathcal{O}_2$}] (C) at (0.92,4.3);
				
				\coordinate[label=right:\Large{$\mathcal{E}\setminus\mathcal{\omegaup} = \Omega$}] (D) at (4,-0.2);
				\draw[line width=0.1mm] (3.5,1.5) -- (D);
				\coordinate[label=right:\Large{$\Lambda$}] (E) at (-4.3,-0.63);
				\draw[line width=0.1mm] (-3.5,2) -- (E);

				\coordinate[label=right:\Large{$\partial\mathcal{E}$}] (E) at (1+10, -6+5);
				\coordinate[label=right:\Large{$\xvec{x}$}] (E) at (-1.55+10, -3.1+5);
				\coordinate[label={[rotate=45]left:\Large{$\xmcal{Y}^*(\xvec{x}, 0 , 1/K)$}}] (E) at (4.38+10,-3.19+5);
			\end{tikzpicture}
		}
		\caption{The arrows schematically indicate $\xmcal{Y}^{*}$'s behavior, while $\Lambda$ and $\Sigma$ refer to the sets introduced in \Cref{subsection:setup}. To connect with \Cref{Figure:exampledomains}, the part $\Omega = \mathcal{E}\setminus\omegaup$ is denoted. Possible choices for~$\mathcal{O}_1$ and~$\mathcal{O}_2$ with \eqref{equation:coveringproperties} are illustrated by dashed rectangles. The small (red) circle attached to a large (red) circle indicates "zooming in": to motivate at this point the choices of $\mathcal{O}_1$ and $\mathcal{O}_2$, the thick (red) strip in the zoomed region indicates the support of an undesired source term/regularity corrector that will appear in the respective proof of \Cref{theorem:tkloc} given in \Cref{subsection:proofA}/\Cref{subsection:proofB}; it is shown for a sample point~$\xvec{x}$ how information flows along~$\xvec{y}^*$; the dashed part of the line joining~$\xvec{x}$ with $\xmcal{Y}^*(\xvec{x}, 0 , 1/K)$ expresses a distance possibly much larger than $\nu_0$.}
		\label{Figure:flushing}
	\end{figure}
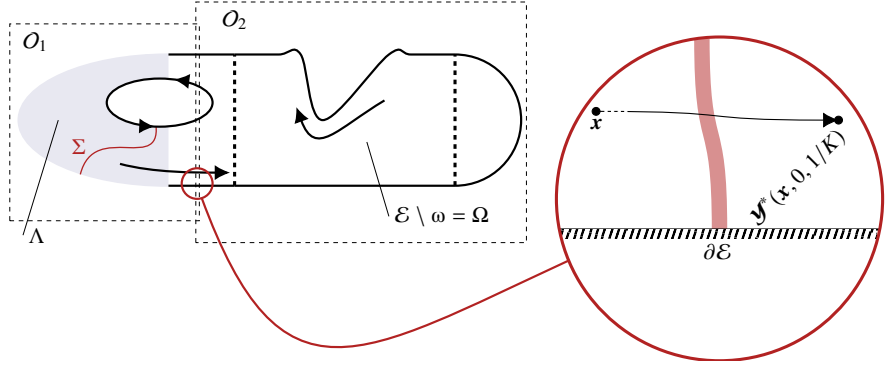

	\subsubsection{Auxiliary problems}\label{subsubsection:aux}
	Let us first consider an incompressible ideal MHD problem in the absence of external electromagnetic forces, namely
	\begin{equation}\label{equation:MHD00unforcedlargedomain}
		\begin{cases}
			\partial_t \xvec{V} + (\xvec{V} \cdot \xsym{\nabla}) \xvec{V} - (\xvec{H} \cdot \xsym{\nabla})\xvec{H} + \xsym{\nabla} P = \xsym{\Xi} & \mbox{ in } \mathcal{E} \times (0,1),\\
			\partial_t \xvec{H} + (\xvec{V} \cdot \xsym{\nabla}) \xvec{H} - (\xvec{H} \cdot \xsym{\nabla}) \xvec{V} = \xsym{0} & \mbox{ in } \mathcal{E} \times (0,1),\\
			\xdiv{\xvec{V}} = \xdiv{\xvec{H}} = 0  & \mbox{ in } \mathcal{E} \times (0,1),\\
			\xvec{V} \cdot \xvec{n} = \xvec{H} \cdot \xvec{n} = 0  & \mbox{ on } \partial\mathcal{E} \times(0,1),\\
			\xvec{V}(\cdot, 0)  =  \xvec{V}_0,\, \xvec{H}(\cdot, 0)  =  \xvec{H}_0  & \mbox{ in } \mathcal{E},
		\end{cases}
	\end{equation}
	where $\xvec{V}_0, \xvec{H}_0 \in \xCn{{m,\alpha}}_{*}(\overline{\mathcal{E}}; \mathbb{R}^2)$ are the prescribed initial states and $\xsym{\Xi}$ plays the role of a velocity control. In the present regularity class, energy estimates show that there exists at most one solution to \eqref{equation:MHD00unforcedlargedomain} when the right-hand side is fixed. Furthermore, recalling the space $\xZ(\mathcal{E})$ from~\eqref{equation:Z}, the magnetic field~satisfies (\cf~\eqref{equation:consv})
	\begin{equation}\label{equation:consvauxe}
		\xdrv{}{t} \int_{\mathcal{E}} \xvec{H}(\xvec{x},t) \cdot \xvec{Q}(\xvec{x}) \, \xdx{\xvec{x}} = 0, \quad \xvec{Q} \in \xZ(\mathcal{E}).
	\end{equation} 
	Hence, if $\xvec{H}_0$ is orthogonal in $\xLtwo(\mathcal{E};\mathbb{R}^2)$ to ~$\xZ(\mathcal{E})$, one can choose the stream function representation~$\xvec{H} = \xnab^{\perp} \psi$ such that~$\psi = 0$ at $\partial\mathcal{E}\times[0,1]$. To see this, let~$\xvec{Q} \in \xZ(\mathcal{E})\setminus\{\xsym{0}\}$ and write $\xvec{Q} = \xnab^{\perp}q$ with a harmonic function~$q$ satisfying \eqref{equation:dpq} for some $a \neq 0$. In particular, $\xnab q \cdot \xvec{n}$ cannot vanish on the connected components of $\partial \mathcal{E}$. The claim follows now after several integrations by parts from
	\begin{equation}\label{equation:constarg}
		0 = \int_{\mathcal{E}} \xvec{H}(\xvec{x},t) \cdot \xvec{Q} \, \xdx{\xvec{x}} = (c_0(t) - c_1(t)) \int_{\Gamma^0} \xnab q \cdot \xvec{n},
	\end{equation}
	where $c_0(t), c_1(t) \in \mathbb{R}$ are the values of $\psi(\cdot, t)$ at the connected components~$\Gamma^0$ and~$\Gamma^1$ of~$\partial\mathcal{E}$.

	As stated by the next theorem (and proved in \Cref{subsection:constr}), there exists a physically localized force
	\[
		\xsym{\Xi}\colon\overline{\mathcal{E}}\times[0,1] \longrightarrow \mathbb{R}^2, \quad \operatorname{supp}(\xsym{\Xi}) \subset \Lambda\times[0,1]
	\]
	such that \eqref{equation:MHD00unforcedlargedomain} admits a solution $(\xvec{V}, \xvec{H})$ living in the vicinity of~$(\xvec{y}^*, \xsym{0})$; if the initial magnetic field is nonzero, we do not aim at this point to reach any specific target state. In order to prepare for iterated applications of the following result, it's formulation involves an increasing sequence of small initial data bounds. Moreover, the constant~$C_* \geq 1$ is that from \eqref{equation:inequality0} and \eqref{equation:inequality1}, and it is recalled that the numbers~$\nu > 0$ and $\nu_0 > 0$ have been selected at the beginning of \Cref{subsection:nonlineardecomposition} (see \eqref{equation:sn0}).
	\begin{thrm}[\cf~\Cref{subsection:constr}]\label{theorem:ret}
		Let $\xvec{V}_0, \xvec{H}_0 \in \xCn{{m,\alpha}}_{*}(\overline{\mathcal{E}}; \mathbb{R}^2)$ and assume that $\xvec{H}_0$ is orthogonal to $\xZ(\mathcal{E})$ in $\xLtwo(\mathcal{E};\mathbb{R}^2)$. There are numbers $0 < \delta_0 < \dots < \delta_K < \delta_{K+1} \coloneq \nu$
		so that for each $l \in \{0,\dots,K\}$ the estimate $\|\xvec{V}_0\|_{m,\alpha,\mathcal{E}} + \|\xvec{H}_0\|_{m,\alpha,\mathcal{E}} < \delta_l$
		implies the existence of profiles $(\xvec{V}, \xvec{H}, P, \xsym{\Xi})$ that obey \eqref{equation:MHD00unforcedlargedomain} and satisfy
		\begin{gather*}
			\xvec{V}, \xvec{H} \in \xCzero([0,1];\xCn{{m-1,\alpha}}(\overline{\mathcal{E}}; \mathbb{R}^2)) \cap \xLinfty((0,1);\xCn{{m,\alpha}}(\overline{\mathcal{E}}; \mathbb{R}^2)), \\ 
			P \in \xCzero([0,1];\xCn{{m-1,\alpha}}(\overline{\mathcal{E}};\mathbb{R})) \cap \xLinfty((0,1);\xCn{{m,\alpha}}(\overline{\mathcal{E}};\mathbb{R})),\\
			\xsym{\Xi} \in \xCzero([0,1];\xCn{{m-2,\alpha}}(\overline{\mathcal{E}};\mathbb{R}^2))\cap \xLinfty((0,1);\xCn{{m-1,\alpha}}(\overline{\mathcal{E}};\mathbb{R}^2)),\\
			\max_{t\in[0,1]} \left(\|\xvec{V}-\xvec{y}^*\|_{m,\alpha,\mathcal{E}}(t) + \|\xvec{H}\|_{m,\alpha,\mathcal{E}}(t)\right) < \delta_{l+1}/3C_*, \quad  \operatorname{supp}(\xsym{\Xi}) \subset \Lambda\times[0,1].
		\end{gather*}
		When $\xvec{H}_0 = \xsym{0}$, the functions $(\xvec{V}, P, \xsym{\Xi})$ solve the following null controllability problem for the incompressible Euler system:
		\begin{equation}\label{equation:Eulernctrl}
			\begin{cases}
				\partial_t \xvec{V} + (\xvec{V} \cdot \xsym{\nabla}) \xvec{V} + \xsym{\nabla} P = \xsym{\Xi} & \mbox{ in } \mathcal{E} \times (0,1),\\
				\xdiv{\xvec{V}} =  0  & \mbox{ in } \mathcal{E} \times (0,1),\\
				\xvec{V} \cdot \xvec{n} = 0  & \mbox{ on } \partial\mathcal{E} \times(0,1),\\
				\xvec{V}(\cdot, 0)  =  \xvec{V}_0, \xvec{V}(\cdot, 1)  =  \xvec{0},  & \mbox{ in } \mathcal{E}.
			\end{cases}
		\end{equation}
	\end{thrm}

	The main building blocks for the controls $\xsym{\xi}$ and $\xsym{\eta}$ in \eqref{equation:MHD00SimplyConnectedExtended} will be constructed by means of the following theorem. To this end, given $r > 0$ and~$\emptyset\neq\xS \subset \overline{\mathcal{E}}$, we write $\mathcal{N}_r(\xS)$ for the relative~$r$-neighborhood of~$\xS$ in~$\overline{\mathcal{E}}$; more precisely, $\mathcal{N}_r(\xS)$ is the restriction to $\overline{\mathcal{E}}$ of the set $\{ \xvec{x} \in \mathbb{R}^2 \, | \, \operatorname{dist}(\xvec{x}, \xS) < r \}$.
	To facilitate iterated applications of the next result, we use \Cref{theorem:ret} to fix the numbers $\delta_j$ with
	\[
		0 < \delta_0 < \dots < \delta_K < \delta_{K+1} = \nu.
	\]
	
	\begin{thrm}[{\cf~Section~\Rref{subsection:proofA} or~\Rref{subsection:proofB}}]\label{theorem:tkloc}
	Let $\widetilde{\xvec{u}}_0, \widetilde{\xvec{B}}_0 \in \xCn{{m,\alpha}}_{*}(\overline{\mathcal{E}}; \mathbb{R}^2)$ with~$\widetilde{\xvec{B}}_0$ orthogonal to $\xZ(\mathcal{E})$ in $\xLtwo(\mathcal{E};\mathbb{R}^2)$, and assume that $\|\widetilde{\xvec{u}}_0\|_{m,\alpha,\mathcal{E}} + \|\widetilde{\xvec{B}}_0\|_{m,\alpha,\mathcal{E}} < \delta_l$ for some~$l \in \{0,\dots, K\}$.
	There exist controls
	\[
		\widetilde{\xsym{\xi}}, \widetilde{\xsym{\eta}}  \in \xCzero([0,1/K];\xCn{{m-2,\alpha}}(\overline{\mathcal{E}};\mathbb{R}^2))  \cap \xLinfty((0,1/K);\xCn{{m-1,\alpha}}(\overline{\mathcal{E}};\mathbb{R}^2))
	\]
	with
		\begin{gather*}
			\operatorname{supp}(\widetilde{\xsym{\xi}}) \subset \omegaup\times[0,1/K], \quad \operatorname{supp}(\widetilde{\xsym{\eta}}) \subset \omegaup\times[1/K-2K_0, 1/K],\\
			\xdiv{\widetilde{\xsym{\eta}}} = 0 \mbox{ in } \mathcal{E}\times[0,1/K], \quad \widetilde{\xsym{\eta}}\cdot \xvec{n} = 0 \mbox{ on } \partial \mathcal{E}\times[0,1/K], \\
			\widetilde{\xsym{\eta}} = \xnab^{\perp} \widetilde{\phi} \mbox{ in } \mathcal{E}\times[0,1/K], \quad \widetilde{\phi} = 0 \mbox{ on } \partial \mathcal{E}\times[0,1/K]
		\end{gather*} 
		and a solution
		\begin{gather*}
			\widetilde{\xvec{u}}, \widetilde{\xvec{B}} \in \xCn{0}([0,1/K];\xCn{{m-1,\alpha}}(\overline{\mathcal{E}}; \mathbb{R}^2)) \cap \xLinfty((0,1/K);\xCn{{m,\alpha}}(\overline{\mathcal{E}}; \mathbb{R}^2)),\\
			\widetilde{p} \in \xCn{0}([0,1/K];\xCn{{m-1,\alpha}}(\overline{\mathcal{E}}; \mathbb{R})) \cap \xLinfty((0,1/K);\xCn{{m,\alpha}}(\overline{\mathcal{E}}; \mathbb{R}))
		\end{gather*}
		to the ideal MHD problem
		\begin{equation}\label{equation:MHD00SimplyConnectedReducedsupp}
			\begin{cases}
				\partial_t \widetilde{\xvec{u}} + (\widetilde{\xvec{u}} \cdot \xsym{\nabla}) \widetilde{\xvec{u}} - (\widetilde{\xvec{B}} \cdot \xsym{\nabla})\widetilde{\xvec{B}} + \xsym{\nabla} \widetilde{p} = \widetilde{\xsym{\xi}} & \mbox{ in } \mathcal{E} \times (0, 1/K),\\
				\partial_t \widetilde{\xvec{B}} + (\widetilde{\xvec{u}} \cdot \xsym{\nabla}) \widetilde{\xvec{B}} - (\widetilde{\xvec{B}} \cdot \xsym{\nabla}) \widetilde{\xvec{u}} = \widetilde{\xsym{\eta}} & \mbox{ in } \mathcal{E} \times (0,1/K),\\
				\xdiv{\widetilde{\xvec{u}}} = \xdiv{\widetilde{\xvec{B}}} = 0  & \mbox{ in } \mathcal{E} \times (0, 1/K),\\
				\widetilde{\xvec{u}} \cdot \xvec{n} = \widetilde{\xvec{B}} \cdot \xvec{n} = 0  & \mbox{ on } \partial\mathcal{E} \times(0, 1/K),\\
				\widetilde{\xvec{u}}(\cdot, 0)  =  \widetilde{\xvec{u}}_0,\, \widetilde{\xvec{B}}(\cdot, 0)  =  \widetilde{\xvec{B}}_0  & \mbox{ in } \mathcal{E}
			\end{cases}
		\end{equation}
		satisfying the estimates
		\begin{equation}\label{equation:scondB}
			\begin{gathered}
				\max_{t\in[0,1]} \|\widetilde{\xvec{u}}-\xvec{y}^*\|_{m,\alpha,\mathcal{E}}(t) < \delta_{l+1}/3C_*, \\
				\|\widetilde{\xvec{u}}\|_{m,\alpha,\mathcal{E}}(1/K) + \|\widetilde{\xvec{B}}\|_{m,\alpha,\mathcal{E}}(1/K) < \delta_{l+1},
			\end{gathered}
		\end{equation}
		the orthogonality relations
		\[
			\forall \xvec{Q} \in \xZ(\mathcal{E}), \, \forall  t \in [0, 1/K] \colon \langle \widetilde{\xvec{B}}(\cdot,t), \xvec{Q} \rangle_{\xLtwo(\mathcal{E};\mathbb{R}^2)},
		\]
		the annihilation property
		\begin{gather}
			\widetilde{\xvec{B}}(\mathcal{N}_{d_{\Lambda}/2}(\Lambda),1/K) = \{\xsym{0}\},\label{equation:suppcondB}
		\end{gather}
		and for all connected $\mathcal{S}\subset\overline{\mathcal{E}}$ with $\partial\mathcal{E}\cap \mathcal{S} \neq \emptyset$ the local flushing mechanism
		\begin{gather}
			\left(\widetilde{\xvec{B}}_0(\xmcal{Y}^*(\mathcal{N}_{2\nu_0}(\mathcal{S}),1/K,0)) = \{\xsym{0}\} \right) \Longrightarrow  \left(\widetilde{\xvec{B}}(\mathcal{S}, 1/K) = \{\xsym{0}\}\right).\label{equation:suppcondB3}
		\end{gather}
	\end{thrm}
	
	\begin{rmrk}\label{remark:locfl}
		We provide two independent proofs for \Cref{theorem:tkloc}, using different ways to split and manipulate trajectories of \eqref{equation:MHD00unforcedlargedomain}. The version given in \Cref{subsection:proofA} directly decomposes such a trajectory and analyzes the support of an undesired source term. In \Cref{subsection:proofB}, new trajectories are started from decomposed initial states, and the support of an undesired regularity corrector is investigated. A combination of Theorems~\Rref{theorem:ret} and~\Rref{theorem:tkloc} is used in \Cref{subsubsection:piecingtogether} to conclude \Cref{theorem:local} (\cf~\Cref{Figure:divide-and-control}).
	\end{rmrk}

	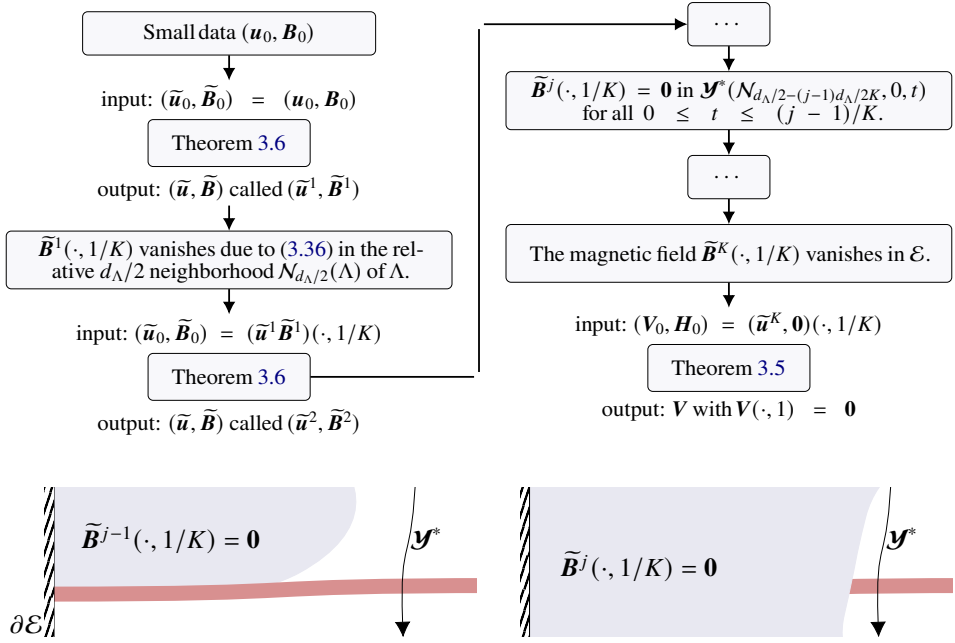
\begin{figure}[ht!]
		\centering
		\subcaptionbox*{}[1\textwidth]{
				\resizebox{0.95\textwidth}{!}{\begin{tikzpicture}[node distance=0.5cm]
					\node (A) [blue0b] {\LARGE Small data $(\xvec{u}_0, \xvec{B}_0)$};

					\node(AB) [blue, below = 0.7cm of A] {\LARGE input: $(\widetilde{\xvec{u}}_0, \widetilde{\xvec{B}}_0) = (\xvec{u}_0, \xvec{B}_0)$};
					
					\node (B) [blue0, below = 0.1cm of AB] {\LARGE \Cref{theorem:tkloc}};

					\node(BB) [blue, below = 0.1cm of B] {\LARGE output: $(\widetilde{\xvec{u}}, \widetilde{\xvec{B}})$ called $(\widetilde{\xvec{u}}^1, \widetilde{\xvec{B}}^1)$};

					\node (C) [blue0bb, below = 0.7cm of BB] {\LARGE $\widetilde{\xvec{B}}^1(\cdot, 1/K)$ vanishes due to \eqref{equation:suppcondB} in the relative $d_{\Lambda}/2$ neighborhood $\mathcal{N}_{d_{\Lambda}/2}(\Lambda)$ of $\Lambda$.};

					\node(CD) [blue, below = 0.7cm of C] {\LARGE input: $(\widetilde{\xvec{u}}_0, \widetilde{\xvec{B}}_0) =(\widetilde{\xvec{u}}^1 \widetilde{\xvec{B}}^1)(\cdot,1/K)$};
					
					\node (D) [blue0, below = 0.1cm of CD] {\LARGE \Cref{theorem:tkloc}};
					
					\node(DD) [blue, below = 0.1cm of D] {\LARGE output: $(\widetilde{\xvec{u}}, \widetilde{\xvec{B}})$ called $(\widetilde{\xvec{u}}^2, \widetilde{\xvec{B}}^2)$};

					\node (NA0) [right = 4.5cm of D] {};
					\node (NA00) [above = 9.5cm of NA0] {};

					\node (E0) [blue0d, right = 5.6cm of NA00] {\LARGE \dots};
					
					\node (NA) [blue0bb, below = 0.7cm of E0] {\LARGE $\widetilde{\xvec{B}}^{j}(\cdot, 1/K) = \xsym{0}$ in~$ \xmcal{Y}^*(\mathcal{N}_{d_{\Lambda}/2 - (j-1)d_{\Lambda}/2K},0,t)$ $\mbox{ for all }$ {$0\leq t\leq (j-1)/K$}.};

					\node (E) [blue0d, below = 0.7cm of NA] {\LARGE \dots};

					\node (G) [blue0bb, below = 0.7cm of E] {\LARGE The magnetic field $\widetilde{\xvec{B}}^K(\cdot, 1/K)$ vanishes in $\mathcal{E}$.};
					
					\node(GH) [blue, below = 0.7cm of G] {\LARGE input: $(\xvec{V}_0, \xvec{H}_0) =(\widetilde{\xvec{u}}^{K},  \xsym{0})(\cdot,1/K)$};
					\node (H) [blue0, below = 0.1cm of GH] {\LARGE \Cref{theorem:ret}};
					\node(HH) [blue, below = 0.1cm of H] {\LARGE output: $\xvec{V}$ with $\xvec{V}(\cdot, 1) = \xsym{0}$};

					\draw [arrow,line width=0.5mm] (A) -- (AB);
					\draw [arrow,line width=0.5mm] (BB) -- (C);
					\draw [arrow,line width=0.5mm] (C) -- (CD);

					\draw [arrow,line width=0.5mm] (E) -- (G);

					\draw [arrow,line width=0.5mm] (G) -- (GH);

					\draw [line width=0.5mm] (D) -- (NA0);
					\draw [line width=0.5mm] (NA0) -- (NA00);
					\draw [arrow,line width=0.5mm] (NA00) -- (E0);
					\draw [arrow,line width=0.5mm] (E0) -- (NA);
					\draw [arrow,line width=0.5mm] (NA) -- (E);
					
			\end{tikzpicture}}
		}\newline\newline
		\subcaptionbox*{}[1\textwidth]{
				\resizebox{0.95\textwidth}{!}{\begin{tikzpicture}
					\clip(-0.9,1.5) rectangle (10,5);

					\draw[line width=0.0mm, color=MidnightBlue!10, fill=MidnightBlue!10] plot[smooth, tension=1] coordinates { (0.2,6)  (2.6,6) (6.5,6) (5.9,3)  (0.2, 2.5)};

					\draw[line width=3.5mm, color=FireBrick!50, ->] (0,2.5)  to [out=0,in=-180,looseness=1.6] (11, 2.7);

					\draw[line width=0.5mm] (0.2,-0.5) -- (0.2,5.5);
					\fill[color=white, fill=white] (0.2, -0.5) rectangle (-1, 5.5);
					\fill[line width=0pt, pattern={mylines[size=2.5pt,line width=0.9pt,angle=72]},
					pattern color=black] (0, -0.5) rectangle (0.2, 5.5);

					\coordinate[label=right:\LARGE{$\widetilde{\xvec{B}}^{j-1}(\cdot,1/K)= \xsym{0}$}] (E) at (0.7, 3.8);

					\draw[line width=0.15mm, color=black, ->] (8.6,5.5)  to [out=-92,in=92,looseness=1.6] (8.3,0.7+0.8);
					\coordinate[label=right:\LARGE{$\xmcal{Y}^*$}] (E) at (8.3,3.8);
					\coordinate[label=right:\LARGE{$\partial\mathcal{E}$}] (F) at (-0.95,1+0.8);
				\end{tikzpicture}
				\begin{tikzpicture}
				\clip(-0.9,1.5) rectangle (10,5);
				
				\draw[line width=3.5mm, color=FireBrick!50, ->] (0,2.5)  to [out=0,in=-180,looseness=1.6] (11, 2.7);
				\draw[line width=0.0mm, color=MidnightBlue!10, fill=MidnightBlue!10] plot[smooth, tension=1] coordinates { (0.2,6)  (2.6,6) (5.5,6) (8.5,6) (8,4.3) (7.5, 2) (6.8,-1)  (0.2, -1)};

				\draw[line width=0.5mm] (0.2,-0.5) -- (0.2,5.5);
				\fill[color=white, fill=white] (0.2, -0.5) rectangle (-1, 5.5);
				\fill[line width=0pt, pattern={mylines[size=2.5pt,line width=0.9pt,angle=72]},
				pattern color=black] (0, -0.5) rectangle (0.2, 5.5);

				\coordinate[label=right:\LARGE{$\widetilde{\xvec{B}}^{j}(\cdot,1/K)= \xsym{0}$}] (E) at (0.78, 3.15);
				
				\draw[line width=0.15mm, color=black, ->] (8.6,5.5)  to [out=-92,in=92,looseness=1.6] (8.3,0.7+0.8);
				\coordinate[label=right:\LARGE{$\xmcal{Y}^*$}] (E) at (8.3,3.8);
				
			\end{tikzpicture}}
			}
			\caption{\Cref{theorem:local}  will be shown in \Cref{subsubsection:piecingtogether} by iterations of~\Cref{theorem:tkloc} and~\Cref{theorem:ret}. To sketch the local flushing mechanism \eqref{equation:suppcondB3}, the flow~$\xmcal{Y}^*$ is indicated. A thick (red) line indicates the support of the undesired source term/regularity corrector discussed in \Cref{subsection:proofA}/\Cref{subsection:proofB} below.}
			\label{Figure:divide-and-control}
	\end{figure}

	\subsubsection{Completing the proof of \Cref{theorem:local}}\label{subsubsection:piecingtogether}
	We conclude \Cref{theorem:local} via finite iterations of \Cref{theorem:tkloc}, and a final application of \Cref{theorem:ret} with $\xvec{H}_0 = \xsym{0}$. More precisely, a controlled solution to \eqref{equation:MHD00SimplyConnectedExtended} of the form \eqref{equation:preprf0} will now be constructed by defining the profiles $(\widetilde{\xvec{u}}^j, \widetilde{\xvec{B}}^j, \widetilde{p}^j, \widetilde{\xsym{\xi}}^j, \widetilde{\xsym{\eta}}^j)_{j \in \{1,\dots, K\}}$, as well as $(\xvec{V}, P, \xsym{\Xi})$.
	
	\paragraph{Definition of $(\widetilde{\xvec{u}}^j, \widetilde{\xvec{B}}^j, \widetilde{p}^j, \widetilde{\xsym{\xi}}^j, \widetilde{\xsym{\eta}}^j)_j$.}
	Let $(\widetilde{\xvec{u}}^1, \widetilde{\xvec{B}}^1, \widetilde{p}^1,\widetilde{\xsym{\xi}}^1, \widetilde{\xsym{\eta}}^1)$ be the solution to \eqref{equation:MHD00SimplyConnectedReducedsupp} obtained via \Cref{theorem:tkloc} for $\widetilde{\xvec{u}}_0 = \xvec{u}_0$ and $\widetilde{\xvec{B}}_0 = \xvec{B}_0$ satisfying
	\begin{equation}\label{equation:delta}
		\|\widetilde{\xvec{u}}_0\|_{m,\alpha,\mathcal{E}} + \|\widetilde{\xvec{B}}_0\|_{m,\alpha,\mathcal{E}} < \delta \coloneq \delta_0.
	\end{equation}
	This process is repeated by taking $(\widetilde{\xvec{u}}^1, \widetilde{\xvec{B}}^1)(\cdot, 1/K)$ as new initial data, giving rise to a solution $(\widetilde{\xvec{u}}^2, \widetilde{\xvec{B}}^2, \widetilde{p}^2, \widetilde{\xsym{\xi}}^2, \widetilde{\xsym{\eta}}^2)$ to \eqref{equation:MHD00SimplyConnectedReducedsupp} in the sense of \Cref{theorem:tkloc}. In particular, 
	\[
		\|\widetilde{\xvec{u}}^2(\cdot, 0)\|_{m,\alpha,\mathcal{E}} + \|\widetilde{\xvec{B}}^2(\cdot, 0)\|_{m,\alpha,\mathcal{E}} < \delta_1.
	\]
	Iteratively, for $j\in\{2,\dots,K\}$,  \Cref{theorem:tkloc} provides controls $(\widetilde{\xsym{\xi}}^j, \widetilde{\xsym{\eta}}^j)$ and an associated trajectory~$(\widetilde{\xvec{u}}^j, \widetilde{\xvec{B}}^j, \widetilde{p}^j)$ of \eqref{equation:MHD00SimplyConnectedReducedsupp} emerging from the initial states
	\begin{equation}\label{equation:itdefubt}
		\begin{gathered}
			\widetilde{\xvec{u}}_0 = \widetilde{\xvec{u}}^{j-1}(\cdot, 1/K), \quad \widetilde{\xvec{B}}_0 \coloneq \widetilde{\xvec{B}}^{j-1}(\cdot, 1/K), \\
			\|\widetilde{\xvec{u}}^{j-1}(\cdot, 1/K)\|_{m,\alpha,\mathcal{E}} + \|\widetilde{\xvec{B}}^{j-1}(\cdot, 1/K)\|_{m,\alpha,\mathcal{E}} < \delta_{j-1}.
		\end{gathered}
	\end{equation}
	Recalling that~$\mathcal{N}_{r}(\Lambda)$ denotes for $r > 0$ the relative~$r$-neighborhood in $\overline{\mathcal{E}}$ of the simply-connected set~$\Lambda$ defined at the beginning of \Cref{subsection:setup}, one obtains $\widetilde{\xvec{B}}^K(\cdot, 1/K) = \xsym{0}$ as follows (see also \Cref{Figure:correctorgl}).
	\paragraph{A)} Due to \eqref{equation:suppcondB}, the magnetic field   $\widetilde{\xvec{B}}^1(\cdot, 1/K)$ vanishes in $\mathcal{N}_{d_{\Lambda}/2}(\Lambda)$.
	\paragraph{B)} Since $\widetilde{\xvec{B}}^{2}(\mathcal{N}_{d_{\Lambda}/2}(\Lambda), 0) = \{\xsym{0}\}$, it can be inferred from~\cref{equation:sn0,equation:suppcondB,equation:suppcondB3} that~$\widetilde{\xvec{B}}^{2}(\cdot, 1/K)$ vanishes throughout
	\[
		\mathcal{N}_{d_{\Lambda}/2}(\Lambda)\cup \xmcal{Y}^*(\mathcal{N}_{d_{\Lambda}/2 - d_{\Lambda}/2K}(\Lambda),0,1/K).
	\]
	Indeed, as $\operatorname{dist}(\xvec{x}, \overline{\mathcal{E}} \setminus \mathcal{N}_{d_{\Lambda}/2}(\Lambda)) > d_{\Lambda}/2K$ holds for any~$\xvec{x} \in \mathcal{N}_{d_{\Lambda}/2 - d_{\Lambda}/2K}(\Lambda)$, the smallness assumption on $\nu_0$ in \eqref{equation:sn0} implies
	\[
	   		\operatorname{dist}(\xmcal{Y}^*(\xvec{x}, 0 , 1/K), \xmcal{Y}^*(\overline{\mathcal{E}} \setminus \mathcal{N}_{d_{\Lambda}/2}(\Lambda), 0 , 1/K)) > 2\nu_0.
	\]
	Moreover, thanks to $\widetilde{\xvec{B}}^{2}(\mathcal{N}_{d_{\Lambda}/2}(\Lambda), 0) = \{\xsym{0}\}$ and $\xmcal{Y}^*(\cdot, 0, 1/K)$ being a homeomorphism of~$\overline{\mathcal{E}}$, one has
	\[
		\widetilde{\xvec{B}}^2(\xmcal{Y}^*(\mathcal{N}_{2\nu_0}(\xmcal{Y}^*(\mathcal{N}_{d_{\Lambda}/2 - d_{\Lambda}/2K}(\Lambda),0,1/K)),1/K,0), 0) = \{\xsym{0}\}.
	\] 
   	Accounting for \eqref{equation:suppcondB}, the definition of $d_{\Lambda}$ via \eqref{equation:dl}, the orientation of $\xvec{y}^*$, and the maximal dragging distance of $\xvec{y}^*$ as stated in~\Cref{equation:draggingproperty}, one has~$\widetilde{\xvec{B}}^{2}(\cdot, 1/K) = \xsym{0}$ on the union
   	\begin{equation}\label{equation:sl0}
   		\mathcal{N}_{d_{\Lambda}/2}(\Lambda) \cup \bigcup_{t \in [0, 1/K]}\xmcal{Y}^*(\mathcal{N}_{d_{\Lambda}/2 - d_{\Lambda}/2K}(\Lambda),0,t).
   	\end{equation}
	\paragraph{C)} Similarly as before, and since the profile $\xvec{y}^*$ from \Cref{lemma:flushing} is time periodic with period~$1/K$, it can be shown that $\widetilde{\xvec{B}}^{3}(\cdot, 1/K)$ vanishes on
	\begin{equation}\label{equation:sl1}
		\bigcup_{i\in\{0,1,2\}}\bigcup_{t \in [0,i/K]} \xmcal{Y}^*(\mathcal{N}_{d_{\Lambda}/2 - id_{\Lambda}/2K}(\Lambda),0,t).
	\end{equation}
	Let us provide more details on how to see this. First, by utilizing the periodicity of~$\xvec{y}^*$, and the properties of $\xmcal{Y}^*$ as being the flow of $\xvec{y}^*$, one finds
	\[
		\xmcal{Y}^*(\xmcal{Y}^*(\cdot,0,1/K),0,1/K) = \xmcal{Y}^*(\cdot,0,2/K).
	\]
	In view of \eqref{equation:sn0} and $\widetilde{\xvec{B}}^{3}(\cdot, 0) = \widetilde{\xvec{B}}^{2}(\cdot, 1/K)$, if one assumes that
	\[
		\xvec{x}_1 \in \mathcal{N}_{d_{\Lambda}/2 - d_{\Lambda}/K}(\Lambda), \quad \xvec{x}_2 \in \mathcal{N}_{d_{\Lambda}/2 - d_{\Lambda}/2K}(\Lambda), \quad \xvec{x}_3 \in \overline{\mathcal{E}}\setminus\mathcal{N}_{d_{\Lambda}/2 - d_{\Lambda}/2K}(\Lambda),
	\]
	then it holds
	\[
		\operatorname{dist}(\xmcal{Y}^*(\xvec{x}_1,0,2/K), \xmcal{Y}^*(\xvec{x}_3,0,2/K)) > 2\nu_0, \quad \widetilde{\xvec{B}}^{3}(\xmcal{Y}^*(\xvec{x}_2, 0 , 1/K), 0) = \xsym{0}.
	\]
	Therefore, since the map $\xmcal{Y}^*(\cdot, s, t)$ is for any $s, t \geq 0$ a homeomorphism of $\overline{\mathcal{E}}$, one can infer
	\begin{gather*}
		\xmcal{Y}^*(\mathcal{N}_{2\nu_0}(\xmcal{Y}^*(\mathcal{N}_{d_{\Lambda}/2 - d_{\Lambda}/K}(\Lambda), 0 , 2/K)), 1/K, 0) \subset \xmcal{Y}^*(\mathcal{N}_{d_{\Lambda}/2 - d_{\Lambda}/2K}(\Lambda), 0 , 1/K),\\
		\widetilde{\xvec{B}}^{3}\left(\xmcal{Y}^*(\mathcal{N}_{2\nu_0}(\xmcal{Y}^*(\mathcal{N}_{d_{\Lambda}/2 - d_{\Lambda}/K}(\Lambda), 0 , 2/K)), 1/K, 0), 0\right) = \{\xsym{0}\}.
	\end{gather*}
	Then, $\widetilde{\xvec{B}}^{3}(\cdot, 1/K)$ vanishes due to \eqref{equation:suppcondB3} on $\xmcal{Y}^*(\mathcal{N}_{d_{\Lambda}/2 - d_{\Lambda}/K}(\Lambda),0,2/K)$, and, by exactly repeating the analysis from the previous step, also on the union in \eqref{equation:sl0}. Given these observations, one can conclude via~\cref{equation:sn0,equation:draggingproperty,equation:suppcondB3} that $\widetilde{\xvec{B}}^{3}(\cdot, 1/K) = \xsym{0}$ holds on the entire union in~\eqref{equation:sl1}.

	\paragraph{D)} By iterating for $j \in \{4,\dots, K\}$ the arguments of the previous steps, the magnetic field  $\widetilde{\xvec{B}}^{j}(\cdot, 1/K)$ is seen to vanish for each~$j \in \{1,\dots, K\}$ throughout the union
	\[
		\bigcup_{t \in [0,(j-1)/K]} \xmcal{Y}^*(\mathcal{N}_{(d_{\Lambda}/2) - (j-1)d_{\Lambda}/2K}(\Lambda),0,t).
	\]
	Thus, acknowledging the choice of $d_{\Lambda}$ together with the flushing property \eqref{equation:flushingproperty} and maximal dragging distance \Cref{equation:draggingproperty}, it can be concluded that
	\[
			\widetilde{\xvec{B}}^K(\cdot, 1/K) = \xsym{0} \mbox{ in } \overline{\mathcal{E}}, \quad \|\widetilde{\xvec{u}}^K(\cdot, 1/K)\|_{m,\alpha,\mathcal{E}}< \delta_{K},
	\]
	where the estimate for $\widetilde{\xvec{u}}^K(\cdot, 1/K)$ is due to \eqref{equation:itdefubt}.

	\begin{figure}[ht!]
		\centering
		\resizebox{0.99\textwidth}{!}{
			\begin{subfigure}[b]{0.5\textwidth}
				\centering
				\resizebox{1\textwidth}{!}{
					\begin{tikzpicture}
						\clip(0.2,-0.1) rectangle (6,2.1);

						\draw[line width=0.0mm, color=MidnightBlue!10, fill=MidnightBlue!10] plot[smooth, tension=0] coordinates { (0.2,0)  (2.4,0) (2,2) (0.6,2)  (0.2,0)};
						\draw[line width=0.3mm]  (0,0) -- (6,0);
						\draw[line width=0.3mm]  (0,2) -- (6,2);

						\draw[line width=0.2mm, dashed]  (2.6,0) -- (2.6,2);
						
						\draw[line width=0.2mm, dashed]  (2.8,0) -- (3.2,2);
						
						\draw[line width=0.2mm, dashed]  (3,0) -- (3.8,2);
						
						\draw[line width=0.2mm, dashed]  (3.2,0) -- (3.2+1.2,2);

						\draw[line width=0.2mm, dashed]  (3.4,0) -- (3.4+1.6,2);
						
						\draw[line width=0.2mm, dashed]  (3.6,0) -- (5.6,2);
						
						\draw[line width=0.2mm, dashed]  (3.8,0) -- (6.2,2);

						\draw[line width=1.6mm, color=FireBrick!50]  (3.5,0) -- (3.9,2);
						\draw[line width=1.6mm, color=FireBrick!50]  (5.8,0) -- (5.6,2);
						
						\coordinate[label=left:\scriptsize$\Sigma$] (A) at (-2.9,1.2);
						\draw[line width=0.2mm, ->] (-3,1.2)  to [out=-15,in=-135,looseness=0.8] (-2.36,1.58);
						
						\draw[line width=0.3mm]  (0,0) -- (6,0);
						\draw[line width=0.3mm]  (0,2) -- (6,2);
					\end{tikzpicture}
				}
				\label{Figure:gg1}
			\end{subfigure}
			
			\begin{subfigure}[b]{0.5\textwidth}
				\centering
				\resizebox{1\textwidth}{!}{
					\begin{tikzpicture}
						\clip(0.2,-0.1) rectangle (6,2.1);

						\draw[line width=0.0mm, color=MidnightBlue!3, fill=MidnightBlue!3, postaction={pattern = dots, pattern color=MidnightBlue!25}] plot[smooth, tension=0] coordinates { (0.2,0)  (2.6,0) (2.6,2) (0.2,2)  (0.2,0)};
						\draw[line width=0.3mm]  (0,0) -- (6,0);
						\draw[line width=0.3mm]  (0,2) -- (6,2);

						\draw[line width=0.2mm, dashed]  (2.6,0) -- (2.6,2);
						
						\draw[line width=0.2mm, dashed]  (2.8,0) -- (3.2,2);
						
						\draw[line width=0.2mm, dashed]  (3,0) -- (3.8,2);
						
						\draw[line width=0.2mm, dashed]  (3.2,0) -- (3.2+1.2,2);

						\draw[line width=0.2mm, dashed]  (3.4,0) -- (3.4+1.6,2);
						
						\draw[line width=0.2mm, dashed]  (3.6,0) -- (5.6,2);
						
						\draw[line width=0.2mm, dashed]  (3.8,0) -- (6.2,2);

						\draw[line width=1.6mm, color=FireBrick!50]  (3.5,0) -- (3.9,2);
						\draw[line width=1.6mm, color=FireBrick!50]  (5.8,0) -- (5.6,2);
						
						\coordinate[label=left:\scriptsize$\Sigma$] (A) at (-2.9,1.2);
						\draw[line width=0.2mm, ->] (-3,1.2)  to [out=-15,in=-135,looseness=0.8] (-2.36,1.58);
						
						\draw[line width=0.3mm]  (0,0) -- (6,0);
						\draw[line width=0.3mm]  (0,2) -- (6,2);
					\end{tikzpicture}
				}
				\label{Figure:gg2}
			\end{subfigure}
			
			\begin{subfigure}[b]{0.5\textwidth}
				\centering
				\resizebox{1\textwidth}{!}{
					\begin{tikzpicture}
						\clip(0.2,-0.1) rectangle (6,2.1);

						\draw[line width=0.0mm, color=MidnightBlue!3, fill=MidnightBlue!3, postaction={pattern = dots, pattern color=MidnightBlue!25}] plot[smooth, tension=0] coordinates { (0.2,0)  (2.8,0) (3.2,2) (0.2,2)  (0.2,0)};
						\draw[line width=0.3mm]  (0,0) -- (6,0);
						\draw[line width=0.3mm]  (0,2) -- (6,2);
						
						\draw[line width=0.2mm, dashed]  (2.6,0) -- (2.6,2);
						
						\draw[line width=0.2mm, dashed]  (2.8,0) -- (3.2,2);
						
						\draw[line width=0.2mm, dashed]  (3,0) -- (3.8,2);
						
						\draw[line width=0.2mm, dashed]  (3.2,0) -- (3.2+1.2,2);

						\draw[line width=0.2mm, dashed]  (3.4,0) -- (3.4+1.6,2);
						
						\draw[line width=0.2mm, dashed]  (3.6,0) -- (5.6,2);

						\draw[line width=0.2mm, dashed]  (3.8,0) -- (6.2,2);

						\draw[line width=1.6mm, color=FireBrick!50]  (3.5,0) -- (3.9,2);
						\draw[line width=1.6mm, color=FireBrick!50]  (5.8,0) -- (5.6,2);
						
						\coordinate[label=left:\scriptsize$\Sigma$] (A) at (-2.9,1.2);
						\draw[line width=0.2mm, ->] (-3,1.2)  to [out=-15,in=-135,looseness=0.8] (-2.36,1.58);
						
						\draw[line width=0.3mm]  (0,0) -- (6,0);
						\draw[line width=0.3mm]  (0,2) -- (6,2);
					\end{tikzpicture}
				}
				\label{Figure:jj1}
			\end{subfigure}
			
			\begin{subfigure}[b]{0.5\textwidth}
				\centering
				\resizebox{1\textwidth}{!}{
					\begin{tikzpicture}
						\clip(0.2,-0.1) rectangle (6,2.1);

						\draw[line width=0.0mm, color=MidnightBlue!3, fill=MidnightBlue!3, postaction={pattern = dots, pattern color=MidnightBlue!25}] plot[smooth, tension=0] coordinates { (0.2,0)  (3,0) (3.8,2) (0.2,2)  (0.2,0)};
						\draw[line width=0.3mm]  (0,0) -- (6,0);
						\draw[line width=0.3mm]  (0,2) -- (6,2);
						
						\draw[line width=0.2mm, dashed]  (2.6,0) -- (2.6,2);
						
						\draw[line width=0.2mm, dashed]  (2.8,0) -- (3.2,2);
						
						\draw[line width=0.2mm, dashed]  (3,0) -- (3.8,2);
						
						\draw[line width=0.2mm, dashed]  (3.2,0) -- (3.2+1.2,2);

						\draw[line width=0.2mm, dashed]  (3.4,0) -- (3.4+1.6,2);
						
						\draw[line width=0.2mm, dashed]  (3.6,0) -- (5.6,2);
						
						\draw[line width=0.2mm, dashed]  (3.8,0) -- (6.2,2);

						\draw[line width=1.6mm, color=FireBrick!50]  (3.5,0) -- (3.9,2);
						\draw[line width=1.6mm, color=FireBrick!50]  (5.8,0) -- (5.6,2);
						
						\coordinate[label=left:\scriptsize$\Sigma$] (A) at (-2.9,1.2);
						\draw[line width=0.2mm, ->] (-3,1.2)  to [out=-15,in=-135,looseness=0.8] (-2.36,1.58);
						
						\draw[line width=0.3mm]  (0,0) -- (6,0);
						\draw[line width=0.3mm]  (0,2) -- (6,2);
					\end{tikzpicture}
				}
				\label{Figure:jj2}
			\end{subfigure}
		}
		\resizebox{0.99\textwidth}{!}{

			\begin{subfigure}[b]{0.5\textwidth}
				\centering
				\resizebox{1\textwidth}{!}{
					\begin{tikzpicture}
						\clip(0.2,-0.1) rectangle (6,2.1);

						\draw[line width=0.0mm, color=MidnightBlue!3, fill=MidnightBlue!3, postaction={pattern = dots, pattern color=MidnightBlue!25}] plot[smooth, tension=0] coordinates { (0.2,0)  (3.2,0) (3.2+1.2,2) (3.9,2)  (0.2,2)  (0.2,0)};
						\draw[line width=0.3mm]  (0,0) -- (6,0);
						\draw[line width=0.3mm]  (0,2) -- (6,2);
						
						\draw[line width=0.2mm, dashed]  (2.6,0) -- (2.6,2);
						
						\draw[line width=0.2mm, dashed]  (2.8,0) -- (3.2,2);
						
						\draw[line width=0.2mm, dashed]  (3,0) -- (3.8,2);
						
						\draw[line width=0.2mm, dashed]  (3.2,0) -- (3.2+1.2,2);

						\draw[line width=0.2mm, dashed]  (3.4,0) -- (3.4+1.6,2);
						
						\draw[line width=0.2mm, dashed]  (3.6,0) -- (5.6,2);
						
						\draw[line width=0.2mm, dashed]  (3.8,0) -- (6.2,2);

						\draw[line width=1.6mm, color=FireBrick!50]  (3.5,0) -- (3.5+0.4/2.5,2/2.5);
						\draw[line width=1.6mm, color=FireBrick!50]  (5.8,0) -- (5.6,2);
						
						\coordinate[label=left:\scriptsize$\Sigma$] (A) at (-2.9,1.2);
						\draw[line width=0.2mm, ->] (-3,1.2)  to [out=-15,in=-135,looseness=0.8] (-2.36,1.58);
						
						\draw[line width=0.3mm]  (0,0) -- (6,0);
						\draw[line width=0.3mm]  (0,2) -- (6,2);
					\end{tikzpicture}
				}
				\label{Figure:kk1}
			\end{subfigure}
			
			\begin{subfigure}[b]{0.5\textwidth}
				\centering
				\resizebox{1\textwidth}{!}{
					\begin{tikzpicture}
						\clip(0.2,-0.1) rectangle (6,2.1);

						\draw[line width=0.0mm, color=MidnightBlue!3, fill=MidnightBlue!3, postaction={pattern = dots, pattern color=MidnightBlue!25}] plot[smooth, tension=0] coordinates { (0.2,0)  (3.4,0) (5,2) (0.2,2)  (0.2,0)};
						\draw[line width=0.3mm]  (0,0) -- (6,0);
						\draw[line width=0.3mm]  (0,2) -- (6,2);
						
						\draw[line width=0.2mm, dashed]  (2.6,0) -- (2.6,2);
						
						\draw[line width=0.2mm, dashed]  (2.8,0) -- (3.2,2);
						
						\draw[line width=0.2mm, dashed]  (3,0) -- (3.8,2);
						
						\draw[line width=0.2mm, dashed]  (3.2,0) -- (3.2+1.2,2);

						\draw[line width=0.2mm, dashed]  (3.4,0) -- (3.4+1.6,2);
						
						\draw[line width=0.2mm, dashed]  (3.6,0) -- (5.6,2);
						
						\draw[line width=0.2mm, dashed]  (3.8,0) -- (6.2,2);

						\draw[line width=1.6mm, color=FireBrick!50]  (3.5,0) -- (3.5+0.4/10,2/10);
						\draw[line width=1.6mm, color=FireBrick!50]  (5.8,0) -- (5.6,2);

						\coordinate[label=left:\scriptsize$\Sigma$] (A) at (-2.9,1.2);
						\draw[line width=0.2mm, ->] (-3,1.2)  to [out=-15,in=-135,looseness=0.8] (-2.36,1.58);
						
						\draw[line width=0.3mm]  (0,0) -- (6,0);
						\draw[line width=0.3mm]  (0,2) -- (6,2);
					\end{tikzpicture}
				}
				\label{Figure:kk2}
			\end{subfigure}
			
			\begin{subfigure}[b]{0.5\textwidth}
				\centering
				\resizebox{1\textwidth}{!}{
					\begin{tikzpicture}
						\clip(0.2,-0.1) rectangle (6,2.1);

						\draw[line width=0.0mm, color=MidnightBlue!3, fill=MidnightBlue!3, postaction={pattern = dots, pattern color=MidnightBlue!25}] plot[smooth, tension=0] coordinates { (0.2,0)  (3.6,0)  (5.6,2)  (0.2,2)  (0.2,0)};
						\draw[line width=0.3mm]  (0,0) -- (6,0);
						\draw[line width=0.3mm]  (0,2) -- (6,2);
						
						\draw[line width=0.2mm, dashed]  (2.6,0) -- (2.6,2);
						
						\draw[line width=0.2mm, dashed]  (2.8,0) -- (3.2,2);
						
						\draw[line width=0.2mm, dashed]  (3,0) -- (3.8,2);
						
						\draw[line width=0.2mm, dashed]  (3.2,0) -- (3.2+1.2,2);

						\draw[line width=0.2mm, dashed]  (3.4,0) -- (3.4+1.6,2);
						
						\draw[line width=0.2mm, dashed]  (3.6,0) -- (5.6,2);

						\draw[line width=0.2mm, dashed]  (3.8,0) -- (6.2,2);

						\draw[line width=1.6mm, color=FireBrick!50]  (5.8,0) -- (5.6,2);
						
						\coordinate[label=left:\scriptsize$\Sigma$] (A) at (-2.9,1.2);
						\draw[line width=0.2mm, ->] (-3,1.2)  to [out=-15,in=-135,looseness=0.8] (-2.36,1.58);
						
						\draw[line width=0.3mm]  (0,0) -- (6,0);
						\draw[line width=0.3mm]  (0,2) -- (6,2);
					\end{tikzpicture}
				}
				\label{Figure:ll1}
			\end{subfigure}
			
			\begin{subfigure}[b]{0.5\textwidth}
				\centering
				\resizebox{1\textwidth}{!}{
					\begin{tikzpicture}
						\clip(0.2,-0.1) rectangle (6,2.1);

						\draw[line width=0.0mm, color=MidnightBlue!3, fill=MidnightBlue!3, postaction={pattern = dots, pattern color=MidnightBlue!25}] plot[smooth, tension=0] coordinates { (0.2,0)  (3.8,0)  (6.2,2) (0.2,2)  (0.2,0)};
						\draw[line width=0.3mm]  (0,0) -- (6,0);
						\draw[line width=0.3mm]  (0,2) -- (6,2);

						\draw[line width=0.2mm, dashed]  (2.6,0) -- (2.6,2);
						
						\draw[line width=0.2mm, dashed]  (2.8,0) -- (3.2,2);
						
						\draw[line width=0.2mm, dashed]  (3,0) -- (3.8,2);
						
						\draw[line width=0.2mm, dashed]  (3.2,0) -- (3.2+1.2,2);

						\draw[line width=0.2mm, dashed]  (3.4,0) -- (3.4+1.6,2);
						
						\draw[line width=0.2mm, dashed]  (3.6,0) -- (5.6,2);
						
						\draw[line width=0.2mm, dashed]  (3.8,0) -- (6.2,2);

						\draw[line width=1.6mm, color=FireBrick!50]  (5.8,0) -- (5.8-0.2/1.28,2/1.28);

						\coordinate[label=left:\scriptsize$\Sigma$] (A) at (-2.9,1.2);
						\draw[line width=0.2mm, ->] (-3,1.2)  to [out=-15,in=-135,looseness=0.8] (-2.36,1.58);
						
						\draw[line width=0.3mm]  (0,0) -- (6,0);
						\draw[line width=0.3mm]  (0,2) -- (6,2);
					\end{tikzpicture}
				}
				\label{Figure:ll2}
			\end{subfigure}
		}
		\caption{The top left picture displays the interior of~$\Lambda$~as a (blue) shaded region. Dotted filled areas in the other sub-pictures mark sets on which exemplary $\widetilde{\xvec{B}}^1(\cdot, 1/K), \dots, \widetilde{\xvec{B}}^6(\cdot, 1/K)$, and~$\widetilde{\xvec{B}}^7(\cdot, 1/K)$ are respectively known to vanish. The (red) strips contain the support of the undesired source term/regularity corrector from \Cref{theorem:tkloc}'s proof in \Cref{subsection:proofA}/\Cref{subsection:proofA}.}
		\label{Figure:correctorgl}
	\end{figure}
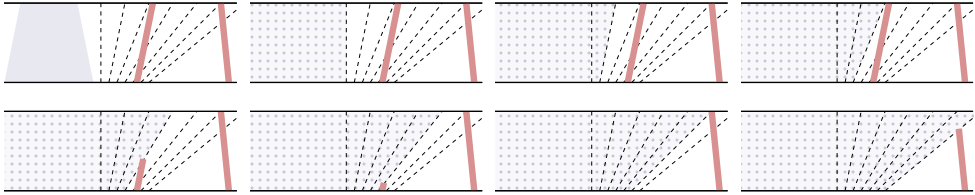
	
	\paragraph{Definition of $(\xvec{V}, P, \xsym{\Xi})$.}
	In order to drive also the velocity field to rest, it remains to apply \Cref{theorem:ret} with the initial states
	\[
		\xvec{V}_0 = \xvec{u}^{K}(\cdot, 1/K), \quad  \xvec{H}_0 = \xsym{0}.
	\]
	This provides a controlled trajectory $(\xvec{V}, P, \xsym{\Xi})$ to the system \eqref{equation:Eulernctrl} which meets the target constraint $\xvec{V}(\cdot, 1) = \xsym{0}$.
	
	\paragraph{Conclusion.} 
	Since the initial values (at $t = 0$) of the auxiliary profiles~$\widetilde{\xvec{u}}^j$, $\widetilde{\xvec{B}}^j$, and~$\widetilde{p}^j$ are for each $j \in \{2, \dots, K\}$ chosen as the final values (at $t = 1/K$) of the respective functions $\widetilde{\xvec{u}}^{j-1}$, $\widetilde{\xvec{B}}^{j-1}$, and $\widetilde{p}^{j-1}$, the profiles $\xvec{u}$, $\xvec{B}$, and~$p$ constructed by means of \eqref{equation:preprf0} are continuous with respect to time. The controls are glued in a way which potentially leads to a finite number of discontinuities. Summarizing the properties that $(\xvec{u}, \xvec{B}, p, \xsym{\xi}, \xsym{\eta})$ inherit from the profiles $(\widetilde{\xvec{u}}^j, \widetilde{\xvec{B}}^j, \widetilde{p}^j, \widetilde{\xsym{\xi}}^j, \widetilde{\xsym{\eta}}^j)_{j \in \{1,\dots, K\}}$ and $(\xvec{V}, P, \xsym{\Xi})$ via Theorems~\Rref{theorem:ret} and~\Rref{theorem:tkloc}, one has  
	\begin{gather*}
			\xvec{u}, \xvec{B} \in \xCn{0}([0,2];\xCn{{m-1,\alpha}}(\overline{\mathcal{E}}; \mathbb{R}^2)) \cap \xLinfty((0,2);\xCn{{m,\alpha}}(\overline{\mathcal{E}}; \mathbb{R}^2)),\\
			p \in \xCn{0}([0,2];\xCn{{m-1,\alpha}}(\overline{\mathcal{E}}; \mathbb{R})) \cap \xLinfty((0,2);\xCn{{m,\alpha}}(\overline{\mathcal{E}}; \mathbb{R})),\\
			\xsym{\xi}, \xsym{\eta} \in \xLinfty((0,2);\xCn{{m-1,\alpha}}(\overline{\mathcal{E}};\mathbb{R}^2)),\\
			\operatorname{supp}(\xsym{\xi}) \subset \omegaup\times[0,2], \quad \operatorname{supp}(\xsym{\eta}) \subset \omegaup \times [0,1],\\
			\xdiv{\xsym{\eta}} = 0 \mbox{ in } \mathcal{E}, \quad \xsym{\eta}\cdot \xvec{n} = 0 \mbox{ on } \partial \mathcal{E}, \quad
			\xsym{\eta} = \xnab^{\perp} \phi, \quad \phi = 0 \, \mbox{ on } \partial \mathcal{E}.
	\end{gather*} 
	Further, we have $\xvec{u}(\cdot, 2)  =  \xsym{0}$ and $\xvec{B}(\cdot, 2)  =  \xsym{0}$.

	\subsection{Proof of \Cref{theorem:ret}}\label{subsection:constr}
	Our proof of \Cref{theorem:ret} is inspired by the return method devised for perfect fluids (\cf~\cite[Part 2, Section 6.2]{Coron2007}), the arguments of \cite[Appendix]{CoronMarbachSueur2020}, and \cite{RisselWang2021}. The velocity shall be viewed as a small perturbation of the curl-free flushing profile $\xvec{y}^*$ selected in \Cref{subsection:flushingprofile}. Meanwhile, the magnetic field is sought near the zero state.
	
	\subsubsection{Preliminaries}\label{subsubsection:prelimconstr}
	To prevent regularity loss, the velocity $\xvec{V}$ and the magnetic field $\xvec{H}$ are temporarily replaced by the symmetrized unknowns (Elsasser variables, \cf~\cite{Elsasser1950})
	\[
		\xvec{z}^{\pm} \coloneq \xvec{V} \pm \xvec{H}.
	\]
	A first observation is that, if $(\xvec{V},\xvec{H}, P)$ satisfy the ideal MHD problem \eqref{equation:MHD00unforcedlargedomain} driven by~$\xsym{\Xi}$, then the pair $(\xvec{z}^+, \xvec{z}^-)$ obeys the interior-controlled system
	\begin{equation}\label{equation:MHD00SimplyConnectedExtendedElsasser}
		\begin{cases}
			\partial_t \xvec{z}^{\pm} + (\xvec{z}^{\mp} \cdot \xsym{\nabla}) \xvec{z}^{\pm}  + \xsym{\nabla} p^{\pm} = \xsym{\Xi} & \mbox{ in } \mathcal{E} \times (0,1),\\
			\xdiv{\xvec{z}^{\pm}} = 0  & \mbox{ in } \mathcal{E} \times (0,1),\\
			\xvec{z}^{\pm} \cdot \xvec{n} = 0  & \mbox{ on } \partial\mathcal{E} \times(0,1),\\
			\xvec{z}^{\pm}(\cdot, 0)  =  \xvec{z}^{\pm}_0 \coloneq \xvec{V}_0 \pm  \xvec{H}_0 & \mbox{ in } \mathcal{E},
		\end{cases}
	\end{equation}
	with $\xnab p^{\pm} = \xnab P$, and vice versa. After acting with the planar curl operator $\xnab\wedge$ on the first line of \eqref{equation:MHD00SimplyConnectedExtendedElsasser}, the gradients $\xnab p^{\pm}$ are dismissed from the analysis. In particular, the scalar quantities \enquote{vorticity $\pm$ current density}, denoted as
	\[
		j^{\pm} \coloneq \xwcurl{\xvec{z}^{\pm}},
	\]
	are governed in $\mathcal{E}\times(0,1)$ by 
	\begin{equation}\label{equation:j+-intr}
		\partial_t j^{\pm} + (\xvec{z}^{\mp} \cdot \xsym{\nabla}) j^{\pm} = G^{\pm}(\xvec{z}^+, \xvec{z}^-) + f,
	\end{equation}
	where
	\[
		G^{\pm}(\xvec{z}^+, \xvec{z}^-) \coloneq - \sum_{l=1}^{2} \xnab z^{\mp}_l \wedge \partial_l \xvec{z}^{\pm}, \quad \xvec{z}^{\pm} = [z^{\pm}_1, z^{\pm}_2], \quad f \coloneq \xwcurl{\xsym{\Xi}}.
	\]
	
	\begin{rmrk}
	The operators $G^{\pm}$ admit favorable representations for obtaining solutions~$\xvec{z}^{\pm}$ to \eqref{equation:MHD00SimplyConnectedExtendedElsasser} as small perturbations of $\xvec{y}^*$. More precisely, given differentiable and divergence-free vector fields $\xvec{z}^{\pm} = [z^{\pm}_1, z^{\pm}_2]$, it holds
	\begin{equation}\label{equation:reprGpm}
		\begin{aligned}
			G^{\pm}(\xvec{z}^+, \xvec{z}^-) & = -\partial_1(z^{\mp}_1 - z^{\pm}_1)(\partial_1z^{\pm}_2 + \partial_2z^{\pm}_1)   - \partial_1(z^{\pm}_2 - z^{\mp}_2)\partial_1z^{\pm}_1 - \partial_2(z^{\pm}_1 - z^{\mp}_1)\partial_1z^{\pm}_1\\
			& = -\partial_1(z^{\mp}_1 - y^*_1)(\partial_1z^{\pm}_2 + \partial_2z^{\pm}_1)  - \partial_1(y^*_1 - z^{\pm}_1)(\partial_1z^{\pm}_2 + \partial_2z^{\pm}_1)   \\ 
			& \quad - \partial_1(z^{\pm}_2 - y^*_2)\partial_1z^{\pm}_1 - \partial_1(y^*_2  - z^{\mp}_2)\partial_1z^{\pm}_1 \\
			& \quad - \partial_2(z^{\pm}_1 - y^*_1)\partial_1z^{\pm}_1 - \partial_2(y^*_1 - z^{\mp}_1)\partial_1z^{\pm}_1.
		\end{aligned}
	\end{equation}
	\end{rmrk}
	
	\paragraph{Functional setup.} For a large number $k > 0$ that will be chosen subsequently, we define the weight (\cf~\cite{RisselWang2021})
	\begin{equation*}\label{equation:weight}
		W_k(t) := \left( \frac{1}{2} + \frac{t}{8} \right)^{-k}
	\end{equation*}
	and seek controlled trajectories of \eqref{equation:MHD00SimplyConnectedExtendedElsasser} in
	\begin{equation*}
		\begin{gathered}
			\xX_{\delta^*,k} \coloneq \Big\{ (\xvec{z}^+, \xvec{z}^-) \in \bigcup_{\beta \in (0,\alpha)}\xCzero([0,1];\xCn{{m,\beta}}_{*}(\overline{\mathcal{E}}; \mathbb{R}^2))^2 \cap \xLinfty((0,1);\xCn{{m,\alpha}}(\overline{\mathcal{E}}; \mathbb{R}^2))^2\, \Big| \\  \xvec{z}^{\pm}(\cdot,0) = \xvec{z}^{\pm}_0 \mbox{ in } \mathcal{E}, \,  \max\limits_{t \in [0,1]} W_k(t)\|\xvec{z}^{\pm} - \xvec{y}^*\|_{m,\alpha,\mathcal{E}}(t) < \delta^* \Big\},
		\end{gathered}
	\end{equation*}
	where $\delta^* \in (0,\nu]$ is arbitrary but fixed. Then, $\xX_{\delta^*,k}\neq \emptyset$ for all initial states $\xvec{z}^{\pm}_0$ that are sufficiently small with respect to $\|\cdot\|_{\alpha,m,\mathcal{E}}$, depending on~$\delta^*$ and $k$. Indeed, given~$\lambda_0 \in \xCinfty([0,1];[0,1])$ with $\lambda_0(0) = 1$,
	one has
	\[
		\max\limits_{t \in [0,1]} W_k(t) \|\xvec{z}^{\pm}_0\|_{m,\alpha,\mathcal{E}}(t) < \delta^* \, \iff \, (\lambda_0\xvec{z}^{+}_0 + \xvec{y}^*,\lambda_0\xvec{z}^{-}_0 + \xvec{y}^*) \in \xX_{\delta^*,k}.
	\]
	
	\subsubsection{Fixed point mapping}\label{subsubsection:constructionF}
	Let $\xZ(\mathcal{E}) \subset \xCinfty(\overline{\mathcal{E}}; \mathbb{R}^2)$ be the space of divergence-free, curl-free, and tangential vector fields introduced in \eqref{equation:Z}, and fix any element $\xvec{Q}^{\#} \in \xZ(\mathcal{E})$ with
	\begin{equation*}\label{equation:normalizedprofileQ}
		\int_{\mathcal{E}} \xvec{Q}^{\#}(\xvec{x}) \cdot \xvec{Q}^{\#}(\xvec{x}) \, \xdx{\xvec{x}} = 1.
	\end{equation*}
	In order to capture the first cohomology projection of $\xvec{V}_0$, we take a non-increasing function $\aleph \in \xCinfty([0,1];\mathbb{R})$ obeying
	\begin{equation}\label{equation:alephdefinition}
		\aleph(s) = \begin{cases}
			\int_{\mathcal{E}} \xvec{V}_0(\xvec{x}) \cdot \xvec{Q}^{\#}(\xvec{x}) \, \xdx{\xvec{x}} & \mbox{ if } s \in [0, 1/3K],\\
			0 & \mbox{ if } s \in [1/2K, 1],
		\end{cases}
	\end{equation}
	recalling that $K$ is the constant from \Cref{lemma:flushing}. 
	
	A mapping $\xmcal{F}\colon\xX_{\delta^*,k} \longrightarrow \xX_{\delta^*,k}$ is defined by assigning to each $(\widetilde{\xvec{z}}^+,\widetilde{\xvec{z}}^-) \in \xX_{\delta^*,k}$ the image
	\[
		\xmcal{F}(\widetilde{\xvec{z}}^+,\widetilde{\xvec{z}}^-) \coloneq (\xvec{z}^+,\xvec{z}^-),
	\]
	where the functions $\xvec{z}^+$ and $\xvec{z}^-$ are constructed from $\widetilde{\xvec{z}}^{\pm}$ through the following steps.

	\paragraph{Step 1 (a). Linearization when $\xvec{H}_0 \neq \xsym{0}$.} Here, we are not aiming to reach any specific target state and simply take $f = 0$ in \eqref{equation:j+-intr}. The curled Elsasser unknowns
	\[
		j^{\pm}\in \bigcup_{\beta \in (0,\alpha)} \xCzero([0,1];\xCn{{m-1,\beta}}(\overline{\mathcal{E}}; \mathbb{R})) \cap \xLinfty((0,1);\xCn{{m-1,\alpha}}(\overline{\mathcal{E}}; \mathbb{R}))
	\]
	are then determined by solving the inhomogeneous linear transport problems
	\begin{equation}\label{equation:IdealControl:constructionf:j+-}
		\begin{cases}
			\partial_t j^{\pm} + (\widetilde{\xvec{z}}^{\mp} \cdot \xsym{\nabla}) j^{\pm} = G^{\pm}(\widetilde{\xvec{z}}^+, \widetilde{\xvec{z}}^-) & \mbox{ in } \mathcal{E} \times (0, 1),\\
			j^{\pm}(\cdot, 0)  =  j^{\pm}_0 \coloneq \xwcurl{\xvec{z}^{\pm}_0} = \xwcurl{(\xvec{V}_0 \pm \xvec{H}_0)} & \mbox{ in } \mathcal{E}.
		\end{cases}
	\end{equation}
	
	\paragraph{Step 1 (b). Linearization when $\xvec{H}_0 = \xsym{0}$.}
	If the initial magnetic field is absent, the goal is to steer the fluid to rest. Thus, we consider the controlled linearized vorticity problem
	\begin{equation}\label{equation:IdealControl:constructionf:w}
		\begin{cases}
			\partial_t w + \frac{1}{2}((\widetilde{\xvec{z}}^{+}+\widetilde{\xvec{z}}^{-}) \cdot \xsym{\nabla}) w = f& \mbox{ in } \mathcal{E} \times (0, 1),\\
			w(\cdot, 0)  =  w_0 \coloneq \xwcurl{\xvec{V}_0} & \mbox{ in } \mathcal{E}
		\end{cases}
	\end{equation}
	and seek a vorticity control $f \in \bigcup_{\beta \in (0,\alpha)} \xCzero([0,1];\xCn{{m-1,\beta}}(\overline{\mathcal{E}};\mathbb{R}))$ that ensures the target condition
	\[
		w(\cdot, 1) = 0  \mbox{ in } \mathcal{E}.
	\]
	
	A suitable force $f$ is defined similarly as in \cite[Appendix]{CoronMarbachSueur2020}. Here, this shall involve a partition of unity adapted to the streamlines of~$\xvec{y}^*$ and~$\widetilde{\xvec{z}}^{\pm}$. 
	Due to Lemmas~\Rref{lemma:flushing} and~\Rref{lemma:nu}, there exist $b \in (0,1/2)$, interior or boundary squares $\xD_1, \dots, \xD_{\widetilde{L}} \subset \mathbb{R}^2$ (contained in $\mathcal{E}$ or having one side in $\mathbb{R}^2\setminus\overline{\mathcal{E}}$) with $\xD_j \cap \overline{\mathcal{E}} \subset \Lambda$ for $j \in \{1,\dots,\widetilde{L}\}$, and open balls~$\widetilde{\xB}_1, \dots, \widetilde{\xB}_{\widetilde{L}} \subset \mathbb{R}^2$ covering the compactum~$\overline{\mathcal{E}}$, and with
	\begin{multline*}\label{equation:cubesquareflush}
		\forall l \in \{1,\dots,\widetilde{L}\}, \, \exists t_l \in (b, 1-b), \, \forall t \in (t_l-b, t_l+b) \colon\\
		\left(\xmcal{Y}^*(\widetilde{\xB}_l\cap \overline{\mathcal{E}}, 0, t) \cup \widetilde{\xsym{\mathcal{V}}}(\widetilde{\xB}_l\cap \overline{\mathcal{E}}, 0, t)\right) \subset \xD_l \cap \Lambda,
	\end{multline*}
	where $\Lambda \subset \omegaup$ denotes the simply-connected subset from \Cref{subsection:setup}, the flow $\xmcal{Y}^*$ is obtained via \eqref{equation:flowofy}, and $\widetilde{\xsym{\mathcal{V}}}$ is the flow of $2^{-1}(\widetilde{\xvec{z}}^{+}+\widetilde{\xvec{z}}^{-})$ given by
	\[
		\xdrv{}{t} \widetilde{\xmcal{V}}(\xvec{x},s,t) = \frac{\widetilde{\xvec{z}}^{+}+\widetilde{\xvec{z}}^{-}}{2}(\widetilde{\xmcal{V}}(\xvec{x},s,t),t), \quad
	\widetilde{\xmcal{V}}(\xvec{x},s,s) = \xvec{x}.
	\]
	Subordinate to $\widetilde{\xB}_1, \dots, \widetilde{\xB}_{\widetilde{L}}$, let $(\widetilde{\mu}_l)_{l = 1,\dots,\widetilde{L}} \subset \xCinfty_0(\mathbb{R}^2;\mathbb{R})$ be a smooth partition of unity in the sense that
	\begin{equation*}
		\forall l \in \{1,\dots,\widetilde{L}\}\colon \operatorname{supp}(\widetilde{\mu}_l) \subset \widetilde{\xB}_l, \quad \forall \xvec{x} \in \overline{\mathcal{E}}\colon \sum_{l=1}^{\widetilde{L}} \widetilde{\mu}_l(\xvec{x}) = 1.
	\end{equation*}
	Thanks to \Cref{lemma:nu}, the balls $\widetilde{\xB}_1, \dots, \widetilde{\xB}_{\widetilde{L}}$ can be selected independently of the pair~$(\widetilde{\xvec{z}}^+, \widetilde{\xvec{z}}^-) \in \xX_{\delta^*,k}$. Hence, we are free to assume throughout that $\widetilde{\mu}_1, \dots, \widetilde{\mu}_{\widetilde{L}}$ only depend on the domain $\mathcal{E}$, the set $\Lambda$, and the profile $\xvec{y}^*$.
	
	\paragraph{Convention.} The partition of unity $(\widetilde{\mu}_l)_{l = 1,\dots,\widetilde{L}}$ is fixed independently of the chosen pair $(\widetilde{\xvec{z}}^{+}, \widetilde{\xvec{z}}^{-})$.
	
	Next, a family $(e_l)_{l\in\{1,\dots,\widetilde{L}\}}$ of building blocks  for $w$ and $f$ is defined by solving the homogeneous linear problems
	\begin{equation}\label{equation:tspsc}
		\begin{cases}
			\partial_t e_l + \frac{1}{2}((\widetilde{\xvec{z}}^++\widetilde{\xvec{z}}^-) \cdot \xdop{\nabla}) e_l = 0 & \mbox{ in } \mathcal{E}\times(0,1),\\
			e_l(\cdot, 0) = \widetilde{\mu}_lw_0 & \mbox{ in } \mathcal{E}.
		\end{cases}
	\end{equation}
	The transport equations in~\eqref{equation:tspsc} describe how the localized initial data~$\widetilde{\mu}_l w_0$ are propagated along the integral curves of $\frac{1}{2}(\widetilde{\xvec{z}}^++\widetilde{\xvec{z}}^-)$. Owing to the flushing property satisfied, due to \Cref{lemma:nu} and the definition of $\xX_{\delta^*,k}$, by the profile~$2^{-1}(\widetilde{\xvec{z}}^++\widetilde{\xvec{z}}^-)$, we thus arrive at explicit expressions for a control~$f$ and the corresponding solution~$w$ to \eqref{equation:IdealControl:constructionf:w}:
	\begin{equation}\label{equation:solftcw}
		w(\xvec{x},t) \coloneq \sum\limits_{l=1}^{\widetilde{L}} \gamma(t-t_l) e_l(\xvec{x},t), \quad f(\xvec{x},t) \coloneq \sum\limits_{l=1}^{\widetilde{L}} \xdrv{\gamma}{t}(t-t_l) e_l(\xvec{x},t),
	\end{equation}
	where $\gamma\colon\mathbb{R}\longrightarrow[0,1]$ is smooth and obeys
	\begin{equation*}
		\gamma(t) = 1 \mbox{ if }  t \in (-\infty,-b], \quad \gamma(t) = 0 \mbox{ if }  t \in [b, +\infty).
	\end{equation*}
	Concerning regularity, from \cref{equation:tspsc,equation:moc} one can infer for some $\beta \in (0, \alpha)$ that
	\begin{equation}\label{equation:regularitywf}
		w, f \in \xCzero([0,1];\xCn{{m-1,\beta}}(\overline{\mathcal{E}};\mathbb{R}^2))  \cap \xLinfty((0,1);\xCn{{m-1,\alpha}}(\overline{\mathcal{E}};\mathbb{R}^2)) .
	\end{equation}
	Let us emphasize that $f$ is supported in $\Lambda\times[0,1]$ and the vortex $w$ given by \eqref{equation:tspsc} meets the constraint $w(\cdot, 1) = 0$.
	Indeed, for $l \in \{1,\dots, \widetilde{L}\}$, one has $\gamma(1-t_l) = 0$, which implies $w(\cdot, 1) = 0$. As the derivative $\xdx{\gamma}/\xdx{t}(t-t_l)$ vanishes (due to \Cref{lemma:nu} and the behavior of $\xmcal{Y}^*$) at all instances $t$ where~$e_l(\cdot,t)$ is nonzero outside of $\Lambda$, it holds
	\begin{equation}\label{equation:suppfeq}
		\operatorname{supp}(f) \subset \omega \times (0,1).
	\end{equation}

	\paragraph{Step 2. Assembly.} When $\xvec{H}_0 \neq \xsym{0}$, let $j^{\pm}$ be the functions determined in \eqref{equation:IdealControl:constructionf:j+-}. Otherwise, set $j^{\pm} \coloneq w$ with $w$ from \eqref{equation:solftcw}. The definition of $\xmcal{F}(\widetilde{\xvec{z}}^+, \widetilde{\xvec{z}}^-)$ is now completed by obtaining the profiles $\xvec{z}^{\pm} \in \bigcup_{\beta \in (0,\alpha)}\xCzero([0,1];\xCn{{m,\beta}}_{*}(\overline{\mathcal{E}}; \mathbb{R}^2))$ as the solutions to
	\begin{equation}\label{equation:IdealControl:divcurlop}
		\begin{aligned}
			\begin{cases}
				\xwcurl{\xvec{z}^{\pm}} = j^{\pm} & \mbox{ in }  \mathcal{E}\times[0,1],\\
				\xdiv{\xvec{z}^{\pm}} = 0 & \mbox{ in }  \mathcal{E}\times[0,1],\\
				\xvec{z}^{\pm} \cdot \xvec{n} = 0 & \mbox{ on }  \partial \mathcal{E}\times[0,1],\\
				\int_{\mathcal{E}} \xvec{z}^{\pm}(\xvec{x}, t) \cdot \xvec{Q}^{\#}(\xvec{x}) \, \xdx{\xvec{x}} = \aleph(t) + \int_{\mathcal{E}} \xvec{y}^*(\xvec{x}, t) \cdot \xvec{Q}^{\#}(\xvec{x}) \, \xdx{\xvec{x}} & \mbox{ for } t \in [0,1],
			\end{cases}
		\end{aligned}
	\end{equation}
	where $\aleph$ is the function from \eqref{equation:alephdefinition}. Namely, we consider for $t \in [0,1]$ the Poisson problems
	\[
		- \Delta A^{\pm}(\cdot, t) = j^{\pm}(\cdot, t) \, \mbox{ in } \mathcal{E}, \quad A^{\pm}(\cdot, t) = 0 \, \mbox{ on } \partial \mathcal{E}
	\]
	and define
	\begin{equation}\label{equation:defF}
		\xvec{z}^{\pm}(\cdot,t) \coloneq \xnab^{\perp} A^{\pm}(\cdot, t) +  \left(\aleph(t) + \int_{\mathcal{E}} \xvec{y}^*(\xvec{x}, t) \cdot \xvec{Q}^{\#}(\xvec{x}) \, \xdx{\xvec{x}}\right) \xvec{Q}^{\#}.
	\end{equation}
	\subsubsection{Equations satisfied by a fixed point}\label{subsubsection:sfp}
	We anticipate the analysis of \Cref{subsection:existencefixedpoint} below, which provides the existence of a fixed point of the map~$\xmcal{F}$ defined via \eqref{equation:defF}, and where it is shown that the respective iterations are uniformly bounded in $\xLinfty((0,1);\xCn{{m,\alpha}}(\overline{\mathcal{E}};\mathbb{R}^2))^2$. Thus, let us for the moment being assume that
	\[
		\xmcal{F}(\xvec{v}^+, \xvec{v}^-) = (\xvec{v}^+, \xvec{v}^-)
	\]
	holds for elements
	\begin{equation}\label{equation:regvpm}
		\xvec{v}^+, \xvec{v}^- \in \xCzero([0,1];\xCn{{1,\alpha}}(\overline{\mathcal{E}};\mathbb{R}^2)) \cap \xLinfty((0,1);\xCn{{m,\alpha}}(\overline{\mathcal{E}};\mathbb{R}^2)).
	\end{equation}
	Recalling that $\xvec{v}^{\pm}$ are Elsasser variables, the associated velocity and magnetic field are recovered via
	\begin{equation}\label{equation:devVH}
		\xvec{V} \coloneq \frac{\xvec{v}^++\xvec{v}^-}{2}, \quad \xvec{H} \coloneq \frac{\xvec{v}^+-\xvec{v}^-}{2}.
	\end{equation}
	We show now that there exist a pressure~$P$ and a physically localized control~$\xsym{\Xi}$ such that $(\xvec{V}, \xvec{H}, P, \xsym{\Xi})$ obey
	\begin{itemize}
		\item the incompressible ideal MHD system \eqref{equation:MHD00unforcedlargedomain} when $\xvec{H}_0 \neq \xsym{0}$,
		\item the incompressible Euler equations \eqref{equation:Eulernctrl} when $\xvec{H}_0 = \xsym{0}$.
	\end{itemize}
	The force
	\[
		\xsym{\Xi} \in \xCzero([0,1];\xCn{{m-2,\alpha}}(\overline{\mathcal{E}};\mathbb{R}^2))  \cap \xLinfty((0,1);\xCn{{m-1,\alpha}}(\overline{\mathcal{E}};\mathbb{R}^2))
	\]
	will split into a control $\xvec{F}$, acting effectively on $\xwcurl{\xvec{V}}$, and a spatially smooth profile~$\xvec{R}$ that corrects the first cohomology projection of the velocity equation. The subsequent analysis eventually leads to
	\begin{equation*}\label{equation:Xidev}
		\xsym{\Xi} = \xvec{F} + \xvec{R}, \quad \operatorname{supp}(\xsym{\Xi}) \subset \omegaup\times[0,1], \quad \xwcurl{\xvec{F}} = f, \quad  \xwcurl{\xvec{R}} = 0.
	\end{equation*}
	\paragraph{The force $\xvec{F}$.} When $\xvec{H}_0 \neq \xsym{0}$, we simply take $\xvec{F} = \xsym{0}$. In the case $\xvec{H}_0 = \xsym{0}$, the vorticity control $f$ from \eqref{equation:solftcw} is active; using the limiting equation for the vorticity, and Sobolev embeddings (\cf~\Cref{remark:regularity} below), one can further specify the regularity of~$f$. In view of \cref{equation:tspsc,equation:regularitywf,equation:suppfeq}, one obtains
	\[
		\xvec{F} \in \xCzero([0,1];\xCn{{m-2,\alpha}}(\overline{\mathcal{E}};\mathbb{R}^2))  \cap \xLinfty((0,1);\xCn{{m-1,\alpha}}(\overline{\mathcal{E}};\mathbb{R}^2))
	\] 
	by integrating the relation $\xwcurl{\xvec{F}} = f$ on a rectangle holding the spatial support of $f$ (see \cite[Appendix]{CoronMarbachSueur2020} for explicit formulas, which involve the squares $\xD_1,\dots,\xD_{\widetilde{L}}$) such that 
	\[
		\xwcurl{\xvec{F}} = f, \quad \operatorname{supp}(\xvec{F}) \subset \omegaup\times(0,1).
	\]

	\paragraph{The force $\xvec{R}$.} Since $(\xvec{V}, \xvec{H})$ arise via \eqref{equation:devVH} from the div-curl systems~\eqref{equation:IdealControl:divcurlop} with prescribed $\xZ(\mathcal{E})$-projections, one has 
	\begin{equation*}
		\xnab^{\perp}(\xvec{H}\wedge\xvec{V}) = (\xvec{V} \cdot \xsym{\nabla}) \xvec{H} - (\xvec{H} \cdot \xsym{\nabla}) \xvec{V}, \quad \int_{\mathcal{E}} \xvec{H}(\xvec{x},\cdot) \cdot \xvec{Q}^{\#}(\xvec{x}) \, \xdx{\xvec{x}} = 0.
	\end{equation*}
	Known vector calculus identities, the boundary conditions in \eqref{equation:IdealControl:divcurlop}, and arguments similar to \eqref{equation:consv}, subsequently yield
	\begin{gather*}
		\xdiv{\left( (\xvec{V} \cdot \xsym{\nabla}) \xvec{H} - (\xvec{H} \cdot \xsym{\nabla}) \xvec{V}\right)} = 0, \quad \left( (\xvec{V} \cdot \xsym{\nabla}) \xvec{H} - (\xvec{H} \cdot \xsym{\nabla}) \xvec{V}\right) \cdot \xvec{n} = 0,\\
		\int_{\mathcal{E}}\left(\partial_t \xvec{H} + (\xvec{V} \cdot \xsym{\nabla}) \xvec{H} - (\xvec{H} \cdot \xsym{\nabla}) \xvec{V} \right)(\xvec{x}, \cdot) \cdot \xvec{Q}^{\#}(\xvec{x}) \, \xdx{\xvec{x}} = 0.
	\end{gather*}
	Accordingly, 
	\begin{equation}\label{equation:MHD00SimplyConnectedExtendedC}
		\begin{cases}
			\xwcurl{\left(\partial_t \xvec{V} + (\xvec{V} \cdot \xsym{\nabla}) \xvec{V} - (\xvec{H} \cdot \xsym{\nabla})\xvec{H} + \xsym{\nabla} \widetilde{P} - \xvec{F}\right)} = 0 & \mbox{ in } \mathcal{E} \times (0,1),\\
			\xdiv{\left(\partial_t \xvec{V} + (\xvec{V} \cdot \xsym{\nabla}) \xvec{V} - (\xvec{H} \cdot \xsym{\nabla})\xvec{H} + \xsym{\nabla} \widetilde{P} - \xvec{F}\right)} = 0 & \mbox{ in } \mathcal{E} \times (0,1),\\
			\left(\partial_t \xvec{V} + (\xvec{V} \cdot \xsym{\nabla}) \xvec{V} - (\xvec{H} \cdot \xsym{\nabla})\xvec{H} + \xsym{\nabla} \widetilde{P} - \xvec{F}\right) \cdot \xvec{n} = 0 & \mbox{ on } \partial \mathcal{E} \times (0,1),\\
			\partial_t \xvec{H} + (\xvec{V} \cdot \xsym{\nabla}) \xvec{H} - (\xvec{H} \cdot \xsym{\nabla}) \xvec{V} = \xsym{0} & \mbox{ in } \mathcal{E} \times (0,1),
		\end{cases}
	\end{equation}
	where the pressure
	\[
		\widetilde{P} \in \xCzero([0,1];\xCn{{1,\alpha}}(\overline{\mathcal{E}};\mathbb{R})) \cap \xLinfty((0,1);\xCn{{m,\alpha}}(\overline{\mathcal{E}};\mathbb{R}))
	\]
	is governed by an elliptic Neumann problem for which classical Schauder estimates (\cf~\cite{Nardi2014}) are available, \ie,
	\begin{equation}\label{equation:NeumannPtilde1}
			\begin{cases}
			\Delta \widetilde{P} = \xdiv{\left(\xvec{F} - (\xvec{V} \cdot \xsym{\nabla}) \xvec{V} + (\xvec{H} \cdot \xsym{\nabla})\xvec{H}\right)}  & \mbox{ in } \mathcal{E} \times (0,1),\\
			\xnab \widetilde{P} \cdot \xvec{n} = \left(\xvec{F} - (\xvec{V} \cdot \xsym{\nabla}) \xvec{V} + (\xvec{H} \cdot \xsym{\nabla})\xvec{H}\right) \cdot \xvec{n} & \mbox{ on } \partial \mathcal{E} \times (0,1).
		\end{cases}
	\end{equation}
	So far, the velocity equation for $\xvec{V}$ is not seen to be satisfied. In general, there exists a nontrivial function $\rho\colon[0,1] \longrightarrow \mathbb{R}$ with
	\begin{equation}\label{equation:ewqrr}
		\partial_t \xvec{V} + (\xvec{V} \cdot \xsym{\nabla}) \xvec{V} - (\xvec{H} \cdot 	\xsym{\nabla})\xvec{H} + \xsym{\nabla} \widetilde{P} - \xvec{F} = \rho \xvec{Q}^{\#}.
	\end{equation}
	Nevertheless, the right-hand side $\rho \xvec{Q}^{\#}$ in \eqref{equation:ewqrr} can be absorbed by $\xvec{R}$ and a pressure gradient. To this end, as the difference $\overline{\mathcal{E}}\setminus\Lambda$ is simply-connected, we begin with selecting any smooth $\widetilde{h}$ with $\xnab \widetilde{h} = \rho \xvec{Q}^{\#}$ in $(\overline{\mathcal{E}}\setminus\Lambda)\times[0,1]$. This function is then extended to the entire domain via
	\[
		\widehat{h}(\xvec{x},t) \coloneq \begin{cases}
			\widetilde{h}(\xvec{x},t) & \mbox{ in }  (\overline{\mathcal{E}}\setminus\Lambda)\times[0,1],\\
			0 & \mbox{ otherwise.}
		\end{cases}
	\]
	To regularize $\widehat{h}$, let $\widehat{\chi} \in \xCinfty_0(\mathbb{R}^2; \mathbb{R})$ be a cutoff with $\widehat{\chi} = 1$ in a neighborhood of $\overline{\mathcal{E}}\setminus\omegaup$ and $\widehat{\chi} = 0$ near $\Lambda$. Subsequently, we define for each $(\xvec{x},t) \in \overline{\mathcal{E}}\times[0,1]$ the smooth potential
	\begin{equation*}\label{equation:hdev}
		h(\xvec{x},t) \coloneq \widehat{\chi}(\xvec{x}) \widehat{h}(\xvec{x},t).
	\end{equation*}
	As a result, the definition of $\xsym{\Xi}$ satisfying $\operatorname{supp}(\xsym{\Xi}) \subset \omegaup\times[0,1]$ is completed by introducing the profiles
	\begin{equation}\label{equation:Rdev}
		\xvec{R} \coloneq \rho \xvec{Q}^{\#} - \xnab h, \quad P \coloneq \widetilde{P}-h, \quad \xsym{\Xi} \coloneq \xvec{F} + \xvec{R},
	\end{equation}
	which ensure in combination with \eqref{equation:MHD00SimplyConnectedExtendedC} that
	\begin{equation}\label{equation:svp}
		\partial_t \xvec{V} + (\xvec{V} \cdot \xsym{\nabla}) \xvec{V} - (\xvec{H} \cdot \xsym{\nabla})\xvec{H} + \xsym{\nabla} P  = \xsym{\Xi}.
	\end{equation}

	\begin{rmrk}\label{remark:regularity}
		The regularity stated in \eqref{equation:regvpm} can be specified further by invoking the equations for $(\xvec{V}, \xvec{H}, P)$, \ie, \cref{equation:MHD00SimplyConnectedExtendedC,equation:NeumannPtilde1,equation:ewqrr,equation:Rdev,equation:svp}, together with the Sobolev embedding
		\[
			\xWn{{1,\infty}}((0,1);\xCn{{m-1,\alpha}}(\overline{\mathcal{E}};\mathbb{R}^r)) \hookrightarrow \xCzero((0,1);\xCn{{m-1,\alpha}}(\overline{\mathcal{E}}; \mathbb{R}^r)).
		\]
		For instance, one finds that $\xvec{V}$, $\xvec{H}$ with $r = 2$ and $P$ with $r = 1$ belong respectively to the space $\xCzero((0,1);\xCn{{m-1,\alpha}}(\overline{\mathcal{E}}; \mathbb{R}^r))$.
	\end{rmrk}

	\subsubsection{Existence of a unique fixed point}\label{subsection:existencefixedpoint}
	The remainder of \Cref{theorem:ret}'s proof is organized in three main steps. 1) It is first demonstrated that~$\xmcal{F}$ constitutes a self-map of $\xX_{\delta^*,k}$, given~$k > 0$ sufficiently large and assuming $\|\xvec{z}^{\pm}_0\|_{m,\alpha,\mathcal{E}}$ small. 2) For a large integer~$n\geq1$, the~$n$-fold composition~$\xmcal{F}^n$ is seen to form a contraction with respect to the Banach space
	\begin{equation}\label{equation:spaceY}
		\begin{gathered}
			\xY \coloneq \xCzero([0,1];\xCn{{1,\alpha}}(\overline{\mathcal{E}};\mathbb{R}^2)) \times \xCzero([0,1];\xCn{{1,\alpha}}(\overline{\mathcal{E}};\mathbb{R}^2)), \\
			\|(\xvec{f},\xvec{g})\|_{\xY} \coloneq \max\limits_{t \in [0,1]} \left( \|\xvec{f}\|_{1,\alpha,\mathcal{E}}(t) + \|\xvec{g}\|_{1,\alpha,\mathcal{E}}(t) \right).
		\end{gathered}
	\end{equation}
	In particular, this allows to conclude that~$\xmcal{F}$ and $\xmcal{F}^n$ uniquely extend to continuous self-maps possessing a fixed point in the $\|\cdot\|_{\xY}$-closure of~$\xX_{\delta^*,k}$. Indeed, one can apply Banach's fixed point theorem (\cf~\cite[Theorem 3]{Fernandez-CaraSantosSouza2016}) to $\xmcal{F}^n$, and by a uniqueness argument verify that $\xmcal{F}$ and $\xmcal{F}^n$ have the same fixed point. The desired solution~$(\xvec{V}, \xvec{H}, P)$ to \eqref{equation:MHD00unforcedlargedomain} and a suitable control $\xsym{\Xi}$ are afterwards found by following the recipe already described in \Cref{subsubsection:sfp}. 3) It is explained in \Cref{subsection:initialdatabounds} how the initial data bounds $0 < \delta_0 < \dots < \delta_K$ can be selected.
	
	The following estimate follows directly from the solution representation \eqref{equation:moc} for the involved transport equations (such estimates have also been stated in \cite{Bardos_EulerEq}, \cite[Lemma 2]{Fernandez-CaraSantosSouza2016}).
	\begin{lmm}\label{lemma:transptest}
		Let $l \in \mathbb{N}$, $M > 0$, $T > 0$, and suppose that $v$ solves in $\mathcal{E}\times(0,T)$ the transport problem
		\[
			\partial_t v + (\xvec{z} \cdot \xnab) v = g \in \xCzero([0,T];\xCn{{l,\alpha}}(\overline{\mathcal{E}};\mathbb{R})),
		\]
		where the drift field obeys
		\begin{gather*}
			\xvec{z} \in \xCzero([0,T];\xCn{{l+1}}(\overline{\mathcal{E}};\mathbb{R}^2)) \cap \xLinfty((0,T);\xCn{{l+1,\alpha}}(\overline{\mathcal{E}};\mathbb{R})), \\ \xvec{z} \cdot \xvec{n} = 0 \mbox{ on } \partial\mathcal{E}\times(0,T), \quad \max\limits_{t \in [0,T]}\|\xvec{z}\|_{l+1,\alpha,\mathcal{E}}(t) < M.
		\end{gather*}
		Then, there is a constant $\widetilde{C} = \widetilde{C}(l,\alpha, M, T) > 0$, independent of $v$, $\xvec{z}$, and $g$, such that
		\[
			\forall t \in [0,T] \colon \,  \|v\|_{l,\alpha,\mathcal{E}}(t) \leq \widetilde{C} \left( \int_0^t  \|g\|_{l,\alpha,\mathcal{E}}(r) \, \xdx{r} + \|v\|_{l,\alpha,\mathcal{E}}(0)\right).
		\]
	\end{lmm}

	\paragraph{Convention.}
		We only present the case $\xvec{H}_0 \neq \xsym{0}$. If $\xvec{H}_0 = \xsym{0}$, one has only to consider the vorticity equation of $2$D perfect fluids, and all efforts reduce to analyzing $(e_l)_{l\in\{1,\dots,\widetilde{L}\}}$ in~\eqref{equation:tspsc} similarly as $j^{\pm}$ below  (one can also use the estimates known from \cite{Coron1996EulerEq,Glass2000,Fernandez-CaraSantosSouza2016}).

	\paragraph{Notation.} Constants of the form $C > 0$ are generic and may change from line to line during the estimates, depending only on fixed objects such as $m$, $\alpha$, $\xvec{y}^*$, $\nu$, $\xvec{Q}^{\#}$, and the domain.

	\paragraph{$\xmcal{F}$ is a self-map on $\xX_{\delta^*,k}$.}   
	Let $(\widetilde{\xvec{z}}^+,\widetilde{\xvec{z}}^-) \in \xX_{\delta^*,k}$ and $(\xvec{z}^+,\xvec{z}^-) = \xmcal{F}(\widetilde{\xvec{z}}^+, \widetilde{\xvec{z}}^-)$ be obtained from $j^{\pm}$ via \eqref{equation:IdealControl:divcurlop}. 
	Hence, the differences~$\xvec{z}^{\pm} - \xvec{y}^{*}$ solve for $t \in [0, 1]$ the div-curl problems
	\begin{equation*}\label{equation:IdealControl:divcurldiffz+-y}
		\begin{aligned}
			\begin{cases}
				\xwcurl{\left(\xvec{z}^{\pm}(\cdot,t) -\xvec{y}^{*}(\cdot,t)\right)} = j^{\pm}(\cdot,t) & \mbox{ in }  \mathcal{E},\\
				\xdiv{\left(\xvec{z}^{\pm}(\cdot,t) -\xvec{y}^{*}(\cdot,t) \right)} = 0 & \mbox{ in }  \mathcal{E},\\
				\left(\xvec{z}^{\pm}(\cdot,t) - \xvec{y}^{*}(\cdot,t)\right)\cdot \xvec{n} = 0 & \mbox{ on }  \partial \mathcal{E},\\
				\int_{\mathcal{E}} \left(\xvec{z}^{\pm}(\xvec{x},t) -\xvec{y}^{*}(\xvec{x},t)\right) \cdot \xvec{Q}^{\#}(\xvec{x}) \, \xdx{\xvec{x}} = \aleph(t),
			\end{cases}
		\end{aligned}
	\end{equation*}
	Moreover, the definition of $\aleph$ in \eqref{equation:alephdefinition} and elliptic regularity yield
	\begin{equation}\label{equation:alephbound}
		\begin{gathered}
			\|\aleph\|_{\xCzero([0,1];\mathbb{R})} \leq \left| \int_{\mathcal{E}} \xvec{V}_0(\xvec{x}) \cdot \xvec{Q}^{\#}(\xvec{x}) \, \xdx{\xvec{x}}\right| \leq C\left(\|\xvec{z}^+_0\|_{m, \alpha, \mathcal{E}} + \|\xvec{z}^-_0\|_{m, \alpha, \mathcal{E}}\right),\\
			\|\xvec{z}^{\pm}-\xvec{y}^{*}\|_{m,\alpha,\mathcal{E}}(t) \leq C\left(\| j^{\pm} \|_{m-1,\alpha,\mathcal{E}}(t) + |\aleph(t)|\right), \quad t \in [0,1].
		\end{gathered}
	\end{equation}
	According to \eqref{equation:alephbound}, the membership $(\xvec{z}^+,\xvec{z}^-) \in \xX_{\delta^*,k}$ can in what follows be verified by means of showing that
	\begin{equation}\label{equation:desbd}
		\max\limits_{t \in [0,1]} CW_k(t)\left(\|j^{+}\|_{m-1,\alpha,\mathcal{E}}(t) + \|j^{-}\|_{m-1,\alpha,\mathcal{E}}(t) + |\aleph(t)|\right) < \delta^*.
	\end{equation}
	To begin with, after applying \Cref{lemma:transptest} to the equations for $j^{\pm}$ in \eqref{equation:IdealControl:constructionf:j+-}, one has for each $t \in [0,1]$ the estimate
	\begin{multline}\label{equation:wkjpmest}
		W_k(t)\|j^{\pm}\|_{m-1,\alpha,\mathcal{E}}(t) \\
		 \leq C W_k(t)\left(\int_0^t \|G^{\pm}(\widetilde{\xvec{z}}^+,\widetilde{\xvec{z}}^-) \|_{m-1,\alpha,\mathcal{E}}(s) \, \xdx{s} +  \|j^{\pm}_0\|_{m-1,\alpha,\mathcal{E}}\right),
	\end{multline}
	where we used the uniform bound
	\begin{equation*}
		\begin{aligned}
		 \max_{t\in[0,1]}\|\widetilde{\xvec{z}}^{\pm}\|_{m,\alpha,\mathcal{E}}(t) \leq \nu + \max_{t\in[0,1]} \|\xvec{y}^*\|_{m,\alpha,\mathcal{E}}(t) \leq C,
		\end{aligned}
	\end{equation*}
	which follows from the definition of $\xX_{\delta^*,k}$ and $\delta^* < \nu$ (noting that $\nu$ is fixed from the start).
	Therefore, resorting to $W_k$'s description and employing in \eqref{equation:wkjpmest} the representations of $G^{\pm}$ from~\eqref{equation:reprGpm}, we arrive at
	\begin{multline}\label{equation:estimatee+-l}
		W_k(t)\|j^{\pm}\|_{m-1,\alpha,\mathcal{E}}(t) \\
		\begin{aligned}
			& \leq C W_k(t)\left(\int_0^t \|\widetilde{\xvec{z}}^+ -\widetilde{\xvec{z}}^-\|_{m,\alpha,\mathcal{E}}(s) \, \xdx{s} +  \|j^{\pm}_0\|_{m-1,\alpha,\mathcal{E}}\right)\\
			& \leq C W_k(t)\int_0^t \frac{W_k(s)}{W_k(s)}\left(\|\widetilde{\xvec{z}}^+ -\xvec{y}^*\|_{m,\alpha,\mathcal{E}}(s) + \|\widetilde{\xvec{z}}^- -\xvec{y}^*\|_{m,\alpha,\mathcal{E}}(s)\right) \, \xdx{s}\\
			& \quad +  C W_k(t) \|j^{\pm}_0\|_{m-1,\alpha,\mathcal{E}}\\
			& \leq C W_k(t)\left(\int_0^t \frac{2\nu}{W_k(s)} \, \xdx{s} +  \|j^{\pm}_0\|_{m-1,\alpha,\mathcal{E}}\right)\\
			& \leq C \left(\frac{1}{2}+\frac{t}{8}\right)^{-k}\left(\int_0^t \left(\frac{1}{2}+\frac{s}{8}\right)^k \, \xdx{s} +  \|j^{\pm}_0\|_{m-1,\alpha,\mathcal{E}}\right)\\
			& \leq \frac{5C}{k+1} + C W_k(t) \|j^{\pm}_0\|_{m-1,\alpha,\mathcal{E}}\\
			& < \delta^*,
		\end{aligned}
	\end{multline}
	where the last estimate holds provided that
	\begin{itemize}
		\item $k > 0$ is taken large in dependence on  $m$, $\alpha$, $\xvec{y}^*$, $\delta^*$, $\nu$, and the domain;
		\item the norms $\|\xvec{z}^{\pm}_0\|_{m,\alpha,\mathcal{E}}$ of the initial states are assumed sufficiently small depending on $k$, $m$, $\alpha$, $\xvec{y}^*$, $\delta^*$, $\nu$, and the domain.
	\end{itemize}
	Since $\aleph(t)$ is for all $t \in [0,1]$ bounded by the initial data (\cf~\eqref{equation:alephbound}), after possibly reducing  $\|\xvec{z}^{\pm}_0\|_{m,\alpha,\mathcal{E}}$ further, the estimate \eqref{equation:estimatee+-l} implies \eqref{equation:desbd}.

	\paragraph{$\xmcal{F}^n$ is a contraction.} In order to determine a large~$n \in \mathbb{N}$ such that the $n$-fold composition $\xmcal{F}^n = \xmcal{F}\circ \dots \circ\xmcal{F}$ forms a contraction with respect to $\|\cdot\|_{\xY}$, as defined in \eqref{equation:spaceY}, we start with arbitrary elements
	\[
		(\widetilde{\xvec{z}}^{+,1},\widetilde{\xvec{z}}^{-,1}), (\widetilde{\xvec{z}}^{+,2},\widetilde{\xvec{z}}^{-,2}) \in \xX_{\delta^*,k}, \quad \widetilde{\xvec{Z}}^{\pm} \coloneq \widetilde{\xvec{z}}^{\pm,1} - \widetilde{\xvec{z}}^{\pm,2}.
	\]
	For $i \in \{1,2\}$, the definition of $\xmcal{F}$ gives rise to the associated functions
	\[
		j^{\pm,i}, \quad J^{\pm} \coloneq j^{\pm,1}-j^{\pm,2}, \quad (\xvec{z}^{+,i}, \xvec{z}^{-,i}) \coloneq \xmcal{F}(\widetilde{\xvec{z}}^{+,i}, \widetilde{\xvec{z}}^{-,i}), \quad \xvec{Z}^{\pm} \coloneq \xvec{z}^{\pm,1} - \xvec{z}^{\pm,2}.
	\] 
	In particular, the differences $J^{\pm}$ solve in $\mathcal{E}\times(0,1)$ the initial value problems
	\begin{equation}\label{equation:eqJpm}
		\begin{cases}
			\partial_t J^{\pm} + (\widetilde{\xvec{z}}^{\mp,1} \cdot \xsym{\nabla}) J^{\pm} = - (\widetilde{\xvec{Z}}^{\mp} \cdot \xsym{\nabla}) j^{\pm,2} + G^{\pm}(\widetilde{\xvec{z}}^{+,1}, \widetilde{\xvec{z}}^{-,1}) - G^{\pm}(\widetilde{\xvec{z}}^{+,2}, \widetilde{\xvec{z}}^{-,2}),\\
			J^{\pm}(\cdot, 0) = 0.
		\end{cases}
	\end{equation}
	Because any pair $(\xvec{v}^+, \xvec{v}^-) \in \xX_{\delta^*,k}$ satisfies
	\[
		\|\xvec{v}^{\pm}\|_{1,\alpha,\mathcal{E}}(t) \leq \nu + \|\xvec{y}^*\|_{1,\alpha,\mathcal{E}}(t) \leq C,
	\]
	the representations for $G^{\pm}$ in \eqref{equation:reprGpm} yield
	\begin{equation}\label{equation:diffGpm}
		\|G^{\pm}(\widetilde{\xvec{z}}^{+,1}, \widetilde{\xvec{z}}^{-,1}) - G^{\pm}(\widetilde{\xvec{z}}^{+,2}, \widetilde{\xvec{z}}^{-,2})\|_{0,\alpha,\mathcal{E}} \leq C \left(\|\widetilde{\xvec{Z}}^{+}\|_{1,\alpha,\mathcal{E}} + \|\widetilde{\xvec{Z}}^{-}\|_{1,\alpha,\mathcal{E}}\right).
	\end{equation}
	Furthermore, the differences $\xvec{Z}^{\pm}$ satisfy for all $t \in [0,1]$ the div-curl system with trivial $\xZ(\mathcal{E})$-projections
	\begin{equation*}
		\begin{aligned}
			\begin{cases}
				\xwcurl{\xvec{Z}^{\pm}(\cdot,t)} = J^{\pm}(\cdot,t) & \mbox{ in }  \mathcal{E},\\
				\xdiv{\xvec{Z}^{\pm}(\cdot,t)} = 0 & \mbox{ in }  \mathcal{E},\\
				\xvec{Z}^{\pm}(\cdot,t)\cdot \xvec{n} = 0 & \mbox{ on }  \partial \mathcal{E},\\
				\int_{\mathcal{E}} \xvec{Z}^{\pm}(\xvec{x},t) \cdot \xvec{Q}^{\#}(\xvec{x}) \, \xdx{\xvec{x}} = 0.
			\end{cases}
		\end{aligned}
	\end{equation*}
	Resorting to elliptic regularity estimates for $\xvec{Z}^{\pm}$, applying \Cref{lemma:transptest} to \eqref{equation:eqJpm} with vanishing initial data, and invoking \eqref{equation:diffGpm}, we get
	\begin{equation}\label{equation:diffesta}
		\begin{aligned}
			\| \xvec{Z}^{+} \|_{1,\alpha,\mathcal{E}}(t) + \| \xvec{Z}^{-} \|_{1,\alpha,\mathcal{E}}(t) & \leq C \left(\|J^{+}\|_{0,\alpha,\mathcal{E}}(t) + \| J^{-} \|_{0,\alpha,\mathcal{E}}(t)\right)\\
			& \leq C \int_0^t \left(\|\widetilde{\xvec{Z}}^{+}\|_{1,\alpha,\mathcal{E}}(s) + \|\widetilde{\xvec{Z}}^{-}\|_{1,\alpha,\mathcal{E}}(s)\right)  \, \xdx{s}.
		\end{aligned}
	\end{equation}
	In \eqref{equation:diffesta}, we made use of the standing assumption $m \geq 2$, which provides the uniform bound (\cf~\eqref{equation:estimatee+-l})
	\begin{equation*}
		\|j^{+,2}\|_{1,\alpha,\mathcal{E}} + \|j^{-,2}\|_{1,\alpha,\mathcal{E}} \leq C.
	\end{equation*}
	Consequently, for all $t \in [0,1]$, we can infer from \eqref{equation:diffesta} that
	\begin{equation}\label{equation:indbase}
		\begin{aligned}
			\|\xmcal{F}(\widetilde{\xvec{z}}^{+,1}, \widetilde{\xvec{z}}^{-,1})-\xmcal{F}(\widetilde{\xvec{z}}^{+,2}, \widetilde{\xvec{z}}^{-,2})\|_{\widetilde{\xY}}(t) & = \| \xvec{Z}^{+} \|_{1,\alpha,\mathcal{E}}(t) + \| \xvec{Z}^{-} \|_{1,\alpha,\mathcal{E}}(t) \\
			& \leq C \int_0^t \|(\widetilde{\xvec{z}}^{+,1}, \widetilde{\xvec{z}}^{-,1})-(\widetilde{\xvec{z}}^{+,2}, \widetilde{\xvec{z}}^{-,2})\|_{\widetilde{\xY}}(s)  \, \xdx{s} \\
			& \leq Ct \|(\widetilde{\xvec{z}}^{+,1}, \widetilde{\xvec{z}}^{-,1})-(\widetilde{\xvec{z}}^{+,2}, \widetilde{\xvec{z}}^{-,2})\|_{\xY},
		\end{aligned}
	\end{equation}
	where $\|\cdot\|_{\widetilde{\xY}}$ is the norm
	\[
		\|(\xvec{f},\xvec{g})\|_{\widetilde{\xY}} \coloneq  \| \xvec{f} \|_{1,\alpha,\mathcal{E}} + \| \xvec{g} \|_{1,\alpha,\mathcal{E}}.
	\]
	
	Let us repeat the arguments that led to \eqref{equation:diffesta} and \eqref{equation:indbase}, but now with $(\widetilde{\xvec{z}}^{+,i}, \widetilde{\xvec{z}}^{-,i})$ being replaced by $\xmcal{F}(\widetilde{\xvec{z}}^{+,i}, \widetilde{\xvec{z}}^{-,i})$ for $i \in \{1,2\}$. Owing to the estimate \eqref{equation:indbase}, it follows for all $t \in [0,1]$ that
	\begin{multline*}\label{equation:secondstep}
		\|\xmcal{F}^2(\widetilde{\xvec{z}}^{+,1}, \widetilde{\xvec{z}}^{-,1})-\xmcal{F}^2(\widetilde{\xvec{z}}^{+,2}, \widetilde{\xvec{z}}^{-,2})\|_{\widetilde{\xY}}(t) = \|\xmcal{F}(\xmcal{F}(\widetilde{\xvec{z}}^{+,1}, \widetilde{\xvec{z}}^{-,1}))-\xmcal{F}(\xmcal{F}(\widetilde{\xvec{z}}^{+,2}, \widetilde{\xvec{z}}^{-,2}))\|_{\widetilde{\xY}}(t)\\
		\begin{aligned}
			& \leq C \int_0^t \|\xmcal{F}(\widetilde{\xvec{z}}^{+,1}, \widetilde{\xvec{z}}^{-,1})-\xmcal{F}(\widetilde{\xvec{z}}^{+,2}, \widetilde{\xvec{z}}^{-,2})\|_{\widetilde{\xY}}(s)  \, \xdx{s} \\
			& \leq C^2 \|(\widetilde{\xvec{z}}^{+,1}, \widetilde{\xvec{z}}^{-,1})-(\widetilde{\xvec{z}}^{+,2}, \widetilde{\xvec{z}}^{-,2})\|_{\xY} \int_0^t s \, \xdx{s}\\
			& \leq \frac{(Ct)^2}{2} \|(\widetilde{\xvec{z}}^{+,1}, \widetilde{\xvec{z}}^{-,1})-(\widetilde{\xvec{z}}^{+,2}, \widetilde{\xvec{z}}^{-,2})\|_{\xY}.
		\end{aligned}
	\end{multline*}
	Via induction over $n \in \mathbb{N}$, one likewise achieves
	\begin{equation*}
		\|\xmcal{F}^n(\widetilde{\xvec{z}}^{+,1}, \widetilde{\xvec{z}}^{-,1})-\xmcal{F}^n(\widetilde{\xvec{z}}^{+,2}, \widetilde{\xvec{z}}^{-,2})\|_{\widetilde{\xY}}(t)
		\leq \frac{(Ct)^n}{n!} \|(\widetilde{\xvec{z}}^{+,1}, \widetilde{\xvec{z}}^{-,1})-(\widetilde{\xvec{z}}^{+,2}, \widetilde{\xvec{z}}^{-,2})\|_{\xY},
	\end{equation*}
	for all $t \in [0,1]$. Therefore, for each $n \in \mathbb{N}$, we arrive at
	\begin{equation}\label{equation:contrest}
		\|\xmcal{F}^n(\widetilde{\xvec{z}}^{+,1}, \widetilde{\xvec{z}}^{-,1})-\xmcal{F}^n(\widetilde{\xvec{z}}^{+,2}, \widetilde{\xvec{z}}^{-,2})\|_{\xY}
		\leq \frac{C^n}{n!} \|(\widetilde{\xvec{z}}^{+,1}, \widetilde{\xvec{z}}^{-,1})-(\widetilde{\xvec{z}}^{+,2}, \widetilde{\xvec{z}}^{-,2})\|_{\xY}.
	\end{equation}
	As the constant $C > 0$ depends only on fixed objects, the estimate~\eqref{equation:contrest} provides a number $N \geq 1$ for which $\xmcal{F}^n$ with~$n \geq N$ is a contraction.
	
	\subsubsection{Initial data bounds}\label{subsection:initialdatabounds}
	For~$\delta^* \in (0,\nu]$ and sufficiently large $k = k(\delta^*) > 0$, let~$\xmcal{F}_{\delta^*}$ denote the mapping~$\xmcal{F}$ defined on~$\xX_{\delta^*,k}$. To determine $0 < \delta_0 < \dots < \delta_K$, we start with $\delta^* = \nu/3C_*$ and choose a small $\delta_K > 0$ such that $\|\xvec{z}^{\pm}_0\|_{m,\alpha,\mathcal{E}} < \delta_K$ facilitates a fixed point of~$\xmcal{F}_{\nu/3C_*}$ in the~$\|\cdot\|_{\xY}$-closure~of~$\xX_{\nu/3C_*,k(\nu/3C_*)}$ for both cases $\xvec{H}_0 \neq \xsym{0}$ and $\xvec{H}_0 = \xsym{0}$. Next, after selecting $\delta^* = \delta_K/6C_*$, one finds $\delta_{K-1} > 0$ yielding a unique fixed point of~$\xmcal{F}_{\delta_K/6C_*}$ in the $\|\cdot\|_{\xY}$-closure of $\xX_{\delta_K/6C_*, k(\delta_K/6C_*)}$ for $\|\xvec{z}^{\pm}_0\|_{m,\alpha,\mathcal{E}} < \delta_{K-1}$. Subsequently, we obtain $\delta_{K-2}$, then~$\delta_{K-3}$, and, after a finite number of steps, reach a good bound~$\delta_0 > 0$. 
	
	\subsection{Proof of \Cref{theorem:tkloc}}\label{subsection:proofTKLOC}
	\subsubsection{Version 1}\label{subsection:proofA}
	
	We fix $(\xvec{V}, \xvec{H}, P, \xsym{\Xi})$ by applying \Cref{theorem:ret} with initial data $(\xvec{V}_0, \xvec{H}_0) = (\widetilde{\xvec{u}}_0, \widetilde{\xvec{B}}_0)$, emphasizing that
	\begin{equation*}
		\xvec{V}, \xvec{H} \in \xCzero([0,1];\xCn{{m-1,\alpha}}(\overline{\mathcal{E}}; \mathbb{R}^2)) \cap \xLinfty((0,1);\xCn{{m,\alpha}}(\overline{\mathcal{E}}; \mathbb{R}^2)).
	\end{equation*}
	The strategy is now as follows. First, the magnetic field $\xvec{H}$ is split into two localized parts that solve individual induction problems with source terms. Second, to define the controlled trajectory $(\widetilde{\xvec{u}}, \widetilde{\xvec{B}}, \widetilde{p})$, the separated trajectories from the first step are recombined after the time $t = 1/K - K_0$ in a way that deletes the magnetic field in a neighborhood of $\Lambda$; here $K_0\in(0,1/2K)$ is the number introduced above \eqref{equation:beta} such that $\xvec{y}^*(\cdot,t) = \xsym{0}$ for all~$t \in [1/K-2K_0,1/K]$. Third, the property \eqref{equation:suppcondB} is verified, and it is observed that $\xvec{H}$ evolves during $[0, 1/K-K_0]$ in a way that facilitates \eqref{equation:suppcondB3}.

	\paragraph{Step 1. Trajectory splitting.} As shown below \eqref{equation:consvauxe}, one can write $\xvec{H} = \xnab^{\perp} \psi$ with $\psi(\xvec{x}, t) = 0$ for all~$(\xvec{x}, t) \in \partial \mathcal{E} \times [0,1]$.
	Further, let $\{\mu_1, \mu_2\} \subset \xCinfty_0(\mathbb{R}^2;\mathbb{R})$ be the partition of unity from \eqref{equation:partitionofunity} and fix
	\begin{equation*}
		\xvec{H}^1,\xvec{H}^2 \in \xCzero([0,1/K];\xCn{{m-1,\alpha}}(\overline{\mathcal{E}}; \mathbb{R}^2)) \cap \xLinfty((0,1/K);\xCn{{m,\alpha}}(\overline{\mathcal{E}}; \mathbb{R}^2))
	\end{equation*}
	by means of
	\begin{equation}\label{equation:H12_version_b}
		\xvec{H}^j(\xvec{x}, t) \coloneq \xnab^{\perp}(\mu_j(\xvec{x})\psi(\xvec{x}, t)), \quad (\xvec{x}, t) \in \overline{\mathcal{E}} \times [0,1/K], \quad \quad j \in \{1,2\}.
	\end{equation}
	Because $\mu_1$ and $\mu_2$ obey \eqref{equation:partitionofunity}, it holds $\xvec{H} = \xvec{H}^1 + \xvec{H}^2$. Furthermore, in view of \eqref{equation:H12_version_b}, and since $\mu_1\psi = \mu_2\psi = 0$ on $\partial \mathcal{E} \times [0,1]$, one has 
	\begin{equation}\label{equation:divfrcstnatb}
		\xdiv{\xvec{H}^1} = \xdiv{\xvec{H}^2} = 0 \mbox{ in } \mathcal{E}\times[0,1/K], \quad \xvec{H}^1 \cdot \xvec{n} = \xvec{H}^2\cdot \xvec{n} = 0 \mbox{ on } \partial \mathcal{E}\times[0,1/K].
	\end{equation}
	Thus, by employing the equations for~$\xvec{V}$ and~$\xvec{H}$ in \eqref{equation:MHD00unforcedlargedomain}, one can verify that~$\xvec{H}^1$ and~$\xvec{H}^2$ satisfy the respective initial boundary value problems
	\begin{equation}\label{equation:MHD00unforcedtrsp_0_version_b}
		\begin{cases}
			\partial_t \xvec{H}^{j} + (\xvec{V} \cdot \xsym{\nabla}) \xvec{H}^{j} - (\xvec{H}^{j} \cdot \xsym{\nabla}) \xvec{V} = \widehat{\xsym{\eta}}^{j}  & \mbox{ in } \mathcal{E} \times (0,1/K),\\
			\xdiv{\xvec{H}^j} = 0 & \mbox{ in } \mathcal{E} \times (0,1/K),\\
			\xvec{H}^j \cdot \xvec{n} = 0 & \mbox{ on } \partial\mathcal{E} \times (0,1/K),\\
			\xvec{H}^{j}(\cdot, 0)  =  \xnab^{\perp}(\mu_j \psi(\cdot, 0))  & \mbox{ in } \mathcal{E},
		\end{cases}
	\end{equation}
	where the source terms are for $j \in \{1,2\}$ given by
	\begin{equation}\label{equation:etahat}
		\widehat{\xsym{\eta}}^{j}  \coloneq \partial_t \xvec{H}^{j} + (\xvec{V} \cdot \xsym{\nabla}) \xvec{H}^{j} - (\xvec{H}^{j} \cdot \xsym{\nabla}) \xvec{V}.
	\end{equation}
	In addition, from the equation for $\xvec{H}$ in \eqref{equation:MHD00unforcedlargedomain}, and also \cref{equation:H12_version_b,equation:divfrcstnatb,equation:etahat}, it can be inferred that
	\begin{equation}\label{equation:prpetahatj1}
		\begin{gathered}
			\widehat{\xsym{\eta}}^{j} \in \xCzero([0,1/K];\xCn{{m-2,\alpha}}(\overline{\mathcal{E}};\mathbb{R}^2)) \cap \xLinfty((0,1/K);\xCn{{m-1,\alpha}}(\overline{\mathcal{E}};\mathbb{R}^2)),\\
			\operatorname{supp}(\widehat{\xsym{\eta}}^{j}) \subset (\mathcal{O}_1 \cap \mathcal{O}_2) \times [0, 1/K], \quad 
			\widehat{\xsym{\eta}}^{1} + \widehat{\xsym{\eta}}^{2} = 0,\\
			\xdiv{\widehat{\xsym{\eta}}^j} = 0 \mbox{ in } \mathcal{E}, \quad \widehat{\xsym{\eta}}^j \cdot \xvec{n} = 0 \mbox{ on } \partial \mathcal{E}, \quad
			\widehat{\xsym{\eta}}^j = \xnab^{\perp} \widehat{\phi}_j, \quad \widehat{\phi}_j = 0 \mbox{ on } \partial \mathcal{E},
		\end{gathered}
	\end{equation}
	which also involves the identities $(\xvec{V} \cdot \xnab)\xvec{H}^j - (\xvec{H}^j \cdot \xnab) \xvec{V} = \xnab^{\perp} (\xvec{H}^j \wedge \xvec{V})$ for $j \in \{1,2\}$.

	\paragraph{Step 2. Controlled solutions.} Given the previous constructions, we define a controlled trajectory of \eqref{equation:MHD00SimplyConnectedReducedsupp} via
	\begin{equation}\label{equation:controlsdef_version_b}
		\begin{gathered}
			\widetilde{\xvec{u}} \coloneq \xvec{V}, \quad \widetilde{\xvec{B}} \coloneq  \beta\xvec{H}^{1}+\xvec{H}^{2}, \quad \widetilde{p} \coloneq P, \quad
			\widetilde{\xsym{\eta}} \coloneq \xdrv{\beta}{t} \xvec{H}^1 + \beta \widehat{\xsym{\eta}}^1 + \widehat{\xsym{\eta}}^2,\\
			\widetilde{\xsym{\xi}} \coloneq \xsym{\Xi} + (1-\beta)(\xvec{H}^{1} \cdot \xnab) \xvec{H}^{2} + (1-\beta)(\xvec{H}^{2} \cdot \xnab) \xvec{H}^{1}  + (1-\beta^2)(\xvec{H}^{1} \cdot \xnab) \xvec{H}^{1}.
		\end{gathered}
	\end{equation}
	where the form of the controls is dictated by the induction and momentum equations in \eqref{equation:MHD00SimplyConnectedReducedsupp}.
	The proclaimed regularity of $(\widetilde{\xvec{u}}, \widetilde{\xvec{B}}, \widetilde{p}, \widetilde{\xsym{\xi}}, \widetilde{\xsym{\eta}})$ is inherited from~$(\xvec{V}, \xvec{H}, P, \xsym{\Xi})$ (\cf~\Cref{theorem:ret}) and the properties of~$(\widehat{\xsym{\eta}}^{1}, \widehat{\xsym{\eta}}^{2})$ from~\eqref{equation:etahat}; in particular, 
	\begin{gather*}
		\widetilde{\xvec{u}}, \widetilde{\xvec{B}} \in \xCn{0}([0,1/K];\xCn{{m-1,\alpha}}(\overline{\mathcal{E}}; \mathbb{R}^2)) \cap \xLinfty((0,1/K);\xCn{{m,\alpha}}(\overline{\mathcal{E}}; \mathbb{R}^2)),\\
		\widetilde{\xsym{\xi}}, \widetilde{\xsym{\eta}} \in \xCzero([0,1/K];\xCn{{m-2,\alpha}}(\overline{\mathcal{E}};\mathbb{R}^2)) \cap \xLinfty((0,1/K);\xCn{{m-1,\alpha}}(\overline{\mathcal{E}};\mathbb{R}^2)).
	\end{gather*}
	Now, the estimates in~\eqref{equation:scondB} follow after recalling the form of $(\widetilde{\xvec{u}}, \widetilde{\xvec{B}})$ given in \eqref{equation:controlsdef_version_b}, the expression of $\xvec{H}^2$ in \eqref{equation:H12_version_b}, the definition of~$\beta$ in \eqref{equation:beta}, the choice of~$C_{*} \geq 1$ with \eqref{equation:inequality1}, the property $\xvec{y}^*(\cdot, 1/K) = \xsym{0}$ from \Cref{lemma:flushing}, and the estimate
	\[
	\max_{t\in[0,1]} \left(\|\xvec{V}-\xvec{y}^*\|_{m,\alpha,\mathcal{E}}(t) + \|\xvec{H}\|_{m,\alpha,\mathcal{E}}(t)\right) < \delta_{l+1}/3C_*
	\]
	from \Cref{theorem:ret}. 
	Further, as~$\xsym{\Xi}$ is given by~\Cref{theorem:ret}, $\widehat{\xsym{\eta}}^1$ and $\widehat{\xsym{\eta}}^2$ obey \eqref{equation:prpetahatj1}, and because~$\xvec{H}^1$ and~$\xvec{H}^2$ are defined via the partition of unity~$\{\mu_1, \mu_2\}$ from \eqref{equation:partitionofunity}, which is subordinate to the covering~$\{\mathcal{O}_1, \mathcal{O}_2\}$ specified in \eqref{equation:coveringproperties}, it holds
	\[
	\operatorname{supp}(\widetilde{\xsym{\xi}}) \subset \omegaup\times[0,1/K], \quad  \operatorname{supp}(\widetilde{\xsym{\eta}}) \subset \omegaup\times[1/K-K_0, 1/K].
	\]
	Resorting to \cref{equation:H12_version_b,equation:etahat,equation:prpetahatj1,equation:controlsdef_version_b}, accounting for the boundary values of $\psi$, and utilizing a stream function for $\widetilde{\xsym{\eta}}$, it can be shown that
	\begin{equation}\label{equation:etaprop_version_b}
		\xdiv{\widetilde{\xsym{\eta}}} = 0 \mbox{ in } \mathcal{E}, \quad \widetilde{\xsym{\eta}}\cdot \xvec{n} = 0 \mbox{ on } \partial \mathcal{E}, \quad
		\widetilde{\xsym{\eta}} = \xnab^{\perp} \widetilde{\phi}, \quad \widetilde{\phi} = 0 \mbox{ on } \partial \mathcal{E}.
	\end{equation}
	
	\paragraph{Step 3. The properties \eqref{equation:suppcondB} and \eqref{equation:suppcondB3}.} In order show that $\widetilde{\xvec{B}}$ has the annihilation property \eqref{equation:suppcondB}, we recall
	\begin{itemize}
		\item the definition $\widetilde{\xvec{B}} = \beta\xvec{H}^{1}+\xvec{H}^{2}$ from \eqref{equation:controlsdef_version_b};
		\item the definition of $\beta$ in \eqref{equation:beta}, which implies $\widetilde{\xvec{B}}(\cdot, 1/K) = \xvec{H}^2(\cdot, 1/K)$;
		\item $\operatorname{supp}(\widehat{\xsym{\eta}}^{2}) \subset (\mathcal{O}_1 \cap \mathcal{O}_2) \times [0, 1/K]$;
		\item $\operatorname{supp}(\xvec{H}^2(\cdot, 0)) \subset \mathcal{O}_2$ due to \eqref{equation:H12_version_b} and \eqref{equation:partitionofunity};
		\item the maximal dragging distance \eqref{equation:draggingproperty} of $\xvec{y}^*$ from \Cref{lemma:flushing}.
	\end{itemize}
	In particular, as $\xvec{H}^2$ is governed by \eqref{equation:MHD00unforcedtrsp_0_version_b}, one can infer from Lemmas~\Rref{lemma:nu} and~\Rref{lemma:mgftrsp} that the inclusion
	\begin{equation}\label{equation:etahatitmc}
		\operatorname{supp}(\xvec{H}^2(\cdot,t)) \subset \mathcal{N}_{\nu_0}\left(\xmcal{Y}^*(\mathcal{O}_2, 0, t)\right) \cup \bigcup_{s\in[0,t]} \mathcal{N}_{\nu_0}\left(\xmcal{Y}^*(\operatorname{supp}(\widehat{\xsym{\eta}}^2(\cdot,s)), s, t)\right)
	\end{equation}
	holds for any $t \in [0,1]$. Due to the above-listed points, and by the choice of $\nu_0$ (\cf~\Cref{remark:smallnu0}), one finds
	\[
	\operatorname{dist}(\operatorname{supp}(\xvec{H}^2(\cdot, 1/K)), \mathcal{N}_{d_{\Lambda}/2}(\Lambda)) > d_{\Lambda}.
	\]
	Hence, \eqref{equation:suppcondB} holds. To show the local flushing property \eqref{equation:suppcondB3}, let~$\mathcal{S}\subset\overline{\mathcal{E}}$ be any connected set with~$\mathcal{S}\cap\partial\mathcal{E} \neq \emptyset$, and assume that
	\[
	\widetilde{\xvec{B}}_0(\xmcal{Y}^*(\mathcal{N}_{2\nu_0}(\mathcal{S}),1/K,0)) = \{\xsym{0}\}.
	\]
	Then, the properties of $\xvec{y}^*$ from \Cref{lemma:flushing}, together with Lemmas~\Rref{lemma:nu} and \Rref{lemma:mgftrsp}, allow to infer that
	\[
		\widetilde{\xvec{B}}(\mathcal{N}_{\nu_0}(\mathcal{S}), 1/K-K_0) = \{\xsym{0}\},
	\]
	where we used that by \eqref{equation:controlsdef_version_b} it holds  $\widetilde{\xsym{\eta}}(\cdot, t) = \xsym{0}$ for $t \in [0, 1/K-K_0]$, followed by applying \Cref{lemma:mgftrsp} directly to the equation \eqref{equation:MHD00SimplyConnectedReducedsupp} satisfied by  $\widetilde{\xvec{B}}$. 
	In view of \eqref{equation:beta} and \eqref{equation:controlsdef_version_b}, this means that
	\[
		(\xnab^{\perp} \psi)(\mathcal{N}_{\nu_0}(\mathcal{S}), 1/K-K_0) = \{\xsym{0}\},
	\]
	which implies by using the boundary condition $\psi(\cdot, 1/K-K_0) = 0$ at $\partial \mathcal{E}$, and by also involving the hypothesis $\mathcal{S}\cap\partial\mathcal{E} \neq \emptyset$, that
	\begin{equation}\label{equation:strzprpv1}
		\psi(\mathcal{N}_{\nu_0}(\mathcal{S}), 1/K-K_0) = \{0\}, \quad \widehat{\xsym{\eta}}^2(\mathcal{N}_{\nu_0}(\mathcal{S})), 1/K-K_0) = \xsym{0}.
	\end{equation}
	Owing to the definition of $\widehat{\xvec{\eta}}^2$ in \eqref{equation:etahat}, the equation for $\xvec{H}^2$ in \eqref{equation:MHD00unforcedtrsp_0_version_b}, the inclusion in \eqref{equation:etahatitmc}, and $\xvec{y}^*(\cdot, t) = \xsym{0}$ for all $t \in [1/K-2K_0,1/K]$, one can conclude $\widehat{\xsym{\eta}}^2(\xvec{x}, t) = \xsym{0}$ for all~$\xvec{x} \in \mathcal{S}$ and~$t \in [1/K-K_0, 1/K]$. After resorting again to Lemmas~\Rref{lemma:flushing},~\Rref{lemma:nu} and~\Rref{lemma:mgftrsp}, but this time to the equation for $\xvec{H}^2$ in \eqref{equation:MHD00unforcedtrsp_0_version_b}, while using also \cref{equation:beta,equation:controlsdef_version_b,equation:strzprpv1}, ones arrives at
	\[
		\widetilde{\xvec{B}}(\mathcal{S}, 1/K) = \beta(1/K)\xvec{H}^{1}(\mathcal{S}, 1/K)+\xvec{H}^{2}(\mathcal{S}, 1/K) = \xvec{H}^2(\mathcal{S}, 1/K) = \{\xsym{0}\}.
	\]

	\subsubsection{Version 2}\label{subsection:proofB}
	Let $(\xvec{V}, \xvec{H}, P, \xsym{\Xi})$ be selected via \Cref{theorem:ret} with $(\xvec{V}_0, \xvec{H}_0) = (\widetilde{\xvec{u}}_0, \widetilde{\xvec{B}}_0)$, which provides
	\begin{equation}\label{equation:rrH}
		\xvec{V}, \xvec{H} \in \xCzero([0,1];\xCn{{m-1,\alpha}}(\overline{\mathcal{E}}; \mathbb{R}^2)) \cap \xLinfty((0,1);\xCn{{m,\alpha}}(\overline{\mathcal{E}}; \mathbb{R}^2)).
	\end{equation}
	First, the~initial magnetic field is separated into two parts: one supported in~$\mathcal{O}_1$ and the other in~$\mathcal{O}_2$. The magnetic field~$\xvec{H}$ is recovered from the evolution of individual contributions with localized data. Second, a regularity corrector is introduced. Third, a controlled solution to~\eqref{equation:MHD00SimplyConnectedReducedsupp} is assembled. Fourth, the support of $\widetilde{\xvec{B}}$ is discussed.
	\paragraph{Step 1. Induction problems with localized data.} Let $\psi_0$ be such that $\widetilde{\xvec{B}}_0 = \xnab^{\perp} \psi_0$ and $\psi_0 = 0$ on $\partial \mathcal{E}$. We split $\xvec{H} = \xvec{H}^1 + \xvec{H}^2$ with $\xvec{H}^1, \xvec{H}^2 \in \xCzero([0,1/K];\xCn{{m-2}}(\overline{\mathcal{E}};\mathbb{R}^2))$
	being defined as the solutions to the linear problems
	\begin{equation}\label{equation:MHD00unforcedtrsp_0}
		\begin{cases}
			\partial_t \xvec{H}^{j} + (\xvec{V} \cdot \xsym{\nabla}) \xvec{H}^{j} - (\xvec{H}^{j} \cdot \xsym{\nabla}) \xvec{V} = \xsym{0} & \mbox{ in } \mathcal{E} \times (0,1/K),\\
			\xdiv{\xvec{H}^j} = 0 & \mbox{ in } \mathcal{E} \times (0,1/K),\\
			\xvec{H}^j \cdot \xvec{n} = 0 & \mbox{ on } \partial\mathcal{E} \times (0,1/K),\\
			\xvec{H}^{j}(\cdot, 0)  =  \xnab^{\perp}(\mu_j \psi_0)  & \mbox{ in } 	\mathcal{E}. 
		\end{cases}
	\end{equation}
	Because the evolution equations in \eqref{equation:MHD00unforcedtrsp_0} preserve the initial  divergence (\eg, see \cite[Section 3.1]{Glass2000}) and the initial normal trace (one has $\partial_t (\xvec{H}^j \cdot \xvec{n}) = 0$ at $\partial \mathcal{E}$), one can conclude from
	\[
		\xdiv{\left(\xnab^{\perp}(\mu_j \psi_0)\right)} = 0 \mbox{ in } \mathcal{E}, \quad \xnab^{\perp}(\mu_j \psi_0) \cdot \xvec{n} = 0 \mbox{ on } \partial \mathcal{E}, \quad j \in \{1,2\}.
	\]
	that the problems in \eqref{equation:MHD00unforcedtrsp_0} are equivalent to
	\begin{equation}\label{equation:MHD00unforcedtrsp}
		\begin{cases}
			\partial_t \xvec{H}^{j} + (\xvec{V} \cdot \xsym{\nabla}) \xvec{H}^{j} = (\xvec{H}^{j} \cdot \xsym{\nabla}) \xvec{V}  & \mbox{ in } \mathcal{E} \times (0,1/K),\\
			\xvec{H}^{j}(\cdot, 0)  =  \xnab^{\perp}(\mu_j \psi_0)  & \mbox{ in } 	\mathcal{E}.
		\end{cases}
	\end{equation}
	The $\xCzero([0,1/K];\xCn{{m-2}}(\overline{\mathcal{E}};\mathbb{R}^2))$ regularity of the unique solutions $\xvec{H}^1$ and $\xvec{H}^2$ to~\eqref{equation:MHD00unforcedtrsp} follows from the representation \eqref{equation:moc} and the explanations given there. Moreover, as~$\xvec{H}$ and~$\xvec{H}^1+\xvec{H}^2$ both solve the same well-posed problem, it holds $\xvec{H} = \xvec{H}^1 + \xvec{H}^2$.
	\paragraph{Step 2. Regularity corrector.} To ensure that $\widetilde{\xvec{B}}$ vanishes in a relative neighborhood of $\Lambda$ at $t = 1/K$, one could define it as $\beta \xvec{H}^1 +\xvec{H}^2$, where $\beta$ is the smooth function from \eqref{equation:beta} with $\beta(t) = 1$ for $t \in [0, 1/K-K_0]$ and $\beta(t) = 0$ for $t \geq 1/K - K_0/2$. However, in this way, we could not ensure that $\widetilde{B}$ is contained in the desired space
	\begin{equation}\label{equation:regsp}
		\xCn{0}([0,1/K];\xCn{{m-1,\alpha}}(\overline{\mathcal{E}}; \mathbb{R}^2)) \cap \xLinfty((0,1/K);\xCn{{m,\alpha}}(\overline{\mathcal{E}}; \mathbb{R}^2)).
	\end{equation}
	Therefore, we aim to define $\widetilde{\xvec{B}}$ of the form $\beta \xvec{H}^1 +\xvec{H}^2 - \xvec{A}$ so that it belongs to the space in \eqref{equation:regsp}, and where the correction term~$\xvec{A}$ will have the representation
	\begin{equation}\label{equation:strfprta}
		\xvec{A} = \xnab^{\perp} \theta \, \mbox{ in } \mathcal{E}\times(0,1/K), \quad \theta = 0 \, \mbox{ on } \partial \mathcal{E}\times(0,1/K).
	\end{equation}
	During $[0, 1/K-K_0]$, the function $\beta\xvec{H}^1 + \xvec{H}^2$ equals $\xvec{H}$ and is sufficiently regular; thus, we will take $\xvec{A} = 0$ on $[0, 1/K - 2K_0]$, and for $t \geq 1/K-2K_0$ localize $\operatorname{supp}(\xvec{A}(\cdot, t))$ in a relative neighborhood of 
	\begin{equation}\label{equation:suppa}
		\mathscr{M} \coloneq \bigcup_{s \in [1/K-2K_0,1/K]}\operatorname{supp}(\xvec{H}^1(\cdot, s)) \cap \operatorname{supp}(\xvec{H}^2(\cdot,s)).
	\end{equation}
	To this end, let us fix $\chi \in \xCinfty(\overline{\mathcal{E}}; \mathbb{R})$ with $\chi = 1$ in $\mathcal{N}_{\nu_0/2}(\mathscr{M})$, $\chi = 0$ in $\overline{\mathcal{E}}\setminus \mathcal{N}_{\nu_0}(\mathscr{M})$, and such that (\cf~\eqref{equation:inequality0})
	\begin{equation}\label{equation:inequality}
		\|\xnab^{\perp}(\chi f)\|_{m, \alpha, \mathcal{E}} \leq C_* \|\xnab^{\perp}f\|_{m, \alpha, \mathcal{E}}
	\end{equation} 
	holds for all $f \in \xCn{{m+1,\alpha}}(\overline{\mathcal{E}}; \mathbb{R})$. As $\xvec{A}$ should obey \eqref{equation:strfprta}, we utilize the stream function representations $\xvec{H}^j = \xnab^{\perp} \widetilde{\theta}^j$ with $\widetilde{\theta}^1 = \widetilde{\theta}^2 = 0$ on $\partial \mathcal{E}\times(0,1/K)$
	and define
	\begin{equation}\label{equation:corrector}
		\xvec{A}(\xvec{x}, t) \coloneq  \xnab^{\perp}\left(\sigma(t)\left(\beta(t)\widetilde{\theta}^1(\xvec{x}, t)+\widetilde{\theta^2}(\xvec{x}, t)\right)\chi(\xvec{x})\right),
	\end{equation}
	where $\sigma\in\xCinfty(\mathbb{R};[0,1])$ is selected such that
	\[
	\sigma(t) = 0 \mbox{ when } t \leq 1/K-2K_0, \quad \sigma(t) = 1 \mbox{ when } t \geq 1/K-K_0.
	\]
	Since either $\xvec{H}(\xvec{x}, t) = \xvec{H}^1(\xvec{x}, t)$ or $\xvec{H}(\xvec{x}, t) = \xvec{H}^2(\xvec{x}, t)$ when $\xvec{x} \notin \mathscr{M}$ for $t \geq 1/K-K_0$, the property
	\begin{equation}\label{equation:rcprp}
		\beta \xvec{H}^1 + \xvec{H}^2 - \xvec{A} \in \xCn{0}([0,1/K];\xCn{{m-1,\alpha}}(\overline{\mathcal{E}}; \mathbb{R}^2)) \cap \xLinfty((0,1/K);\xCn{{m,\alpha}}(\overline{\mathcal{E}}; \mathbb{R}^2))
	\end{equation}
	follows now from (\cf~\cref{equation:beta,equation:corrector,equation:rrH})
	\begin{equation}\label{equation:rpa0}
		\begin{gathered}
			\xvec{H} \in \xCn{0}([0,1/K];\xCn{{m-1,\alpha}}(\overline{\mathcal{E}}; \mathbb{R}^2)) \cap \xLinfty((0,1/K);\xCn{{m,\alpha}}(\overline{\mathcal{E}}; \mathbb{R}^2)),\\
			\xvec{A} \in \xCzero([1/K-K_0, 1/K];\xCn{{m-1,\alpha}}(\overline{\mathcal{E}}\setminus\mathcal{N}_{\nu_0/2}(\mathscr{M}); \mathbb{R}^2)) \\ \cap \xLinfty((1/K-K_0, 1/K);\xCn{{m,\alpha}}(\overline{\mathcal{E}}\setminus\mathcal{N}_{\nu_0/2}(\mathscr{M}); \mathbb{R}^2)),\\
			\xvec{A} = \beta\xvec{H}^{1}+\xvec{H}^{2} \mbox{ in } \mathcal{N}_{\nu_0/2}(\mathscr{M}) \times [1/K-K_0, 1/K].
		\end{gathered}
	\end{equation}

	\paragraph{Step 3. Controlled solutions.} In view of \cref{equation:rrH,equation:rcprp,equation:rpa0} a controlled solution $(\widetilde{\xvec{u}}, \widetilde{\xvec{B}}, \widetilde{p},\widetilde{\xsym{\xi}}, \widetilde{\xsym{\eta}})$ to \eqref{equation:MHD00SimplyConnectedReducedsupp} is now given by
	\begin{equation}\label{equation:vmpd}
		\begin{gathered}
			\widetilde{\xvec{u}} \coloneq \xvec{V}, \quad \widetilde{\xvec{B}} \coloneq \beta\xvec{H}^{1}+\xvec{H}^{2} - \xvec{A}, \quad \widetilde{p} \coloneq P,\\
			\widetilde{\xsym{\eta}} \coloneq \xdrv{\beta}{t} \xvec{H}^1 - \partial_t \xvec{A} - (\xvec{V}\cdot\xnab)\xvec{A} + (\xvec{A}\cdot\xnab)\xvec{V},\\
			\widetilde{\xsym{\xi}} \coloneq \xsym{\Xi} + (1-\beta)(\xvec{H}^{1} \cdot \xnab) 	\xvec{H}^{2} + (1-\beta)(\xvec{H}^{2} \cdot \xnab) \xvec{H}^{1}  + (1-\beta^2)(\xvec{H}^{1} \cdot \xnab) \xvec{H}^{1} 
			\\
			+ (\xvec{A}\cdot \xnab) (\beta \xvec{H}^1 + \xvec{H}^2 - \xvec{A}) + (\beta \xvec{H}^1 + \xvec{H}^2) \cdot \xnab) \xvec{A}.
		\end{gathered}
	\end{equation}
	The representations $\xvec{A} = \beta\xvec{H}^{1}+\xvec{H}^{2}$ in $\mathcal{N}_{\nu_0/2}(\mathscr{M}) \times [1/K-K_0, 1/K]$ and $\xvec{H} = \xvec{H}^{1}+\xvec{H}^{2}$ in $\mathcal{E}\times[0,1/K]$, and also the equations \eqref{equation:MHD00unforcedtrsp} for $\xvec{H}^1$ and~$\xvec{H}^2$ and the equation for~$\xvec{H}$ from \eqref{equation:MHD00unforcedlargedomain}, provide together with \eqref{equation:vmpd} the regularity of the controls. In particular, 
	\begin{equation*}\label{equation:prregctrl}
		\begin{gathered}
			\widetilde{\xvec{u}}, \widetilde{\xvec{B}} \in \xCn{0}([0,1/K];\xCn{{m-1,\alpha}}(\overline{\mathcal{E}}; \mathbb{R}^2)) \cap \xLinfty((0,1/K);\xCn{{m,\alpha}}(\overline{\mathcal{E}}; \mathbb{R}^2)),\\
			\widetilde{\xsym{\xi}}, \widetilde{\xsym{\eta}} \in \xCzero([0,1/K];\xCn{{m-2,\alpha}}(\overline{\mathcal{E}};\mathbb{R}^2))\cap \xLinfty((0,1/K);\xCn{{m-1,\alpha}}(\overline{\mathcal{E}};\mathbb{R}^2)),\\
			\operatorname{supp}(\widetilde{\xsym{\xi}}) \subset \operatorname{supp}(\xvec{H}^1) \cup \operatorname{supp}(\xvec{A}) \cup \operatorname{supp}(\xsym{\Xi})\subset\omegaup\times[0, 1/K],\\
			\operatorname{supp}(\widetilde{\xsym{\eta}}) \subset \operatorname{supp}(\xvec{H}^1) \cup \operatorname{supp}(\xvec{A}) \subset \omegaup\times[1/K-2K_0, 1/K],
		\end{gathered}
	\end{equation*}
	where the inclusions for $\operatorname{supp}(\widetilde{\xsym{\xi}})$ and $\operatorname{supp}(\widetilde{\xsym{\eta}})$ are verified by inserting into \eqref{equation:vmpd} the following information: 1) $\operatorname{supp}(\xsym{\Xi})\subset\omegaup\times[0, 1/K]$ by \Cref{theorem:ret}; 2) according to \eqref{equation:beta}, it holds~$\beta = 1$ on $[0, 1/K-2K_0]$; 3) $\sigma(t) = 0$ for~$t\leq 1/K-2K_0$; 4) in view of $\nu_0$'s smallness (\cf~\Cref{remark:smallnu0}) and \cref{equation:suppa,equation:corrector}, the problem for $\xvec{H}^1$ in \eqref{equation:MHD00unforcedtrsp} provides with Lemmas~\Rref{lemma:nu} and \Rref{lemma:mgftrsp} the inclusions
	\[
		\operatorname{supp}(\xvec{H}^1)\subset\omegaup\times[0, 1/K], \quad \operatorname{supp}(\xvec{A})\subset\omegaup\times[1/K-2K_0, 1/K].
	\]
	Further, since $\overline{\mathcal{E}}\setminus\mathcal{N}_{\nu_0/3}(\mathscr{M})$ consists of disjoint regions where either~$\xvec{H}^1(\cdot,1/K) = \xsym{0}$ or $\xvec{H}^2(\cdot,1/K) = \xsym{0}$ (\cf~\eqref{equation:suppa}), and because $\xvec{H} = \xvec{H}^1 + \xvec{H}^2$, it can be inferred via \cref{equation:rrH,equation:rpa0,equation:vmpd} that
	\[
		\|\widetilde{\xvec{B}}\|_{\alpha,m,\mathcal{E}}(1/K) \leq 2C_*\|\xvec{H}\|_{\alpha,m,\mathcal{E}}(1/K).
	\]
	Then, \eqref{equation:scondB} can be concluded by using $\xvec{y}^*(\cdot, 1/K) = \xsym{0}$ and (\cf~\Cref{theorem:ret})
	\[
		\max_{t\in[0,1]} \left(\|\xvec{V}-\xvec{y}^*\|_{m,\alpha,\mathcal{E}}(t) + \|\xvec{H}\|_{m,\alpha,\mathcal{E}}(t)\right) < \delta_{l+1}/3C_*.
	\]
	Finally, the representation $\widetilde{\xsym{\eta}} = \xnab^{\perp} \widetilde{\phi}$ with $\widetilde{\phi} = 0$ at $\mathcal{E}\times[0,1/K]$ is inherited from~$\xvec{A}$ and~$\xvec{H}^1$, thereby ensuring that the $\xLtwo(\mathcal{E};\mathbb{R}^2)$-projection of $\widetilde{\xvec{B}}(\cdot, t)$ to $\xZ(\mathcal{E})$ vanishes for each $t \in [0,1/K]$.

	\paragraph{Step 4. The properties \eqref{equation:suppcondB} and \eqref{equation:suppcondB3}.} In view of \Cref{lemma:mgftrsp},~$\operatorname{supp}(\xvec{H}^1)$ and~$\operatorname{supp}(\xvec{H}^2)$ are transported along $\xvec{V}$, which is a small perturbation of~$\xvec{y}^*$ by \Cref{lemma:nu}. It further holds $\widetilde{\xvec{B}} = \xvec{H}^2(\cdot, 1/K)$ (\cf~\cref{equation:beta,equation:vmpd}). At any given time~$t\in[0,1/K]$, due to the properties of $\xvec{y}^*$ from~\Cref{lemma:flushing}, but also using \cref{equation:sn0,equation:coveringproperties}, one finds that $\operatorname{supp}(\xvec{H}^2(\cdot, t))$ keeps a distance of at least~$3d_{\Lambda}/2 - \nu_0$ to~$\Lambda$, noting that~$3d_{\Lambda}/2 - \nu_0 > d_{\Lambda}$ (\cf~\Cref{remark:smallnu0}).
	Therefore, after further taking into account that $\xvec{A}$ is supported in $\mathcal{N}_{\nu_0}(\mathscr{M})$ (\cf~\eqref{equation:corrector}), one arrives at \eqref{equation:suppcondB}.  
	To verify the property \eqref{equation:suppcondB3}, we fix any connected set~$\mathcal{S}\subset\overline{\mathcal{E}}$ with~$\mathcal{S}\cap\partial\mathcal{E} \neq \emptyset$ and assume that
	\[
	\widetilde{\xvec{B}}_0(\xmcal{Y}^*(\mathcal{N}_{2\nu_0}(\mathcal{S}),1/K,0)) = \{\xsym{0}\}.
	\]
	Since $\widetilde{\xvec{B}}_0 = \xnab^{\perp} \psi_0$ with $ \psi_0 = 0$ on $\partial \mathcal{E}$, and owing to $\mathcal{S}\cap\partial\mathcal{E} \neq \emptyset$, one has $\psi_0(\xvec{x}) = 0$ for all $\xvec{x} \in \xmcal{Y}^*(\mathcal{N}_{2\nu_0}(\mathcal{S}),1/K,0)$. This implies
	\begin{equation}\label{equation:idhida}
		\left[\xnab^{\perp}(\mu_j \psi_0)\right](\xmcal{Y}^*(\mathcal{N}_{2\nu_0}(\mathcal{S}),1/K,0)) = \{\xsym{0}\}, \quad j \in \{1,2\}.
	\end{equation}
	Let us recall that we have already fixed (above \eqref{equation:beta}) the small number $K_0 \in (0,1/2K)$ in a way that~$\xvec{y}^*(\cdot, t) = \xsym{0}$ for all $t \in [1/K-2K_0, 1/K]$. Thus, referring to \Cref{lemma:nu} and applying \Cref{lemma:mgftrsp} to the trajectories $\xvec{H}^1$ and $\xvec{H}^2$ for \Cref{equation:MHD00unforcedtrsp}, where the respective initial states $\xnab^{\perp}(\mu_1 \psi_0)$ and $\xnab^{\perp}(\mu_2 \psi_0)$ satisfy \eqref{equation:idhida},
	it follows that
	\begin{equation}\label{equation:idhida2}
		\xvec{H}^1(\mathcal{N}_{\nu_0}(\mathcal{S}), t) = \xvec{H}^2(\mathcal{N}_{\nu_0}(\mathcal{S}), t) = \{\xsym{0}\}
	\end{equation}
	for all $t \in [1/K-2K_0, 1/K]$. Hence, by the definition of the cutoff $\chi$ employed in \eqref{equation:corrector}, one obtains $\xvec{A}(\mathcal{S}, t) = \{\xsym{0}\}$ for all $t\in[0,1/K]$. This allows via \eqref{equation:vmpd} to conclude $\widetilde{\xvec{B}}(\mathcal{S}, 1/K) = \{\xsym{0}\}$.
	
	\bibliographystyle{abbrv}
	\bibliography{general_domain}
\end{document}